\documentclass[a4paper,12pt,twoside]{amsart}
\usepackage{subfig}
\usepackage{vmargin}
\usepackage{amssymb}
\usepackage{amsfonts}
\usepackage{amsmath}
\usepackage{mathrsfs}
\usepackage[dvips]{graphicx}
\usepackage{color}
\DeclareGraphicsExtensions{.pdf, .png, .jpg, .mps}
\usepackage[colorlinks=true]{hyperref}
\definecolor{myc}{cmyk}{0.0009,0.8,0.8,0.00}
\usepackage{xspace}

\newtheorem{theorem}{Theorem}[section]
\newtheorem{lem}[theorem]{Lemma}
\newtheorem{pro}[theorem]{Proposition}

\theoremstyle{definition}

\newtheorem{rem}[theorem]{Remark}

\newtheorem*{nb}{\footnotesize {N.B}}

\numberwithin{equation}{section}

\def\dive{\operatorname{div}}
\def\p{\partial}
\def\no{\noindent}
\def\io{{\infty}}

\def\re{\operatorname{Re}}
\def\im{\operatorname{Im}}
\def\Id{\operatorname{Id}}

\def\Con{{\mathscr C}}
\def\moo{\Con^{\io}}
\def\mooc{\Con^{\io}_{\textit c}}

\def\N{\mathbb N}
\def\Z{\mathbb Z}

\def\R{\mathbb R}

\def\poscal#1#2{\langle#1,#2\rangle}

\def\poi#1#2{\left\{#1,#2\right\}}
\def\norm#1{\Vert#1\Vert}
\def\val#1{\vert#1\vert}

\def\valjp#1{\langle#1\rangle}

\def\l2{L^2(\R^{n})}
\def\L2{L^2(\R^{2n})}
\def\supp{\operatorname{supp}}

\def\vs{\vskip.3cm}
\def\ind#1{\mathbf 1_{#1}}
\let \dis=\displaystyle
\let\no=\noindent
\let \dis=\displaystyle

\let\no=\noindent
\def\mat22#1#2#3#4{\begin{pmatrix}#1&#2\\ #3&#4\end{pmatrix}}

\def\wt#1{\widetilde{#1}}

\def\XXint#1#2#3{{\setbox0=\hbox{$#1{#2#3}{\int}$}
     \vcenter{\hbox{$#2#3$}}\kern-.5\wd0}}

\def\beq{\begin{equation}}
\def\eeq{\end{equation}}

\def\hf{{\frac{1}{2}}}

\newcommand{\rhs}{rhs\xspace}
\newcommand{\lhs}{l.h.s.\@\xspace}
\newcommand{\ie}{i.e.\@\xspace}
\newcommand{\eg}{e.g.\@\xspace}
\newcommand{\resp}{resp.\@\xspace}
\newcommand{\viz}{viz.\@\xspace}

\newcommand{\nhd}{neighborhood\xspace}

\newcommand{\suff}{sufficiently\xspace}

\def\op#1{{\operatorname{op}(#1)}}
\def\opw#1{{\operatorname{op}}^w(#1)}

\newcommand{\A}{\ensuremath{\mathcal A}}

\renewcommand{\H}{\ensuremath{\mathcal H}}
\renewcommand{\L}{\ensuremath{\mathcal L}}
\renewcommand{\P}{\ensuremath{\mathcal P}}
\renewcommand{\S}{\ensuremath{\mathcal S}}

\newcommand{\V}{\ensuremath{\mathcal V}}
\newcommand{\W}{\ensuremath{\mathcal W}}

\newcommand{\ST}{\ensuremath{\mathcal S}_\tau}

\newcommand{\br}{_{|x_n=0^+}}
\newcommand{\bl}{_{|x_n=0^-}}

\DeclareMathSymbol{\intop}{\mathop}{symbols}{115}
\def\int{\intop}
\def\iint{\intop\!\!\intop}

\DeclareMathOperator*{\ssum}{\textstyle{\sum}}

\def\qedsymbol{$\blacksquare$}

\newcommand{\bld}[1]{\mbox{\boldmath $#1$}}
\begin{document}
\title[Carleman estimates for operators with  jumps]
{Carleman estimates for elliptic operators with jumps at an interface:
Anisotropic case\\ and sharp geometric conditions
}
\author[J. Le Rousseau]{{J\'er\^ome Le Rousseau}}
\begin{address}{J\'er\^ome Le Rousseau, MAPMO, UMR CNRS 6628, 
    Route de Chartres, Universit\'e d'Orl\'eans B.P. 6759 -- 
    45067 Orl\'eans cedex 2 
    France}
  \email{\href{mailto:jlr@univ-orleans.fr}{jlr@univ-orleans.fr}}
  \urladdr{\url{http://www.univ-orleans.fr/mapmo/membres/lerousseau/}}
\end{address}
\author[N. Lerner]{Nicolas  Lerner}
\begin{address}{Nicolas Lerner,
    Projet analyse fonctionnelle, Institut de Math\'ematiques de Jussieu,
    UMR CNRS 7586,
    Universit\'e Pierre-et-Marie-Curie (Paris 6),
    Bo\^{\i}te 186 - 4, Place Jussieu - 75252 Paris cedex 05, France}
  \email{\href{mailto:lerner@math.jussieu.fr}{lerner@math.jussieu.fr}}
  \urladdr{\url{http://www.math.jussieu.fr/~lerner/}}
\end{address}


\date{\today}

% Acknowledgments
\thanks{The authors wish to thank E. Fern\'{a}ndez-Cara for bringing
  to their attention the importance of Carleman estimates for
  anisotropic elliptic operators towards applications to biological
  tissues.}

\thanks{The first author was partially supported by l'Agence
  Nationale de la Recherche under grant ANR-07-JCJC-0139-01.}

\begin{abstract}
  We consider a second-order selfadjoint elliptic operator with an
  anisotropic diffusion matrix  having a jump across a smooth
  hypersurface. We prove the existence of a weight-function such that
  a Carleman estimate holds true.  We moreover prove that the
  conditions imposed on the weight function are necessary.
\end{abstract}

\maketitle
{%\footnotesize\baselineskip=0.72\normalbaselineskip
\tableofcontents
}

\numberwithin{equation}{section}
\baselineskip=1.1\normalbaselineskip

%%%%%%%%%%%%%%%%%%%%%%%%%%%%%%%%%%%%%%%%%%%%%
%% Sections
% intro
\section{Introduction}

\subsection{Carleman estimates}
Let $P(x,D_{x})$ be a differential operator defined on some open
subset of $\R^n$. A {\it Carleman estimate} for this operator is the
following weighted a priori inequality
\begin{equation}\label{1.carleman}
  \norm{e^{\tau \varphi} P w}_{L^2(\R^n)}
  \gtrsim \norm{e^{\tau \varphi} w}_{L^2(\R^n)},
\end{equation}
where the weight function $\varphi$ is real-valued with a
non-vanishing gradient, $\tau$ is a large positive parameter and $w$
is any smooth compactly supported function.  This type of estimate was
used for the first time in 1939 in T.~Carleman's article
\cite{MR0000334} to handle uniqueness properties for the Cauchy
problem for non-hyperbolic operators.  To this day, it remains
essentially the only method to prove unique continuation properties
for ill-posed problems\footnote{The 1960 article by F.~John
  \cite{MR0130456} showed that, although Hadamard well-posedness
  property is a privilege of hyperbolic operators, some weaker type of
  continuous dependence, called in \cite{MR0130456} {\it H\"older
    continuous well-behaviour}, could occur. Strong connections between
  the well-behavior property and Carleman estimates can be found in an
  article by H.~Bahouri \cite{MR896184}.}, in particular to handle
uniqueness of the Cauchy problem for elliptic operators with
non-analytic coefficients\footnote{For analytic operators, Holmgren's
  theorem provides uniqueness for the non-characteristic Cauchy
  problem, but that analytical result falls short of giving a control
  of the solution from the data.}.  This tool has been refined,
polished and generalized by manifold authors and plays now a very
important r\^ole in control theory  and inverse
  problems.  The 1958 article by A.P.~Calder\'on \cite{Calderon:58}
gave a very important development of the Carleman method with a proof
of an estimate of the form of \eqref{1.carleman} using a
pseudo-differential factorization of the operator, giving a new start
to singular-integral methods in local analysis.  In the article
\cite{MR0104924} and in his first PDE book (Chapter VIII,
\cite{MR0161012}), L.~H\"ormander showed that local methods could
provide the same estimates, with weaker assumptions on the regularity
of the coefficients of the operator.  \par For instance, for
second-order elliptic operators with real coefficients\footnote{The
  paper \cite{MR597745} by S. Alinhac shows nonunique continuation
  property for second-order elliptic operators with non-conjugate
  roots; of course, if the coefficients of the principal part are
  real, this is excluded.} in the principal part, Lipschitz continuity
of the coefficients suffices for a Carleman estimate to hold and thus
for unique continuation across a $\Con^1$ hypersurface.  Naturally,
pseudo-differential methods require more derivatives, at least
tangentially, i.e., essentially on each level surface of the weight
function $\varphi$.  Chapters 17 and 28 in the 1983-85 four-volume
book \cite{MR1481433} by L.~H\"ormander contain more references and
results.  \par Furthermore, it was shown by A.~Pli{\'s}
\cite{MR0153959} that H\"older continuity is not enough to get unique
continuation: this author constructed a real homogeneous linear
differential equation of second order and of elliptic type on $\R^3$
without the unique continuation property although the coefficients are
H\"older-continuous with any exponent less than one.  The
constructions by K.~Miller in \cite{MR0342822}, and later by N.~Filonov
in \cite{Filonov:01}, showed that H\"older continuity is not sufficient to
obtain unique continuation for second-order elliptic operators, even
in divergence form (see also \cite{MR1822406} and
\cite{MR1630571} for the particular 2D case where boundedness is
essentially enough to get unique continuation for elliptic equations
 in the case of $W^{1,2}$ solutions).

\subsection{Jump discontinuities}
Although the situation seems to be almost completely clarified by the
previous results, with a minimal and somewhat necessary condition on
Lipschitz continuity, we are interested in the following second-order
elliptic operator $\L$,
\begin{equation}\label{1.operator}
  \L w= - \dive (A(x)\nabla w),\ A(x)=(a_{jk}(x))_{1\le j,k\le n}={A^T(x)},\ 
  \inf_{\norm{\xi}_{\R^n}=1}\poscal{A(x)\xi}{\xi}>0,
\end{equation}
in which the matrix $A$ has a jump discontinuity across a smooth
hypersurface.  However we shall impose some stringent --yet natural--
restrictions on the domain of functions $w$, which will be required to
satisfy some homogeneous {\em transmission conditions}, detailed in
the next sections.  Roughly speaking, it means that $w$ must belong to
the domain of the operator, with continuity at the interface, so that
$\nabla w$ remains bounded and continuity of the flux across the
interface, so that $\dive(A\nabla w)$ remains bounded, avoiding in
particular the occurrence of a simple or multiple layer at the
interface\footnote{In the sections below we shall also consider
    non-homogeneous boundary conditions.}.

The article \cite{MR1932966} by A.~Doubova, A.~Osses, and J.-P.~Puel
tackled that problem, in the isotropic case (the matrix $A$ is scalar
$c \, \mathrm{Id}$) with a monotonicity assumption: the observation
takes place in the region where the diffusion coefficient $c$ is the
`lowest'. (Note that the work of \cite{MR1932966} concerns the case of
a parabolic operator but an adaptation to an elliptic operator is
straightforward.) In the one-dimensional case, the monotonicity
assumption was relaxed for general piecewise $\Con^1$ coefficients by
A.~Benabdallah, Y.~Dermenjian and J.~Le~Rousseau \cite{BDLR:07}, and
for coefficients with bounded variations \cite{LeRousseau:07}. The
case of an arbitrary dimension without any monotonicity condition in
the elliptic case was solved by J.~Le~Rousseau and L.~Robbiano in
\cite{LRR:10}: there the isotropic case is treated as well as a
particular case of anisotropic medium. An extension of their approach
to the case of parabolic operators can be found in \cite{LRR:09}. A.~Benabdallah, Y.~Dermenjian and J.~Le~Rousseau also 
tackled the situation in which the interface meets the boundary, a case
that is typical of stratified media \cite{BDLR:10}. They treat
particular forms of anisotropic coefficients.

 \par The purpose of the present article is to show that a Carleman
 estimate can be proven for any operator of type (1.2) without an
 isotropy assumption: $A(x)$ is a symmetric positive-definite matrix
 with a jump discontinuity across a smooth hypersurface. We also
 provide conditions on the Carleman weight function that are rather
 simple to handle and we prove that they are sharp.  \par The approach
 we follow differs from that of \cite{LRR:10} where the authors base
 their analysis on the usual Carleman method for certain microlocal
 regions and on Calder\'on projectors for others. The regions they
 introduce are determined by the ellipticity or non-ellipticity of the
 conjugated operator. The method in \cite{BDLR:10} exploits a
 particular structure of the anisotropy that allows one to use Fourier
 series. The analysis is then close to that of \cite{LRR:10,LRR:09} in
 the sense that second-order operators are inverted in some frequency
 ranges.  Here, our approach is somewhat closer to A.~Calder\'on's
 original work on unique continuation~\cite{Calderon:58}: the
 conjugated operator is factored out in first-order
 (pseudo-differential) operators for which estimates are derived.
 Naturally, the quality of these estimates depends on their elliptic
 or non-elliptic nature; we thus recover microlocal regions that
 correspond to that of \cite{LRR:10}.  Note that such a factorization
 is also used in \cite{IP:03} to address non-homogeneous boundary
 conditions.

\subsection{Notation and statement of the main result}
Let $\Omega$ be an open subset of $\R^n$ and $\Sigma$ be a $\moo$
oriented hypersurface of $\Omega$: we have the partition
\begin{align}\label{2.ome}
  \Omega&=\Omega_{+}\cup \Sigma\cup \Omega_{-},\quad 
  \overline{\Omega_{\pm}}=\Omega_{\pm}\cup \Sigma,\quad\Omega_{\pm} 
  \text{ open subsets of $\R^n$,}
\end{align}
and we introduce the following Heaviside-type functions
\begin{equation}\label{2.hea}
  H_{\pm}={\mathbf 1}_{\Omega_{\pm}}.
\end{equation}
We consider  the elliptic second-order operator
\begin{equation}\label{2.ope}
  \L=D\cdot A D=-\dive(A(x)\nabla),\qquad (D=-i\nabla),
\end{equation}
where $A(x)$ is a symmetric positive-definite $n\times n$ matrix, such that
\begin{equation}\label{2.mat}
  A=H_{-} A_{-}+H_{+}A_{+},\quad
  A_{\pm}\in \moo({\Omega}).
\end{equation}
We shall consider functions $w$ of the following type:
\begin{equation}\label{2.tes}
  w=H_{-}w_{-}+H_{+}w_{+},\quad
  w_{\pm}\in \moo({\Omega}).
\end{equation}
We have $dw=H_{-}dw_{-}+H_{+}dw_{+}+(w_{+}-w_{-})\delta_{\Sigma} \nu
$, where $\delta_{\Sigma}$ is the Euclidean hypersurface measure on
$\Sigma$ and $\nu$ is the unit conormal vector field to $\Sigma$
pointing into $\Omega_{+}$.  To remove the singular term, we assume
\begin{align}\label{2.tr1}
  w_{+}=w_{-}\quad \text{at $\Sigma$},
\end{align}
so that $A dw=H_{-}A_{-}dw_{-}+H_{+}A_{+}dw_{+}$ and
\begin{equation*}
  \dive{(A dw)}=H_{-}\dive{(A_{-}dw_{-})}+H_{+}\dive{(A_{+}dw_{+})}
  +\poscal{A_{+}dw_{+}-A_{-}dw_{-}}{\nu}\delta_{\Sigma}.
\end{equation*}
Moreover, we shall assume that
\begin{align}\label{2.tr2}
\poscal{A_{+}dw_{+}-A_{-}dw_{-}}{\nu}=0\quad \text{at $\Sigma$,\  i.e. 
$\poscal{dw_{+}}{A_{+}\nu}=\poscal{dw_{-}}{A_{-}\nu}$},
\end{align}
so that
\begin{equation}\label{2.trt}
\dive(A dw)=H_{-}\dive{(A_{-}dw_{-})}+H_{+}\dive{(A_{+}dw_{+})}.
\end{equation}
Conditions \eqref{2.tr1}-\eqref{2.tr2}
will be called {\it transmission conditions} on the function $w$ and 
we define the vector space
\begin{equation}\label{1.transspace}
\W=\{H_{-}w_{-}+H_{+}w_{+}\}_{
w_{\pm }\in \moo(\Omega)\text{ satisfying {\eqref{2.tr1}-\eqref{2.tr2}}}  
}.
\end{equation}
Note that \eqref{2.tr1} is a continuity condition
of $w$ across $\Sigma$ and \eqref{2.tr2}
is concerned with the continuity of 
$\poscal{Adw}{\nu}$
across $\Sigma$,
i.e. the continuity of the flux of the vector field $Adw$ across $\Sigma$. 
%\subsubsection{Weight function}
A weight function ``suitable for observation from $\Omega_{+}$'' is defined as
 a Lipschitz continuous function $\varphi$
 on $\Omega$
such that
\begin{equation}\label{2.wei}
\varphi=H_{-}\varphi_{-}+H_{+}\varphi_{+},\quad \varphi_{\pm}\in \moo(\Omega),\quad \varphi_{+}=\varphi_{-},\quad
\poscal{d\varphi_{\pm}}{X}>0 \quad\text{at $\Sigma$},
\end{equation}
for any positively transverse vector field $X$ to $\Sigma$
(i.e. $\poscal{\nu}{X}>0$).

%%%%%%%%%%%%%%%%%%%%%%%%
% theorem              %
%%%%%%%%%%%%%%%%%%%%%%%%
\begin{theorem}
  \label{1.thm.main}
  Let $\Omega, \Sigma, \L,\W$ be as in \eqref{2.ome},
  \eqref{2.ope} and \eqref{1.transspace}.  Then for any compact subset
  $K$ of $\Omega$, there exist a weight function $\varphi$ satisfying
  \eqref{2.wei} and positive constants $C$, $\tau_1$ such that for all
  $\tau\ge \tau_1$ and all $w\in \W$ with $\supp w\subset K$,
  \begin{align}
    \label{eq: Carleman main}
    &C\norm{e^{\tau \varphi}\L w}_{L^2(\R^n)}
    \ge 
    \\ 
    &\quad\tau^{3/2}\norm{e^{\tau \varphi} w}_{L^2(\R^n)}
    +\tau^{1/2}\norm{H_{+}e^{\tau \varphi} \nabla w_{+}}_{L^2(\R^n)}
    +\tau^{1/2}\norm{H_{-}e^{\tau \varphi} \nabla w_{-}}_{L^2(\R^n)}
    \notag\\
    &\quad+\tau^{3/2}\val{(e^{\tau \varphi} w)_{\vert \Sigma}}_{L^2(\Sigma)} +\tau^{1/2}\val{(e^{\tau\varphi}  \nabla w_{+})_{\vert \Sigma}}_{L^2(\Sigma)}
    +\tau^{1/2}\val{(e^{\tau\varphi}  \nabla w_{-})_{\vert \Sigma}}_{L^2(\Sigma)}.
  \notag\end{align}
\end{theorem}

%%%%%%%%%%%%%%%%%%%%%%%%
% remark               %
%%%%%%%%%%%%%%%%%%%%%%%%
\begin{rem}
  It is important to notice that whenever a true discontinuity occurs
  for the vector field $A\nu$, then the space $\W$ does {\it
    not} contain $\moo(\Omega)$: the inclusion $\moo(\Omega)\subset
  \W$ implies from \eqref{2.tr2} that for all $w\in
  \moo(\Omega)$, $\poscal{dw}{A_{+}\nu-A_{-}\nu}=0$ at $\Sigma$ so
  that $A_{+}\nu=A_{-}\nu$ at $\Sigma$, that is continuity for $A\nu$.
  The Carleman estimate which is proven in the present paper takes
  naturally  into account these transmission conditions on the
  function $w$ and it is important to keep in mind that the occurrence
  of a jump is excluding many smooth functions from the space
  $\W$.  On the other hand, we have $\W\subset
  \text{Lip}(\Omega)$.
\end{rem}
%%%%%%%%%%%%%%%%%%%%%%%%
% remark               %
%%%%%%%%%%%%%%%%%%%%%%%%
\begin{rem}\label{2.rem.geomet}
  We can also point out the geometric content of our assumptions,
  which do not depend on the choice of a coordinate system.  For each
  $x\in \Omega$, the matrix $A(x)$ is a positive-definite symmetric
  mapping from $T_{x}(\Omega)^*$ onto $T_{x}(\Omega)$ so that $A(x)
  dw(x)$ belongs indeed to $T_{x}(\Omega)$ and $Adw$ is a vector field
  with a $L^2$ divergence (Inequality~\eqref{eq: Carleman main} yields the $L^2$ bound
  by density).
\end{rem}
If we were to consider a more general framework in which the matrix
$A(x)$, symmetric, positive-definite belongs to $BV(\Omega)\cap
L^\io(\Omega)$, and $w$ is a Lipschitz continuous function on $\Omega$
the vector field $Adw$ is in $L^\io(\Omega)$: the second transmission
condition reads in that framework $\dive(Adw)\in L^\io(\Omega).$
Proving a Carleman estimate in such a case is a wide open question.

\subsection{Sketch of the proof}
\label{sec: sketch of the proof}
We provide in this subsection an outline of the main arguments used in
our proof. To avoid technicalities, we somewhat simplify the geometric
data and the weight function, keeping of course the anisotropy.  We
consider the operator
\begin{equation}\label{1.toy001}
\L_{0}=\ssum_{1\le j\le n}D_{j}c_{j}D_{j},\
c_{j}(x)=H_{+}c_{j}^++H_{-}c_{j}^-,\
c_{j}^{\pm}\text{\small $>0$ constants,}\
H_{\pm}=\ind{\{\pm x_{n}>0\}},
\end{equation}
with $ D_{j}=\frac{\p}{ i\p x_{j}},$ and the vector space $\W_{0}$ of functions
$H_{+}w_{+}+H_{-}w_{-}$, $w_{\pm}\in \mooc(\R^n)$,
such that 
\begin{equation}\label{1.toy002}\text{at $x_{n}=0$,\ }
w_{+}=w_{-},\
c_{n}^+\p_{n}w_{+}=c_{n}^-\p_{n}w_{-}
\
\text{\footnotesize (transmission conditions across $x_{n}=0$).}
\end{equation}
As a result, for $w\in \W_{0}$,
we have $D_{n}w=H_{+}D_{n}w_{+}+H_{-}D_{n}w_{-}$
and 
\begin{equation} \label{1.toygtre}
\L_{0}w=\ssum_{j}(H_{+}c_{j}^+D_{j}^2w_{+}
+H_{-}c_{j}^-D_{j}^2w_{-}).
\end{equation} 
We also consider a weight function\footnote{In the main
  text, we shall introduce some minimal requirements on the weight function
  and suggest other possible choices.
  }
\begin{equation}\label{1.toy003}
\varphi=\underbrace{(\alpha_{+}x_{n}+\beta x_{n}^2/2)}_{\varphi_{+}}H_{+}
+\underbrace{(\alpha_{-}x_{n}+\beta x_{n}^2/2)}_{\varphi_{-}}H_{-},\quad \alpha_{\pm}>0,\quad \beta>0,
\end{equation}
a positive parameter $\tau$ and the vector space
$\W_{\tau}$ of functions 
$H_{+}v_{+}+H_{-}v_{-}$, $v_{\pm}\in \mooc(\R^n)$,
such that at $x_{n}=0$,
\begin{align}
  \label{eq: trans cond1}
  &v_{+}=v_{-},\\
  \label{eq: trans cond2}
  &c_{n}^+(D_{n}v_{+}+i\tau \alpha_{+}v_{+})=c_{n}^-(D_{n}v_{-}+i\tau
  \alpha_{-}v_{-}).
\end{align}
Observe that $w \in \W_{0}$ is equivalent to
$ v = e^{\tau \varphi} w \in W_\tau$. 
We have
$$
e^{\tau \varphi}\L_{0} w=\underbrace{e^{\tau \varphi}\L_{0}e^{-\tau \varphi}}_{\L_{\tau}}
(e^{\tau \varphi}w)
$$
so that proving a weighted a priori estimate
$
\norm{e^{\tau \varphi}\L_{0} w}_{L^2(\R^n)}\gtrsim\norm{e^{\tau \varphi}w}_{L^2(\R^n)}
$
for $w\in \W_{0}$
amounts to getting
$
\norm{\L_{\tau} v}_{L^2(\R^n)}\gtrsim\norm{v}_{L^2(\R^n)}
$
for $v\in \W_{\tau}$.

\medskip
\par\no
{\small \bf Step 1: pseudo-differential factorization.}
Using Einstein convention on repeated indices $j\in \{1,\dots, n-1\}$,
we have
\begin{equation*}
\L_{\tau}=(D_{n}+i\tau\varphi') c_{n}(D_{n}+i\tau\varphi') 
+D_{j}c_{j}D_{j}
\end{equation*}
and for $v\in \W_{\tau}$, from \eqref{1.toygtre},
with
$m_\pm = m_{\pm}(D')=(c_{n}^{\pm})^{-1/2}(c_{j}^{\pm}D_{j}^2)^{1/2}$,
$$
\L_{\tau} v=H_{+}c_{n}^+
\bigl( (D_{n}+i\tau\varphi'_{+})^2+m_{+}^2
\bigr)v_{+}
+H_{-}c_{n}^-
\bigl( (D_{n}+i\tau\varphi'_{-})^2+m_{-}^2
\bigr)v_{-}
$$
so that
\begin{multline}\label{1.toy007}
\L_{\tau} v=
H_{+}c_{n}^+
\bigl(D_{n}+i(\overbrace{\tau \varphi'_{+}+m_{+}}^{e_{+}})\bigr)
\bigl(D_{n}+i(\overbrace{\tau \varphi'_{+}-m_{+}}^{f_{+}})\bigr)v_{+}
\\+
H_{-}c_{n}^-
\bigl(D_{n}+i(\underbrace{\tau \varphi'_{-}-m_{-}}_{f_{-}})\bigr)
\bigl(D_{n}+i(\underbrace{\tau \varphi'_{-}+m_{-}}_{e_{-}})\bigr)v_{-}.
\end{multline}
Note that $e_{\pm}$ are elliptic positive in the sense that
$e_{\pm}=\tau\alpha_{\pm}+m_{\pm}\gtrsim \tau+\val {D'}$.  We want at
this point to use some natural estimates for first-order factors on
the half-lines $\R_{\pm}$: let us for instance check on $t>0$ for
$\omega\in \mooc(\R)$, $\lambda, \gamma$ positive,
\begin{align}
  \label{eq: nice estimate e positive}
  &\norm{D_{t}\omega+i(\lambda+\gamma t) \omega }_{L^2(\R_{+})}^2
  \\
  & \qquad=\norm{D_{t}\omega}_{L^2(\R_{+})}^2
  +
  \norm{(\lambda+\gamma t) \omega }_{L^2(\R_{+})}^2
  +2\re\poscal{D_{t}\omega}{iH(t)(\lambda+\gamma t)\omega}
  \nonumber \\
  & \qquad \ge \int_{0}^{+\io}
  \bigr((\lambda+\gamma t)^2+\gamma\bigr)\val{\omega(t)}^2 dt 
  +\lambda\val{\omega(0)}^2
  \ge \norm{\lambda \omega}_{L^2(\R_{+})}^2+\lambda\val{\omega(0)}^2,
  \nonumber 
\end{align}
which is somehow a perfect estimate of elliptic type, 
suggesting that the first-order factor containing $e_{+}$
should be easy to handle.
Changing $\lambda$ in $-\lambda$ gives
\begin{align*}
  \norm{D_{t}\omega+i(-\lambda+\gamma t) \omega }_{L^2(\R_{+})}^2
  &\ge
  2\re\poscal{D_{t}\omega}{iH(t)(-\lambda+\gamma t)\omega}
  \\
  &= \int_{0}^{+\io} \gamma \val{\omega(t)}^2 dt -\lambda\val{\omega(0)}^2,
\end{align*}
so that
$
\norm{D_{t}\omega+i(-\lambda+\gamma t) \omega }_{L^2(\R_{+})}^2+\lambda\val{\omega(0)}^2\ge \gamma\norm{\omega}_{L^2(\R_{+})}^2,
$
an estimate of lesser quality, because we need to secure a control
of $\omega(0)$ to handle this type of factor.

\medskip
\par\no
{\small \bf Step 2:  case $\bld{f_{+}\ge 0}$.}
Looking at formula \eqref{1.toy007}, since the factor containing
$e_{+}$ is elliptic in the sense given above, we have to discuss on
the sign of $f_{+}$. Identifying the operator with its symbol, we have
$f_{+}=\tau (\alpha_{+}+\beta x_{n})-m_{+}(\xi'),$ and thus
$\tau\alpha_{+}\ge m_{+}(\xi')$ yielding a positive $f_+$.  Iterating the
method outlined above on the half-line $\R_{+}$, we get a nice
estimate of the form of \eqref{eq: nice estimate e positive} on
$\R_{+}$; in particular we obtain a control of $v_{+}(0)$.  From the
transmission condition, we have $v_{+}(0) =v_{-}(0)$ and hence this
amounts to also controlling $v_{-}(0)$.  That control along with the
natural estimates on $\R_{-}$ are enough to prove an inequality
  of the form of the sought Carleman estimate.

\medskip
\par\no
{\small \bf Step 3: case $\bld{f_{+}<0}$.} Here, we assume that 
$\tau\alpha_{+}< m_{+}(\xi')$.  We can still use on $\R_{+}$ the
factor containing $e_{+}$, and by \eqref{1.toy007} and \eqref{eq: nice
  estimate e positive} control the following quantity
\begin{equation}\label{1.toy010}
  c_{n}^+(D_{n}+if_+)v_{+}(0)
  = \overbrace{c_{n}^+(D_{n}v_{+}+i\tau \alpha_{+})v_{+}(0)}^{={\V}_{+}}
  -c_{n}^+im_{+}v_{+}(0).
\end{equation}
Our key assumption is
\begin{equation}\label{1.toy008}
  f_{+}(0)<0\Longrightarrow f_{-}(0)\le 0.
\end{equation}
Under that hypothesis, we can use the negative factor $f_{-}$ on
$\R_{-}$ (note that $f_{-}$ is increasing with $x_{n}$, so that
$f_{-}(0)\le 0\Longrightarrow f_{-}(x_{n})<0$ for $x_{n}<0$).  We then
control
\begin{equation}\label{1.toy009}
  c_{n}^-(D_{n}+i e_-)v_{-}(0)
  = \underbrace{c_{n}^-(D_{n}v_{-}+i\tau \alpha_{-})v_{-}(0)}_{={\V}_{-}}
  +c_{n}^-im_{-}v_{-}(0).
\end{equation}
Nothing more can be achieved with inequalities on each side of the
interface.  At this point we however notice that the second
transmission condition in \eqref{eq: trans cond2} implies ${\mathcal
  V}_{-}={\V}_{+}$, yielding the control of the
difference of \eqref{1.toy009} and \eqref{1.toy010}, \ie,
of
$$
c_{n}^-im_{-}v_{-}(0)+c_{n}^+im_{+}v_{+}(0)
=i\bigl(c_{n}^-m_{-}+c_{n}^+m_{+}\bigr)v(0).
$$
Now, as $c_{n}^-m_{-}+c_{n}^+m_{+}$ is elliptic positive, this gives a
control of $v(0)$ in (tangential) $H^1$-norm, which is enough then to
get an estimate on both sides that leads to the sought Carleman estimates.

\medskip
\par\no
{\small \bf Step 4: patching estimates together.}
The analysis we have sketched here relies on a separation into two
zones in the $(\tau,\xi')$ space. Patching the estimates of the form
of \eqref{eq: Carleman main} in each zone together allows us to
conclude the proof of the Carleman estimate.

\subsection{Explaining the key assumption}
\label{sec: explaining assumption}

In the first place, our key assumption, condition~\eqref{1.toy008}, can
be reformulated as
\begin{equation}\label{1.toybis}
\forall \xi'\in \mathbb S^{n-2},\quad
\frac{\alpha_{+}}{\alpha_{-}}\ge \frac{m_{+}(\xi')}{m_{-}(\xi')}.
\end{equation}

In fact \footnote{ For the main theorem, we shall in fact require the
  stronger strict inequality
  \begin{equation}\label{1.toy011}
    \frac{\alpha_{+}}{\alpha_{-}}>\frac{m_{+}(\xi')}{m_{-}(\xi')}.
  \end{equation}
  However, we shall see in Section~\ref{sec: counter-example} that in the
  particular case presented here, where the matrix $A$ is piecewise
  constant and the weight function $\varphi$ solely depends on $x_{n}$
  the inequality \eqref{1.toybis} is actually a {\em necessary and
    sufficient} condition to obtain a Carleman estimate with weight
  $\varphi$. }, \eqref{1.toy008} means $\tau
\alpha_{+}<m_{+}(\xi')\Longrightarrow \tau \alpha_{-}\le m_{-}(\xi')$
and since $\alpha_{\pm},m_{\pm}$ are all positive, this is equivalent
to having $ {m_{+}(\xi')}/{\alpha_{+}}\le {m_{-}(\xi')}/{\alpha_{-}},\
\text{which is \eqref{1.toybis}}.  $ An analogy with an estimate for a
first-order factor may shed some light on this condition.  With
$$f(t)=H(t)(\tau\alpha_{+}+\beta t-m_{+})+H(-t)(\tau\alpha_{-}+\beta t-m_{-}),
\quad
\text{\small $\tau, \alpha_{\pm},\beta , m_{\pm}$ positive constants,}
$$
we want to prove an injectivity estimate of the type
$\norm{D_{t}v+if(t) v}_{L^2(\R)}\gtrsim \norm{v}_{L^2(\R)}$, say for
$v\in \mooc(\R)$.  It is a classical fact (see \eg Lemma~3.1.1 in
\cite{lernerbook}) that such an estimate (for a smooth $f$) is
equivalent to the condition that $t\mapsto f(t)$ does not change sign
from $+$ to $-$ while $t$ increases: it means that the adjoint
operator $D_{t}-if(t)$ satisfies the so-called condition $(\Psi)$.
Looking at the function $f$, we see that it increases on each
half-line $\R_{\pm}$, so that the only place to get a ``forbidden''
change of sign from $+$ to $-$ is at $t=0$: to get an injectivity
estimate, we have to avoid the situation where $f(0^{+})<0$ and $
f(0^{-})>0, $ that is, we have to make sure that $
f(0^{+})<0\Longrightarrow f(0^{-})\le 0,\text{\ which is indeed the
  condition \eqref{1.toybis}}.  $ The function $f$ is increasing
affine on $\R_{\pm}$ with the same slope $\beta$ on both sides, with a
possible discontinuity at 0.\par
\begin{figure}[h]
\scalebox{1.5}{\includegraphics{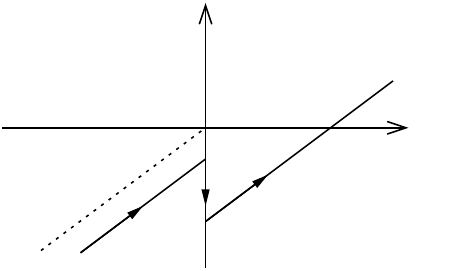}\includegraphics{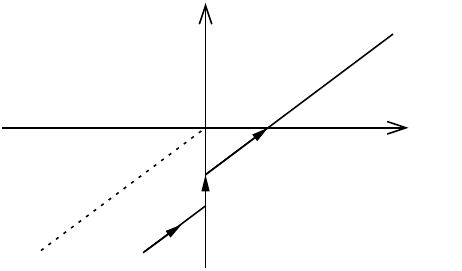}}
\caption{$f(0^-) \leq 0$; $f(0^+) <0$.}
\end{figure}
When $f(0^{+})<0$ we should have $f(0^{-})\le 0$ and the line
on the left cannot go above the dotted line, in such a way that 
the discontinuous  zigzag curve with the arrows has only a change of sign from $-$ to $+$.
\par\vs
\begin{figure}[h]
\scalebox{1.5}{\includegraphics{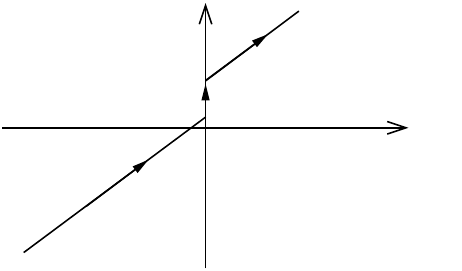}}
\caption{$f(0^-) \gtrless 0$; $f(0^+) \geq 0$.}
\label{pic001}
\end{figure}
When $f(0^{+})\ge 0$, there is no other constraint on $f(0^{-})$: even with a discontinuity, the change of sign can only occur from $-$ to $+$.

We prove below (Section~\ref{sec: counter-example}) that condition
\eqref{1.toybis} is relevant to our problem in the sense that it is
indeed necessary to have a Carleman estimate with this weight: if
\eqref{1.toybis} is violated, we are able for this model to
construct a quasi-mode for $\L_{\tau}$, i.e. a $\tau$-family
of functions $v$ with $L^2$-norm 1 such that $\norm{\mathcal
  L_{\tau}v}_{L^2} \ll \norm{v}_{L^2}$, as $\tau$ goes to $\infty$,
ruining any hope to prove a Carleman estimate.  As usual for this type
of construction, it uses some type of complex geometrical optics
method, which is easy in this case to implement directly, due to the
simplicity of the expression of the operator.

%%%%%%%%%%%%%%%%%%%%%%%%
% remark               %
%%%%%%%%%%%%%%%%%%%%%%%%
\begin{rem}
  A very particular case of anisotropic medium was tackled in
  \cite{LRR:10} for the purpose of proving a controllability result
  for linear parabolic equations. The condition imposed on the weight
  function in \cite{LRR:10} (Assumption 2.1 therein) is much more
  demanding than what we impose here.  In the isotropic case, $c^\pm_j
  = c_\pm$ for all $j \in \{1,\dots, n\}$, we have $m_+ = m_- =
  |\xi'|$ and our condition \eqref{1.toy011} reads $\alpha_+ >
  \alpha_-$.  Note also that the isotropic case $c_- \geq c_+$ was
  already considered in \cite{MR1932966}.

  In \cite{LRR:10}, the controllability result concerns an isotropic
  parabolic equation. The Carleman estimate we derive here extends
  this result to an anisotropic parabolic equation.
\end{rem}

%  Framework
\section{Framework}

\subsection{Presentation}
\label{sec: presentation}
Let $\Omega, \Sigma$ be as in \eqref{2.ome}.  With $$\Xi=\{
\text{positive-definite $n\times n$ matrices} \},$$ we consider
$A_{\pm}\in \moo(\Omega;\Xi)$ and let $\L, \varphi$ be as in
\eqref{2.ope} and \eqref{2.wei}.
We set 
\begin{equation*}
  \L_\pm=D\cdot A_\pm D=-\dive(A_\pm \nabla).
\end{equation*}

Here, we generalize our analysis to non-homogeneous transmission conditions:
for $\theta$ and $\Theta$ smooth functions of the interface $\Sigma$ we set
\begin{align}
  \label{eq: trans cond}
  w_{+}-w_{-}= \theta, \quad \text{and}\ \ 
  \poscal{A_{+}dw_{+}-A_{-}dw_{-}}{\nu}=\Theta \quad \text{at} \ \Sigma,
\end{align}
(compare with \eqref{2.tr1}-\eqref{2.tr2})
and introduce
\begin{align}
  \label{eq: trans space nonhomogeneous}
  \W^{\theta,\Theta}_0 = \big\{ H_- w_- + H_+ w_+\big\}_{w_\pm \in \mooc(\Omega) \ \text{satisfying}\ \eqref{eq: trans cond}.} 
\end{align}
For $\tau\ge 0$ we define the affine space
\begin{equation}\label{2.doublev}
\W^{\theta,\Theta}_{\tau}=\{e^{\tau \varphi}w\}_{w\in \W_{0}^{\theta,\Theta}}.
\end{equation}
For $v\in \W^{\theta,\Theta}_{\tau}$, we have $v= e^{\tau \varphi}w$ with $w\in
\W^{\theta,\Theta}_{0}$ so that, using the notation introduced in \eqref{2.hea},
\eqref{2.tes}, with $v_{\pm}=e^{\tau \varphi_{\pm}}w_{\pm},$ we have
\begin{equation}\label{2.caf}
v=H_{-}v_{-}+H_{+}v_{+},\quad 
\end{equation}
and we see that the transmission conditions \eqref{eq: trans cond}
on $w$ read for $v$ as 
\begin{equation}\label{eq: trans cond varphi}
v_{+}-v_{-}  = \theta_\varphi, \quad 
\poscal{dv_{+}-\tau v_{+}d\varphi_{+}}{A_{+}\nu}-
\poscal{dv_{-}-\tau v_{-} d\varphi_{-}}{A_{-}\nu} = \Theta_\varphi, 
\quad \text{at $\Sigma$,}
\end{equation}
with 
\begin{equation}
  \theta_\varphi = e^{\tau \varphi_{|\Sigma}} \theta, \quad 
  \Theta_\varphi = e^{\tau \varphi_{|\Sigma}} \Theta.
\end{equation}
Observing that $e^{\tau \varphi_{\pm}} D e^{-\tau
  \varphi_{\pm}}=D+i\tau d\varphi_{\pm}$, for $w\in \W^{\theta,\Theta}$, we
obtain
\begin{align*}
&e^{\tau \varphi_\pm}\L_\pm w_\pm
=e^{\tau \varphi_{\pm}} D\cdot A_{\pm} D e^{-\tau \varphi_{\pm}}v_{\pm}
=(D+i\tau d\varphi_{\pm})\cdot 
A_{\pm}(D+i\tau d\varphi_{\pm})v_{\pm}
\end{align*}

We define
\begin{equation}\label{2.ppm}
\P_{\pm}=(D+i\tau d\varphi_{\pm})\cdot A_{\pm}(D+i\tau d\varphi_{\pm}).
\end{equation}

%%%%%%%%%%%%%%%%%%%%%%%%
% proposition          %
%%%%%%%%%%%%%%%%%%%%%%%%
\begin{pro}
  \label{mainprop}
  Let $\Omega, \Sigma, \L,\W^{\theta,\Theta}_{\tau}$
  be as in \eqref{2.ome},  \eqref{2.ope} and \eqref{2.doublev}.
  Then for any compact subset
  $K$ of $\Omega$, there exist a weight function $\varphi$ satisfying
  \eqref{2.wei} and positive constants $C$, $\tau_1$ such that for all
  $\tau\ge \tau_1$ and all $v\in \W_{\tau}$ with $\supp v\subset K$
  \begin{multline*}
    C \big(\norm{H_{-}\P_{-}v_{-}}_{L^2(\R^n)}+ \norm{H_{+}\P_{+}v_{+}}_{L^2(\R^n)}
     + {\mathcal T}_{\theta, \Theta} \big)  \ge \tau^{3/2}\val{{v_\pm}}_{L^2(\Sigma)} 
     + \tau^{1/2}\val{( \nabla v_{\pm})}_{L^2(\Sigma)}\\  
     + \tau^{3/2}\norm{v}_{L^2(\R^n)}
     +{\tau^{1/2}\norm{H_{+} \nabla v_{+}}_{L^2(\R^n)}}
     +{\tau^{1/2}\norm{H_{-} \nabla v_{-}}_{L^2(\R^n)}},
  \end{multline*}
where $\displaystyle{
  {\mathcal T}_{\theta, \Theta} =  \tau^{3/2} \val{\theta_\varphi}_{L^2(\Sigma)} 
     +  \tau^{1/2} \val{\nabla_\Sigma \theta_\varphi}_{L^2(\Sigma)} 
     +  \tau^{1/2} \val{\Theta_\varphi}_{L^2(\Sigma)}}$.
\end{pro}
Here, $\nabla_\Sigma$ denotes the tangential gradient to $\Sigma$.
The proof of this proposition will occupy a large part of the
remainder of the article (Sections~\ref{sec: estimate 1st order
  factor} and \ref{Sec: proof Carleman}) as it implies the result of
the following theorem, a non-homogenous version of 
Theorem~\ref{1.thm.main}.
%%%%%%%%%%%%%%%%%%%%%%%%
% theorem              %
%%%%%%%%%%%%%%%%%%%%%%%%
\begin{theorem}
  \label{thm.main - nonhomogeneous}
  Let $\Omega, \Sigma, \L,\W^{\theta,\Theta}_0$ be as in
  \eqref{2.ome}, \eqref{2.ope} and \eqref{eq: trans space
    nonhomogeneous}.  Then for any compact subset $K$ of $\Omega$,
  there exist a weight function $\varphi$ satisfying \eqref{2.wei} and
  positive constants $C$, $\tau_1$ such that for all $\tau\ge \tau_1$
  and all $w\in \W$ with $\supp w\subset K$,
  \begin{multline}
    \label{eq: Carleman main - nonhomogeneous}
    C\big(\norm{H_- e^{\tau \varphi_-}\L_- w_-}_{L^2(\R^n)}
    + \norm{H_+ e^{\tau \varphi_+}\L_+ w_+}_{L^2(\R^n)} + T_{\theta,\Theta} \big)
    \\
    \geq \tau^{3/2}\norm{e^{\tau \varphi} w}_{L^2(\R^n)}
    +\tau^{1/2} \big(\norm{H_{+}e^{\tau \varphi} \nabla w_{+}}_{L^2(\R^n)}
    +\norm{H_{-}e^{\tau \varphi} \nabla w_{-}}_{L^2(\R^n)}\big)
    \\
    +\tau^{3/2} 
    \val{e^{\tau \varphi} w_\pm}_{L^2(\Sigma)}
    +\tau^{1/2}  
    \val{e^{\tau\varphi} \nabla w_{\pm}}_{L^2(\Sigma)}.
  \end{multline}
where $\displaystyle{
T_{\theta,\Theta} = \tau^{3/2} \val{e^{\tau \varphi_{|\Sigma}} \theta}_{L^2(\Sigma)} 
     +  \tau^{1/2} \val{e^{\tau \varphi_{|\Sigma}} \nabla_\Sigma \theta}_{L^2(\Sigma)} 
     +  \tau^{1/2} \val{e^{\tau \varphi_{|\Sigma}} \Theta}_{L^2(\Sigma)}}$.
\end{theorem}

\medskip
Theorem~\ref{1.thm.main} corresponds to the case $\theta=\Theta=0$ since
by \eqref{2.trt} we then have
\begin{align*}
&\norm{e^{\tau \varphi}\L w}_{L^2(\R^n)}=
\norm{H_- e^{\tau \varphi_-}\L_- w_-}_{L^2(\R^n)}
    + \norm{H_+ e^{\tau \varphi_+}\L_+ w_+}_{L^2(\R^n)} 
\end{align*} 

\begin{rem}
  It is often useful to have such a Carleman estimate at hand for the
  case non-homogeneous transmission conditions, for examples when on
  tries to patch such local estimates together in the \nhd of the
  interface.
  
  Here, we derive local Carleman estimates.  We can in fact consider
  similar geometrical situation on a Riemannian manifold (with
  or without boundary) with an metric exhibiting jump discontinuities
  across interfaces.  For the associated Laplace-Beltrami operator the
  local estimates we derive can be patched together to yield a
  global estimate. We refer to \cite[Section 5]{LRR:09} for such
  questions.
\end{rem}

\begin{proof}[Proof that Proposition \ref{mainprop}
implies Theorem~\ref{thm.main - nonhomogeneous}]
Replacing
$v$ by $e^{\tau \varphi} w$,
we get
\begin{multline}\label{884412}
  \norm{H_- e^{\tau \varphi_-}\L_- w_-}_{L^2(\R^n)}
    + \norm{H_+ e^{\tau \varphi_+}\L_+ w_+}_{L^2(\R^n)} 
    + T_{\theta,\Theta}\\
 \gtrsim
  \tau^{3/2}\norm{e^{\tau \varphi}w}_{L^2(\R^n)}
  +
  \tau^{1/2}\big(\norm{H_{+} \nabla e^{\tau \varphi}w_{+}}_{L^2(\R^n)}
  + \norm{H_{-}\nabla e^{\tau \varphi}w_{-}}_{L^2(\R^n)}\big)\\
  +\tau^{3/2} \val{e^{\tau \varphi}{w_\pm}}_{L^2(\Sigma)}
  +\tau^{1/2} \val{ \nabla e^{\tau \varphi}{w_\pm}}_{L^2(\Sigma)}.
 \end{multline}
Commuting $\nabla$ with $e^{\tau \varphi}$ produces
\begin{multline*}
  C\big(\norm{H_- e^{\tau \varphi_-}\L_- w_-}_{L^2(\R^n)}
    + \norm{H_+ e^{\tau \varphi_+}\L_+ w_+}_{L^2(\R^n)} 
    + T_{\theta,\Theta}\big)\\
 +C_{1}\tau^{3/2}\norm{e^{\tau \varphi}w}_{L^2(\R^n)}
 +C_{2}\tau^{3/2}\big(
 \val{e^{\tau \varphi}{w_\pm}_{\vert \Sigma}}_{L^2(\Sigma)}\big)\\
 \ge
   \tau^{1/2}\norm{H_{-}e^{\tau \varphi} Dw_{-}}_{L^2(\R^n)}
  + \tau^{1/2}\norm{H_{+}e^{\tau \varphi} Dw_{+}}_{L^2(\R^n)}
  +\tau^{3/2}\norm{e^{\tau \varphi}w}_{L^2(\R^n)}\\
   +\tau^{1/2}  \val{e^{\tau \varphi}{Dw_{\pm}}}_{L^2(\Sigma)}
   + \tau^{3/2}
   \val{e^{\tau \varphi}{w_\pm}}_{L^2(\Sigma)}   
   ,
 \end{multline*}
but from \eqref{884412} we have
 \begin{multline*}
 C_{1}\tau^{3/2}\norm{e^{\tau \varphi}w}
 +C_{2}\tau^{3/2}\val{e^{\tau \varphi}w}
 \\ \le C\max(C_{1},C_{2}) \big(
 \norm{H_- e^{\tau \varphi_-}\L_- w_-}_{L^2(\R^n)}
    + \norm{H_+ e^{\tau \varphi_+}\L_+ w_+}_{L^2(\R^n)} 
    + T_{\theta,\Theta}\big),
 \end{multline*}
 proving the implication.
\end{proof}

\subsection{Description in local coordinates}
\label{subcoord}

Carleman estimates of types \eqref{eq: Carleman main}
  and \eqref{eq: Carleman main - nonhomogeneous} can be handled
  locally as they can be patched together.  Assuming as we may that the
hypersurface $\Sigma$ is given locally by the equation $\{x_{n}=0\}$,
we have, using the Einstein convention on repeated indices $j\in
\{1,\dots, n-1\}$, and noting from the ellipticity condition that
$a_{nn}>0$ (the matrix $A(x)=(a_{jk}(x))_{1\le j,k\le n}$),
\begin{align*}
\L&= D_{n}a_{nn}D_{n}+D_{n}a_{nj} D_{j}+D_{j}a_{jn}D_{n}+D_{j}a_{jk}D_{k},
\\
&= D_{n}a_{nn}\bigl( D_{n}+{a_{nn}^{-1}a_{nj} D_{j}}\bigr)+D_{j}a_{jn}D_{n}+D_{j}a_{jk}D_{k},
\end{align*}
With $T=a_{nn}^{-1}a_{nj} D_{j}$
we have
$$
\L= \bigl( D_{n}+T^*)a_{nn}\bigl( D_{n}+T\bigr)
-T^*a_{nn}D_{n}-T^*a_{nn}T+D_{j}a_{jn}D_{n}+D_{j}a_{jk}D_{k}.
$$
and since
 $T^*=D_{j}a_{nn}^{-1} a_{nj}$, we have
$T^*a_{nn}D_{n}=D_{j}a_{nj}D_{n}=D_{j}a_{jn}D_{n}$
and
\begin{equation}\label{3.kbgf44}
\L= \bigl( D_{n}+T^*)a_{nn}\bigl( D_{n}+T\bigr)+D_{j}b_{jk}D_{k},
\end{equation}
where the $(n-1)\times (n-1)$ matrix $(b_{jk})$ is positive-definite
since with $\xi'=(\xi_{1},\dots, \xi_{n-1})$ and $\xi=(\xi',
\xi_{n})$,
$$
\poscal{B\xi'}{\xi'}=\ssum_{1\le j,k\le n-1}b_{jk}\xi_{j}\xi_{k}
=\poscal{A\xi}{\xi},
$$ where
$a_{nn}\xi_{n}=-\ssum_{1\le j\le n-1}a_{nj}\xi_{j}$.  Note also that
$b_{jk}=a_{jk}-({a_{nj}a_{nk}}/{a_{nn}}).  $
%%%%%%%%%%%%%%%%%%%%%%%%
% remark               %
%%%%%%%%%%%%%%%%%%%%%%%%
\begin{rem}
  The positive-definite quadratic form $B$ is the restriction of
  $\poscal{A\xi}{\xi}$ to the hyperplane $\H$ defined by $
  \poi{\poscal{A\xi}{\xi}}{x_{n}}= {\p
    _{\xi_{n}}}\bigl(\poscal{A\xi}{\xi}\bigr)=0, $ where
  $\poi{\cdot}{\cdot}$ stands for the Poisson bracket.  In fact the
  principal symbol of $\L$ is $\poscal{A(x) \xi}{\xi}$ and if
  $\Sigma$ is defined by the equation $\psi(x)=0$ with $d\psi\not=0$
  at $\Sigma$, we have
  $$
  \frac12\Bigl\{{\poscal{A(x) \xi}{\xi}},{\psi}\Bigr\}
  =\poscal{A(x) \xi}{d\psi(x)}
  $$
  so that $\H_{x}=\bigl(A(x) d\psi(x)\bigr)^\perp=\{\xi\in
  T_{x}^*(\Omega), \poscal{\xi}{A(x) d\psi(x)}_{T_{x}^*(\Omega),
    T_{x}(\Omega)}=0\}.  $ When $x\in \Sigma$, that set does not
  depend on the choice of the defining function $\psi$ of $\Sigma$ and
  we have simply
  $$
  \H_{x}=\bigl(A(x)\nu(x)\bigr)^\perp=
  \{\xi\in T_{x}^*(\Omega),\poscal{\xi}{A(x) \nu(x)}_{T_{x}^*(\Omega), T_{x}(\Omega)}
  =0\}
  $$
  where $\nu(x)$ is the conormal vector to $\Sigma$ at $x$ (recall
  that from Remark \ref{2.rem.geomet}, $\nu(x)$ is a cotangent vector
  at $x$, $A(x)\nu(x)$ is a tangent vector at $x$).  Now, for $x\in
  \Sigma$, we can restrict the quadratic form $A(x)$ to $\mathcal
  H_{x}$: this is the positive-definite quadratic form $B(x)$,
  providing a coordinate-free definition.
\end{rem}
For $w\in \W^{\theta,\Theta}_0$, we have
\begin{align}\label{2.opp}
  \L_\pm w_\pm= (D_{n}+T_{\pm}^*)a_{nn}^\pm(D_{n}+T_{\pm})w_{\pm}
  +D_{j}b_{jk}^\pm D_{k}w_{\pm}
\end{align}
and the non-homogeneous transmission conditions \eqref{eq: trans cond}
 read
\begin{equation}\label{eq: trans cond 2}
  w_{+}-w_{-}= \theta, \quad a_{nn}^+(D_{n}+T_{+})w_{+} - a_{nn}^-(D_{n}+T_{-})w_{-}=\Theta ,
  \quad\text{at $\Sigma$.}
\end{equation}
%% subsection
\subsection{Pseudo-differential factorization on each side}
\label{sec: pseudo factor}
At first we consider the weight function $\varphi = H_+ \varphi_+ +
H_- \varphi_-$ with $\varphi_\pm$ that solely depend on $x_n$. Later
on we shall allow for some dependency upon the tangential variables $x'$
(see Section~\ref{sec: convexification}).
We define for $m\in \R$ the class of tangential
standard symbols $\S^m$ as the smooth functions on
$\R^n\times\R^{n-1}$ such that, for all $(\alpha, \beta)\in \N^n\times
\N^{n-1}$,
\begin{equation}\label{2.pseudo}
  \sup_{(x,\xi')\in \R^n\times \R^{n-1}}
  \valjp{\xi'}^{-m+\val\beta} \val{(\p_{x}^\alpha\p_{\xi'}^\beta a)(x,\xi')}<\io,
\end{equation}
with $\valjp{\xi'} = \big(1 + |\xi'|^2\big)^\hf$.  Some basic
properties of standard pseudo-differential operators are recalled in
Appendix \ref{app.sub.pseudo}.  
Section \ref{subcoord} and formul\ae~\eqref{2.ppm}, \eqref{2.opp} give
\begin{align}
  \label{eq: conjugated operator}
  \P_{\pm} = \big(D_{n}+i \tau \varphi_\pm' + T_{\pm}^*\big)a_{nn}^\pm
  \big(D_{n}+i \tau \varphi_\pm' +T_{\pm}\big) +D_{j}b_{jk}^\pm D_{k}.
\end{align}
We define $m_{\pm}\in \S^1$ such that 
\begin{align}
  \label{eq: m pm for xi' large}
  \text{for $\val {\xi'}\ge 1$, }
  m_\pm = \Big( \frac{b^\pm_{j k}}{a^\pm_{n n}} \xi_j \xi_k\Big)^\hf,
  \quad m_\pm \geq C \valjp{\xi'}, 
  \quad
  M_{\pm}=\op{m_{\pm}}.
\end{align}
We have then
$M_{\pm}^2\equiv D_{j}b_{jk}^\pm D_{k}\mod\op{\S^1}$.

We define
\begin{equation}\label{3.classe}
  {\Psi^1}=\op{\S^1}+\tau\op{\S^0}+\op{\S^0}D_{n}.
\end{equation}
Modulo the operator class $\Psi^1$ we may write 
\begin{equation}\label{001}
  \P_{+}\equiv \P_{E+} a_{nn}^+ \P_{F+}, 
  \qquad 
  \P_{-}\equiv \P_{F-} a_{nn}^- \P_{E-},
\end{equation}
where 
\begin{equation}\label{002}
  \P_{E\pm} =  D_{n}+S_{\pm}+i(\underbrace{\tau \varphi'_{\pm}+M_{\pm}}_{E_{\pm}}), 
  \quad 
  \P_{F\pm} = D_{n}+S_{\pm}+i(\underbrace{\tau \varphi'_{\pm}-M_{\pm}}_{F_{\pm}}),
\end{equation}
with
\begin{equation}
\label{eq: def S}
S_{\pm}=s^w(x,D'), \quad 
s_\pm = \ssum_{1\le j\le n-1}\frac{a_{nj}^{\pm}}{a_{nn}^{\pm}}\xi_{j},
\quad\text{so that}\ S_{\pm}^\ast=S_{\pm},\ \
S_{\pm}=T_{\pm}+\frac12\dive T_{\pm},
\end{equation}
where  
  \begin{equation}\label{vectT}
\text{$T_{\pm}$ is the vector field}
\quad
\ssum_{1\le j\le
  n-1}\frac{a_{nj}^{\pm}}{i a_{nn}^{\pm}}\p_{j}.
\end{equation}

We denote by $f_\pm$ and $e_\pm$ the homogeneous principal
symbols of $F_\pm$ and
  $E_\pm$ respectively, determined modulo the symbol class $\S^1 + \tau \S^0$.
The transmission conditions
\eqref{eq: trans cond 2} with our choice of coordinates read, at $x_{n}=0$,
\begin{equation}
  \label{eq: transmission conditions varphi}
  \begin{cases}
    v_{+}-v_{-}  = \theta_\varphi = e^{\tau \varphi_{|x_n=0}} \theta,\\
    a_{nn}^+(D_{n}+T_{+}+i\tau\varphi'_{+})v_{+} - 
    a_{nn}^-(D_{n}+T_{-}+i\tau\varphi'_{-})v_{-}= \Theta_\varphi 
    = e^{\tau \varphi_{|x_n=0}} \Theta.
    \end{cases}
\end{equation}

%%%%%%%%%%%%%%%%%%%%%%%%
% remark               %
%%%%%%%%%%%%%%%%%%%%%%%%
\begin{rem}
  \label{remark: symbol modulo psi1}
  Note that the Carleman estimate we shall prove is insensitive to
  terms in $\Psi^1$ in the conjugated operator
  $\P$. Formul\ae~\eqref{001}
  and \eqref{002} for $\P_+$ and $\P_-$
  will thus be the base of our analysis.
\end{rem}
\begin{rem}
  \label{remark: roots}
  In the articles \cite{LRR:10,LRR:09}, the zero crossing of the roots
  of the symbol of $\P_\pm$, as seen as a polynomial in $\xi_n$, is
  analyzed. Here the factorization into first-order operators isolates
  each root. In fact, $f_\pm$ changes sign and we shall impose a
  condition on the weight function at the interface to obtain a certain
  scheme for this change of sign. See Section~\ref{Sec: proof
    Carleman}.
\end{rem}

\subsection{Choice of weight-function}
\label{sec: choice weight function}
The weight function can be taken of the form 
\begin{equation}
  \label{30weight}
  \varphi_{\pm}(x_{n})=
  \alpha_{\pm}x_{n}+\beta x_{n}^2/2,\quad\alpha_{\pm}>0,\quad \beta>0.
\end{equation}
The choice of the parameters $\alpha_{\pm}$ and $\beta$ will be done below
and that choice will take into account the geometric data of our
problem: $\alpha_\pm$ will be chosen to fulfill a geometric condition
at the interface and $\beta>0$ will be chosen large.  Here, we shall
require $\varphi' \geq 0$, that is, an ``observation'' region on the
\rhs of $\Sigma$. As we shall need $\beta$ large, this amounts to
working in a small \nhd of the interface, \ie, $|x_n|$ small.  Also,
we shall see below (Section~\ref{sec: convexification}) that this
weight can be perturbed by any smooth function with a small gradient.

Other choices for the weight functions are possible.  In fact, two
sufficient conditions can put forward. We shall describe them now.

The operators $M_{\pm}$ have a principal symbol $m_{\pm}(x,\xi')$ in
$\mathcal S^1$, which is positively-homogeneous\footnote{The
  homogeneity property means as usual $m_{\pm}(x,\rho \xi')=\rho
  m_{\pm}(x,\xi')$ for $\rho\ge 1, \val{\xi'}\ge 1$.}  of degree 1 and
elliptic, i.e. there exists $\lambda^{\pm}_{0}, \lambda^{\pm}_{1}$
positive such that for $\val {\xi'}\ge 1, x\in \R^n$,
\begin{equation}\label{3.ell012}
\lambda^{\pm}_{0}\val{\xi'}\le m_{\pm}(x,\xi')\le \lambda_{1}^{\pm}\val{\xi'}.
\end{equation}
We choose $\varphi'_{|x_n=0^\pm}=\alpha_\pm$ 
such that
\begin{equation}\label{mainassu}
  \frac{\alpha_{+}}{\alpha_{-}}>
  \sup_{x',\xi' \atop |\xi'|\geq 1} \frac{m_+(x',\xi')\br}
  {m_-(x',\xi')\bl}.
\end{equation}
The consequence of this condition will be made clear in
Section~\ref{Sec: proof Carleman}. We shall also prove that this condition
is sharp in Section~\ref{sec: counter-example}: a strong violation of
this condition, \viz, $\alpha_{+}/\alpha_{-} < \sup (m_+ / m_-)_{|x_n=0}$, ruins any possibility of deriving a Carleman
estimate of the form of Theorem~\ref{1.thm.main}.

Condition~\eqref{mainassu} concerns the behavior of the weight
function at the interface. Conditions away from the
interface are also needed.  These conditions are more classical.  From~\eqref{eq:
  conjugated operator}, the symbols of $\P_\pm$, modulo the symbol
class $\S^1+ \tau \S^0 + \S^0 \xi_n$, are given by $p_\pm(x,\xi,\tau)
= a_{nn}^\pm \big( q_2^\pm + 2 i q^\pm_1\big)$, with 
\begin{align*}
  q^\pm_{2} = (\xi_n +s_\pm)^2 + \frac{b^\pm_{jk}}{a^\pm_{n n}}\xi_j
  \xi_k - \tau^2 (\varphi_\pm')^2, \qquad q^\pm_{1} = \tau
  \varphi_\pm' (\xi_n +s_\pm),
\end{align*}
for $\varphi$ solely depending on $x_n$, and from the construction of $m_\pm$, for $|\xi'|\geq 1$, we have
\begin{align}
  \label{eq: q2 for xi' large}
  q^\pm_{2} = (\xi_n+s_\pm)^2 + m_\pm^2 - (\tau \varphi_\pm')^2 
  = (\xi_n+s_\pm)^2 - f_\pm e_\pm.
\end{align}
We can then formulate the usual {\em sub-ellipticity} condition, with
{\em loss of a half-derivative}:
\begin{align}
  \label{eq: sub-ellipticity condition}
  q_{2}^\pm =0\ \text{and}\ q_{1}^\pm=0\quad \Longrightarrow 
  \quad \{ q_{2}^\pm, q_{1}^\pm\}>0.
\end{align}
It is important to note that this property is coordinate free.  For
second-order elliptic operators with real smooth coefficients this
property is necessary and sufficient for a Carleman estimate as that
of Theorem~\ref{1.thm.main} to hold (see \cite{MR0161012} or \eg
\cite{LL:09}).

With the weight functions provided in \eqref{30weight} we choose
$\alpha_\pm$ according to condition~\eqref{mainassu} and we choose
$\beta>0$ large enough and we restrict ourselves to a small \nhd
of $\Sigma$, \ie, $|x_n|$ small to have $\varphi' >0$, and so that
\eqref{eq: sub-ellipticity condition} is fulfilled.

%%%%%%%%%%%%%%%%%%%%%%%%
% remark               %
%%%%%%%%%%%%%%%%%%%%%%%%
\begin{rem}
Other ``classical'' forms for the weight function
$\varphi$ are also possible. For instance, one may use
$\varphi (x_n) = e^{\beta \phi(x_n)}$ with
the function $\phi$ depending solely on $x_n$ of the form
$$
\phi  = H_- \phi_- + H_+ \phi_+, \quad \phi_\pm \in \mooc(\R),
$$ such that $\phi$ is {\em continuous} and $|\phi_\pm'|\geq C >0$.
In this case, property \eqref{mainassu} can be fulfilled by properly
choosing $\phi'_{|x_n=0^\pm}$ and  \eqref{eq:
  sub-ellipticity condition} by choosing $\beta$ sufficiently large.
\end{rem}

Property~\eqref{eq: sub-ellipticity condition} concerns the conjugated
second-order operator. We show now that this condition concerns in
fact only one of the first-order terms in the pseudo-differential
factorization that we put forward above, namely $\P_{F\pm}$.

%%%%%%%%%%%%%%%%%%%%%%%%
% lemma                %
%%%%%%%%%%%%%%%%%%%%%%%%
\begin{lem}
  \label{lemma: effective sub-ellipticiy condition}
  There exist $C>0$, $\tau_1>1$ and $\delta>0$ such that for $\tau\geq \tau_1$
  $$|f_\pm| \leq \delta \lambda \quad \Longrightarrow \quad 
  C^{-1} \tau \leq |\xi'| \leq C \tau 
  \ \ \text{and}\ \  \{ \xi_n + s_\pm, f_\pm\} \geq C' \lambda, $$
  with $\lambda^2 = \tau^2 + |\xi'|^2$.
\end{lem}
See Appendix~\ref{proof: lemma: effective sub-ellipticiy condition}
for a proof.  This is the form of the sub-ellipticty condition, with
loss of half derivative, that we shall use. This will be further
highlighted by the estimates we derive in Section~\ref{sec: estimate
  1st order factor} and by the proof of the main theorem.

% Estimates for first-order factors
\section{Estimates for first-order factors}
\label{sec: estimate 1st order factor}

Unless otherwise specified, the notation $\norm{\cdot}$ will stand for
the $L^2(\R^n)$-norm and $\val{\cdot}$ for the $L^2(\R^{n-1})$-norm.
The $L^2(\R^n)$ and $L^2(\R^{n-1})$ dot-products will be both denoted
by $\poscal{\cdot}{\cdot}$.

\subsection{Preliminary estimates}
\label{sec: prelim cut-off}

Most of our pseudo-differential arguments concern a
calculus with the large parameter $\tau\ge 1$:  with 
\begin{equation}\label{2.lam}
\lambda^2 = \tau^2 + |\xi'|^2,
\end{equation}
we define for $m\in \R$ the class of tangential symbols
$\ST^m$ as the smooth functions on $\R^n\times\R^{n-1}$,
depending on the parameter $\tau\ge 1$, such that, for all $(\alpha,
\beta)\in \N^n\times \N^{n-1}$,
\begin{equation}
  \sup_{(x,\xi')\in \R^n\times \R^{n-1}}
  \lambda^{-m+\val\beta} \val{(\p_{x}^\alpha\p_{\xi'}^\beta a)(x,\xi',\tau)}<\io.
\end{equation}
Some basic properties of the calculus of the associated
pseudo-differential operators are recalled in Appendix~\ref{section:
semi-classical calculus}. We shall refer to this calculus as to the
semi-classical calculus (with a large parameter).  In particular we
introduce the following Sobolev norms:
\begin{equation}\label{sobolev}
   \norm{u}_{\H^s} := \norm{\Lambda^s
   u}_{L^2(\R^{n-1})}, \quad \text{with}\ \Lambda^s := \op{\lambda^s}.
\end{equation}
For $s
\geq 0$ note that we have
  $\norm{u}_{\H^s} \sim \tau^s \norm{u}_{L^2(\R^{n-1})}
  + \norm{\valjp{D'}^s u}_{L^2(\R^{n-1})}$.

The operator $M_\pm$ is of pseudo-differential nature in the standard
calculus.  Observe however that in any region where $\tau \gtrsim
|\xi'|$ the symbol $m_{\pm}$ does not satisfy the estimates of
$\ST^1$.  We shall circumvent this technical point by introducing a 
 cut-off procedure.

Let $C_0, C_1>0$ be such that $\varphi' \geq C_0$ and 
\begin{align} 
  \label{eq: bornitude M}
  (M_\pm u, H^+u) \leq C_1 \norm{H^+ u}_{L^2(\R;H^\hf(\R^{n-1}))}^2.
\end{align}
We choose $\psi \in \moo(\R^+)$ nonnegative such that $\psi = 0$
in $[0, 1]$ and $\psi = 1$ in $[2, + \infty)$.  We
introduce the following Fourier multiplier
\begin{align}
  \label{eq: Fourier multiplier}
  \psi_\epsilon(\tau,\xi') = \psi \Big( \frac{\epsilon  \tau}{\valjp{\xi}}\Big) \in \ST^0, \qquad \text{with} \ 0 < \epsilon \leq \epsilon_0.
\end{align} 
such that $\tau \gtrsim \valjp{\xi'}/\epsilon$ in its support.
We choose $\epsilon_0$ \suff small so that $\supp(\psi_\epsilon)$ is disjoint
from a conic \nhd (for $|\xi'|\geq 1$) of the sets $\{f_\pm=0\}$ (see
Figure~\ref{fig: prelim cut-off}).
\begin{figure}
\begin{center}
   \begin{picture}(0,0)%
\includegraphics{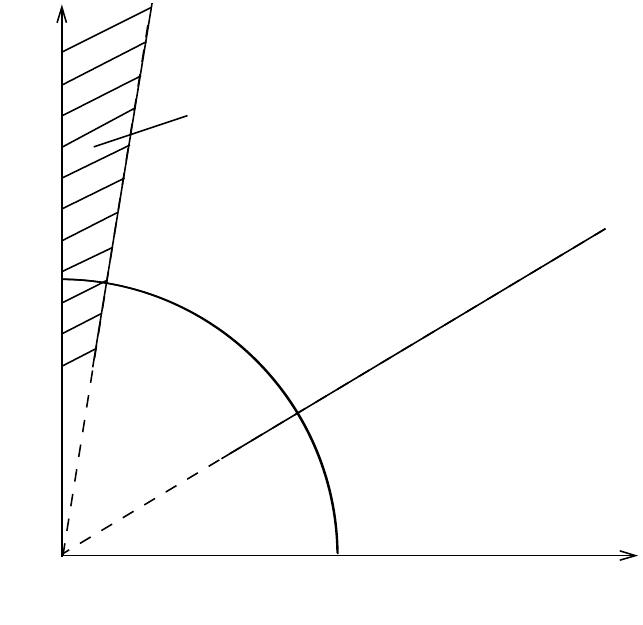}%
\end{picture}%
\setlength{\unitlength}{3947sp}%
\begingroup\makeatletter\ifx\SetFigFont\undefined%
\gdef\SetFigFont#1#2#3#4#5{%
  \reset@font\fontsize{#1}{#2pt}%
  \fontfamily{#3}\fontseries{#4}\fontshape{#5}%
  \selectfont}%
\fi\endgroup%
\begin{picture}(3076,2997)(1051,-2956)
\put(2637,-2836){\makebox(0,0)[lb]{\smash{{\SetFigFont{10}{12.0}{\rmdefault}{\mddefault}{\updefault}{\color[rgb]{0,0,0}$1$}%
}}}}
\put(1051,-211){\makebox(0,0)[lb]{\smash{{\SetFigFont{10}{12.0}{\rmdefault}{\mddefault}{\updefault}{\color[rgb]{0,0,0}$\tau$}%
}}}}
\put(3826,-2911){\makebox(0,0)[lb]{\smash{{\SetFigFont{10}{12.0}{\rmdefault}{\mddefault}{\updefault}{\color[rgb]{0,0,0}$|\xi'|$}%
}}}}
\put(3151,-1786){\makebox(0,0)[lb]{\smash{{\SetFigFont{10}{12.0}{\rmdefault}{\mddefault}{\updefault}{\color[rgb]{0,0,0}$f_\pm =0$}%
}}}}
\put(2026,-511){\makebox(0,0)[lb]{\smash{{\SetFigFont{10}{12.0}{\rmdefault}{\mddefault}{\updefault}{\color[rgb]{0,0,0}$\supp(\psi_\epsilon)$}%
}}}}
\end{picture}%
\end{center}
\caption{Relative positions of $\supp(\psi_\epsilon)$ and the sets $\{ f_\pm =0\}$.}
\label{fig: prelim cut-off}
\end{figure}
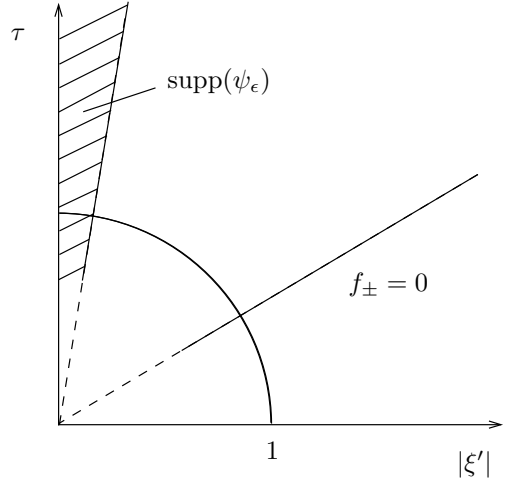

The following lemma states that we can obtain very natural estimates
on both sides of the interface in the region $|\xi'| \ll \tau$, \ie for $\epsilon$ small. We
refer to Appendix~\ref{proof: estimates for large tau} for a proof.
%%%%%%%%%%%%%%%%%%%%%%%%
% lemma                %
%%%%%%%%%%%%%%%%%%%%%%%%
\begin{lem}
  \label{lemma: estimates for large tau}
   Let $\ell \in \R$. There
  exist $\tau_1\geq 1$, $0< \epsilon_1 \leq \epsilon_0$ and $C>0$ such
  that \begin{align*}
  &C \norm{H_+ \A_+ \op{\psi_\epsilon} \omega}_{L^2(\R;\H^\ell)} 
  \geq \val{\op{\psi_\epsilon} \omega\br}_{\H^{\ell+\hf}}
  + \norm{H_+ \op{\psi_\epsilon} \omega}_{L^2(\R;\H^{\ell+1})} ,\\
  &C\Big(\norm{H_- \A_- \op{\psi_\epsilon} \omega}_{L^2(\R;\H^\ell)}
  + \val{\op{\psi_\epsilon} \omega\bl}_{\H^{\ell+\hf}}\Big)
   \geq \norm{H_- \op{\psi_\epsilon} \omega}_{L^2(\R;\H^{\ell+1})}
  , \end{align*} for $0<\epsilon \leq \epsilon_1$, with $A_+=\P_{E+}$
  or $\P_{F+}$, $\A_- = \P_{E-}$ or $\P_{F-}$, for $\tau \geq \tau_1$
  and $\omega \in \mooc(\R^n)$.
\end{lem}

\subsection{Positive imaginary part on a half-line}
We have the following estimates for the operators $\P_{E+}$ and
$\P_{E-}$.
%%%%%%%%%%%%%%%%%%%%%%%%
% lemma                %
%%%%%%%%%%%%%%%%%%%%%%%%
\begin{lem}
  \label{lemma: estimate PE+ PE-}
  Let $\ell \in \R$. There exist $\tau_1\geq 1$ and $C>0$ such that 
  \begin{align}
    \label{eq: estimate PE+}
    C\norm{H_+ \P_{E_+} \omega}_{L^2(\R;\H^\ell)} 
    \geq \val{\omega_{|x_n=0^+}}_{\H^{\ell+\hf}} 
    + \norm{H_+ \omega}_{L^2(\R;\H^{\ell+1})}
    + \norm{H_+D_n \omega}_{L^2(\R;\H^{\ell})},
  \end{align}
  and 
  \begin{multline}
    \label{eq: estimate PE-}
    C \Big(\norm{H_- \P_{E_-} \omega}_{L^2(\R;\H^\ell)}  
    + \val{\omega_{|x_n=0^-}}_{\H^{\ell+\hf}}\Big) \\
    \geq  \norm{H_- \omega}_{L^2(\R;\H^{\ell+1})}
    + \norm{H_+D_n \omega}_{L^2(\R;\H^{\ell})},
  \end{multline}
  for $\tau \geq \tau_1$ and $\omega \in\mooc(\R^n)$.
\end{lem}
Note that the first estimate, in $\R_+$, is of very good quality as
both the trace and the volume norms are dominated: we have a perfect
elliptic estimate. In $\R_-$, we obtain an estimate of lesser quality.
Observe also that no assumption on the weight function,
apart from the positivity of $\varphi'$, is used in the proof below.
%%%% proof of lemma
\begin{proof}
Let $\psi_\epsilon$ be defined as in Section~\ref{sec: prelim
cut-off}.  We let $\tilde{\psi} \in \moo(\R^+)$ be nonnegative and
such that $\tilde{\psi} = 1$ in $[4, + \infty)$ and $\tilde{\psi} = 0$
in $[0,3]$. We then define $\tilde{\psi}_\epsilon$ according
to \eqref{eq: Fourier multiplier} and we have $\tau \lesssim \valjp{\xi'}$
in $\supp(1 - \tilde{\psi}_\epsilon)$ and
$\supp(1- \psi_\epsilon) \cap
\supp(\tilde{\psi}_\epsilon) = \emptyset$. We set $\tilde{m}_\pm = m_\pm (1- \tilde{\psi}_\epsilon)$ and observe that $\tilde{m}_\pm \in
\ST^1$. We define
\begin{align*}
  \tilde{e}_\pm = \tau \varphi' + \tilde{m}_\pm \in \ST^1, \quad 
  \tilde{E}_\pm = \opw{\tilde{e}_\pm}, \quad 
\end{align*}       
Observe that from the definition of $\tilde{\psi}_\epsilon$ we have
\begin{align}
  \label{eq: positivity e bar}
  \tilde{e}_\pm \geq C \lambda.
\end{align}
Next, we note that 
\begin{align*}
  M_\pm \op{1-\psi_\epsilon} \omega 
  &= \opw{\tilde{m}_\pm} \op{1-\psi_\epsilon} \omega  
  + \opw{m_\pm \tilde{\psi}_\epsilon}  \op{1-\psi_\epsilon} \omega, 
\end{align*}
and, since $m_\pm \tilde{\psi}_\epsilon \in \S^1$ and $1-\psi_\epsilon \in \ST^0$,
with the latter vanishing in a region $\valjp{\xi'} \leq C \tau$,
Lemma~\ref{lemma: mixed composition} yields
\begin{align}
   \label{eq: microlocal change for M}
  M_\pm \op{1-\psi_\epsilon} \omega 
  &= \opw{\tilde{m}_\pm} \op{1-\psi_\epsilon} \omega  
  + R_1 \omega, \quad \text{with} \ \ R_1 \in \op{\ST^{-\infty}}.  
\end{align}
%Similarly we find
%\begin{align}
%   \label{eq: microlocal change for M - 2}
%  \op{1-\psi_\epsilon} M_\pm  \omega 
%  &= \op{1-\psi_\epsilon} \opw{\tilde{m}_\pm}  \omega  
%  + R_2 \omega, \quad \text{with} \ \ R_2 \in \op{\ST^{-\infty}}.  
%\end{align}

  We set
  $u = \op{1-\psi_\epsilon} \omega$.  For $s=2 \ell+1$, we compute, 
  \begin{align}
    \label{eq: start lemma 3.1}
    2 \re \poscal{\P_{E_+} u}{i  H_+ \Lambda^s u} 
    &= \poscal{i [D_n, H_+] u}{\Lambda^s u} 
    + \poscal{i [S_+, \Lambda^s] u}{H_+ u}
    + 2 \re \poscal{E_+ u}{H_+\Lambda^s u}\\
    &\geq  \val{u\br}_{\H^{\ell+\hf}}^2 
    + 2 \re \poscal{E_+ u}{H_+ \Lambda^s u}
    - C \norm{H_+u}_{L^2(\R; \H^{\ell+\hf})}^2.
    \nonumber
  \end{align}
  By~\eqref{eq: microlocal change for M} we have $E_+ u = \tilde{E}_+ u +
  R_1 \omega$.
  This yields 
  \begin{align*}
    \re \poscal{E_+ u}{H_+\Lambda^s u} 
    + \norm{H_+ \omega}^2
    \gtrsim \re \poscal{\tilde{E}_+ u}{H_+\Lambda^s u}
    \gtrsim \norm{H_+u}_{L^2(\R; \H^{\ell+1})}^2,
  \end{align*}
 
 for $\tau$ \suff large by \eqref{eq: positivity e bar} and
  Lemma~\ref{lem: postivity garding tau}. 
   We thus  obtain 
  \begin{multline*}
    \re \poscal{\P_{E_+} u}{i  H_+ \Lambda^s u} 
    + \norm{H_+u}_{L^2(\R; \H^{\ell+\hf})}^2
    + \norm{H_+ \omega}^2\\
    \gtrsim 
     \val{u\br}_{\H^{\ell+\hf}}^2 
    + \norm{H_+u}_{L^2(\R; \H^{\ell+1})}^2,
  \end{multline*}
 With the Young
  inequality and taking $\tau$ \suff large we then find
   \begin{align*}
     \norm{H_+ \P_{E_+} u}_{L^2(\R;\H^\ell)}  
     +\norm{H_+ \omega}
     \gtrsim \val{u\br}_{\H^{\ell+\hf}} 
     + \norm{H_+u}_{L^2(\R; \H^{\ell+1})}.
   \end{align*}
  We now invoke the corresponding estimate provided by
  Lemma~\ref{lemma: estimates for large tau},
  \begin{align*}
    \norm{H_+ \P_{E+} \op{\psi_\epsilon} \omega}_{L^2(\R;\H^\ell)}
    \gtrsim \val{\op{\psi_\epsilon} \omega\br}_{\H^{\ell+\hf}}  
    + \norm{H_+ \op{\psi_\epsilon} \omega}_{L^2(\R; \H^{\ell+1})}.
  \end{align*}
  Adding the two estimates, with the triangular inequality, we obtain
  \begin{multline*}
    \norm{H_+ {\P}_{E_+}  \op{1-\psi_\epsilon} \omega}_{L^2(\R;\H^\ell)}
    + \norm{H_+ \P_{E+} \omega}_{L^2(\R;\H^\ell)}
    + \norm{H_+ \omega}\\
    \gtrsim \val{\omega\br}_{\H^{\ell+\hf}}
    + \norm{H_+ \omega}_{L^2(\R; \H^{\ell+1})}.
  \end{multline*}
  Lemma~\ref{lemma: mixed composition} gives $\big[\P_{E_+},
  \op{1-\psi_\epsilon}  \big] \in \op{\ST^0}$.
  We thus have 
  \begin{align*}
    \norm{H_+ \P_{E_+} \op{1-\psi_\epsilon} \omega}_{L^2(\R;\H^\ell)} 
    &\lesssim \norm{H_+ \op{1-\psi_\epsilon} \P_{E_+}  \omega}_{L^2(\R;\H^\ell)} 
    +  \norm{H_+  \omega}_{L^2(\R;\H^\ell)}\\
    &\lesssim \norm{H_+  \P_{E_+}  \omega}_{L^2(\R;\H^\ell)} 
    +  \norm{H_+  \omega}_{L^2(\R;\H^\ell)}.
  \end{align*}
  By taking $\tau$ \suff large, we thus obtain
  \begin{align}
    \norm{H_+ \P_{E_+} \omega}_{L^2(\R;\H^\ell)} 
    \gtrsim \val{\omega_{|x_n=0^+}}_{\H^{\ell+\hf}} 
    + \norm{H_+ \omega}_{L^2(\R;\H^{\ell+1})}.
  \end{align}
  The term $\norm{H_+D_n \omega}_{L^2(\R;\H^{\ell})}$ can simply be introduced
  on the \rhs of this estimates, to yield \eqref{eq: estimate PE+},
  thanks to the form of the first-order operator $\P_{E_+}$.  To
  obtain estimate~\eqref{eq: estimate PE-} we compute $2 \re
  \poscal{\P_{E_-} \omega}{i H_-\omega}$. The argument is similar
  whereas the trace term comes out with the opposite sign.
\end{proof}

For the operator $\P_{F+}$ we can also obtain a microlocal estimate.
We place ourselves in a microlocal region where $f_+ = \tau \varphi^+
- m_+$ is positive. More precisely, let $\chi(x,\tau,\xi') \in
\ST^0$ be such that $|\xi'| \leq C
\tau$ and $f_+ \geq C_1 \lambda$ in $\supp(\chi)$, $C_1>0$, and
$|\xi'| \geq C' \tau$ in $\supp(1- \chi)$.
%%%%%%%%%%%%%%%%%%%%%%%%
% lemma                %
%%%%%%%%%%%%%%%%%%%%%%%
\begin{lem}
  \label{lemma: estimate PF+ F+ positive}
  Let $\ell \in \R$. There exist $\tau_1\geq 1$ and $C>0$  such that
  \begin{multline*}
    C \Big(\norm{H_+ \P_{F_+} \opw{\chi} \omega}_{L^2(\R;\H^\ell)} 
    + \norm{H_+ \omega}\Big)\\
    \geq  \val{\opw{\chi} \omega_{|x_n=0^+}}_{\H^{\ell+\hf}} 
    + \norm{H_+ \opw{\chi} \omega}_{L^2(\R;\H^{\ell+1})}
    + \norm{H_+ D_n \opw{\chi} \omega}_{L^2(\R;\H^{\ell})},
  \end{multline*}
  for $\tau \geq \tau_1$ and $\omega \in\mooc(\R^n)$.
\end{lem}
As for \eqref{eq: estimate PE+} of Lemma~\ref{lemma: estimate PE+
  PE-}, up to an harmless remainder term, we obtain an elliptic
estimate in this microlocal region.
%% proof of lemma
\begin{proof}
  Let $\psi_\epsilon$ be as defined in Section~\ref{sec: prelim
  cut-off} and let $\tilde{\psi}_\epsilon$ be as in the proof of
  Lemma~\ref{lemma: estimate PE+ PE-}.
  We set 
  \begin{align}
  \label{eq: def tilde f}
    \tilde{f}_\pm = \tau \varphi' - \tilde{m}_\pm \in \ST^1, \quad 
    \tilde{F}_\pm = \opw{\tilde{f}_\pm}.
  \end{align}
  Observe that we have
  \begin{align*}
    \tilde{f}_\pm = \tau \varphi' - \tilde{m}_\pm
    =  \tau \varphi' - m_\pm (1-\tilde{\psi}_\epsilon)
    = f_\pm + \tilde{\psi}_\epsilon m_\pm \geq f_\pm.
  \end{align*}
  This gives  $\tilde{f}_+ \geq C \lambda$ in $\supp(\chi)$.

  We set $u = \op{1-\psi_\epsilon} \opw{\chi} \omega$.  Following the
  proof of Lemma~\ref{lemma: estimate PE+ PE-}, for $s=2 \ell+1$, we
  obtain 
  \begin{multline*} 
  \re \poscal{\P_{F+} u}{i H_+\Lambda^s u}
  + \norm{H_+ \omega}^2 +  \norm{H_+u}_{L^2(\R; \H^{\ell+\hf})}^2\\
  \gtrsim \val{u\br}_{\H^{\ell+\hf}}^2
  + \re \poscal{\tilde{F}_+ u}{H_+\Lambda^s u} 
  \end{multline*} 
  Let now
  $\tilde{\chi} \in \ST^0$ satisfy the same properties as $\chi$,
  with moreover $\tilde{\chi}=1$ on a \nhd of $\supp(\chi)$. We then
  write
  $$
  \tilde{f}_+ = \check{f}_+ + r, \quad 
  \text{with}\ \ \check{f}_+= \tilde{f}_+\tilde{\chi} 
  +\lambda(1-\tilde{\chi}) \in \ST^1, 
  \quad r = (\tilde{f}_+- \lambda)(1-\tilde{\chi}) \in \ST^1.
  $$
  As $\supp(1-\tilde{\chi}) \cap \supp(\chi) =\emptyset$, we find
  $r \sharp (1- \psi_\epsilon) \sharp \chi \in \ST^{-\infty}$. Since
  $\check{f}_+ \geq C \lambda$ by construction, with
  Lemma~\ref{lem: postivity garding tau} we
  obtain 
  \begin{multline*} 
  \re \poscal{\P_{F+} u}{i H_+\Lambda^s u}
  + \norm{H_+ \omega}^2 
  +  \norm{H_+u}_{L^2(\R; \H^{\ell+\hf})}^2\\
  \gtrsim \val{u\br}_{\H^{\ell+\hf}}^2
  + \norm{H_+ u}_{L^2(\R;\H^{\ell+1})}^2.  
  \end{multline*} 
  With the Young inequality, 
  taking $\tau$ \suff large, we obtain 
  \begin{align*} \norm{H_+ \P_{F+}u}_{L^2(\R;\H^{\ell})}
  + \norm{H_+ \omega} 
  \gtrsim \val{u\br}_{\H^{\ell+\hf}}
  + \norm{H_+ u}_{L^2(\R;\H^{\ell+1})}.  
  \end{align*} 
  
  Invoking the corresponding estimate provided by Lemma~\ref{lemma:
  estimates for large tau} for
  $\opw{\chi} \omega$,
  \begin{align*} 
  \norm{H_+ \P_{F+} \op{\psi_\epsilon} \opw{\chi} \omega}_{L^2(\R;\H^{\ell})}
  &\gtrsim \val{\op{\psi_\epsilon} \opw{\chi} \omega\br}_{\H^{\ell+\hf}}\\
  &\quad+ \norm{H_+ \op{\psi_\epsilon} \opw{\chi} \omega}_{L^2(\R;\H^{\ell+1})},
   \end{align*}
  and arguing as in the end of Lemma~\ref{lemma: estimate PE+ PE-} we
  obtain the result.
\end{proof}

For the operator $\P_{F-}$ we can also obtain a microlocal estimate.
We place ourselves in a microlocal region where $f_- = \tau \varphi^-
- m_-$ is positive. More precisely, let $\chi(x,\tau,\xi') \in
\ST^0$ be such that $|\xi'| \leq C
\tau$ and $f_- \geq C_1 \lambda$ in $\supp(\chi)$, $C_1>0$, and
$|\xi'| \geq C' \tau$ in $\supp(1- \chi)$.  We have the following
lemma whose form is adapted to our needs in Section~\ref{sec:
convexification} where we allow some dependency upon the variable $x'$
for the weight function.
%%%%%%%%%%%%%%%%%%%%%%%%
% lemma                %
%%%%%%%%%%%%%%%%%%%%%%%%
\begin{lem}
  \label{lemma: estimate PF- F- positive}
  Let $\ell \in \R$.
  There exist $\tau_1\geq 1$ and $C>0$  such that
  \begin{multline}
    \label{eq: estimate PF- F- positive}
    C \Big( \norm{H_- \P_{F_-}u}_{L^2(\R;\H^\ell)} + \norm{H_- \omega} 
    +  \norm{H_- D_n \omega} +  \val{u_{|x_n=0^-}}_{\H^{\ell+\hf}} \Big)\\
    \geq
     \norm{H_- u}_{L^2(\R; \H^{\ell+1})},
  \end{multline}
  for $\tau \geq \tau_1$ and 
  $u = a_{n n}^- \P_{E-} \opw{\chi} \omega$ with $\omega \in\mooc(\R^n)$.
\end{lem}
%%%% proof of lemma
\begin{proof}
   Let $\psi_\epsilon$ be defined as in Section~\ref{sec: prelim
cut-off}. We define $ \tilde{f}_-$ and $\tilde{F}_-$ as in \eqref{eq:
def tilde f}. We have $\tilde{f}_- \geq f_- \geq C \lambda$ in
$\supp(\chi)$.  We set $z = \op{1-\psi_\epsilon} u$ and for $s =
2\ell+1$ we compute 
  \begin{align*}
    &2 \re \poscal{\P_{F_-} z}
    {i H_- \Lambda^s z}\\ 
    &\qquad = \poscal{i [D_n, H_-] z}{\Lambda^s z}
    + i \poscal{[S_-, \Lambda^s] z}{H_- z}
    + 2 \re \poscal{F_- z}{ H_- \Lambda^s z}\\
    &\qquad\geq 
    - \val{z\bl}_{\H^{\ell+\hf}}^2 
    + 2 \re \poscal{F_- z} {H_- \Lambda^s z}
    - C \norm{H_- z}_{L^2(\R; \H^{\ell+\hf})}^2.
  \end{align*}
  Arguing as in the proof of Lemma~\ref{lemma: estimate PE+ PE-} (see  \eqref{eq: microlocal change for M} and \eqref{eq: start lemma 3.1}) we obtain 
  \begin{multline*}
    2 \re \poscal{\P_{F_-} z}
    {i H_- \Lambda^s z} + C \norm{H_- u}^2 + \val{z\bl}_{\H^{\ell+ \hf}}^2 
    + C \norm{H_- z}_{L^2(\R; \H^{\ell + \hf})}^2\\
    \geq 
    2 \re \poscal{\tilde{F}_- z} {H_- \Lambda^s z}.
  \end{multline*}
  Let now $\tilde{\chi} \in \ST^0$ satisfy the same properties as
  $\chi$, with moreover $\tilde{\chi}=1$ on a \nhd of $\supp(\chi)$. We then
  write
  $$
  \tilde{f}_- = \check{f}_- + r, \quad 
  \text{with}\ \ \check{f}_- = \tilde{f}_- \tilde{\chi} +\lambda(1-\tilde{\chi}) \in \ST^1, 
  \quad r = (\tilde{f}_- - \lambda)(1-\tilde{\chi}) \in \ST^1.
  $$
  As $\check{f}_- \geq C \lambda$ and $\supp(1-\tilde{\chi}) \cap \supp (\chi) = \emptyset$
  with Lemma~\ref{lem: postivity garding tau}
  we obtain, for $\tau$ large,  
  \begin{multline*}
    2 \re \poscal{\P_{F_-} z}
    {i H_- \Lambda^s z} + C \norm{H_- u}^2 + \val{z\bl}_{\H^{\ell+\hf}}^2 
    + C \norm{H_- z}_{L^2(\R; \H^{\ell + \hf})}^2\\
    + \norm{H_- \omega}^2 + \norm{H_- D_n \omega}^2 \geq 
    C' \norm{H_-z}_{L^2(\R; \H^{\ell+1})}^2.
  \end{multline*}
  With the Young
  inequality and taking $\tau$ \suff large we then find
  \begin{multline*}
    \norm{H_- \P_{F_-} z}_{L^2(\R; \H^\ell)}
    + \norm{H_- u} + \val{z\bl}_{\H^{\ell+ \hf}} 
    + \norm{H_- \omega} + \norm{H_- D_n \omega}\\ \gtrsim
    \norm{H_-z}_{L^2(\R; \H^{\ell+1})}.
  \end{multline*}
  Invoking the corresponding estimate provided by Lemma~\ref{lemma:
  estimates for large tau} for $u$ yields
  \begin{align*}
    \norm{H_- \P_{F_-} \op{\psi_\epsilon}u}_{L^2(\R; \H^\ell)}
    + \val{\op{\psi_\epsilon}u\bl}_{\H^{\ell+\hf}} 
    \gtrsim
    \norm{H_-\op{\psi_\epsilon}u}_{L^2(\R; \H^{\ell+1})}.
  \end{align*}
  and arguing as in the end of Lemma~\ref{lemma: estimate PE+ PE-} we
  obtain the result. 
\end{proof}

\subsection{Negative  imaginary part on the negative half-line}
\label{3.sec.k5jhbg}
Here we place ourselves in a microlocal region where $f_- = \tau
\varphi^- - m_-$ is negative. More precisely, let $\chi(x,\tau,\xi') \in
\ST^0$ be such that $|\xi'| \geq C
\tau$ and $f_- \leq - C_1 \lambda$ in $\supp(\chi)$, $C_1>0$.
We have the following lemma whose form is adapted to our needs in the
next section. Up to harmless remainder terms, this can also be
considered as a good elliptic estimate.
%%%%%%%%%%%%%%%%%%%%%%%%
% lemma                %
%%%%%%%%%%%%%%%%%%%%%%%%
\begin{lem}
  \label{lemma: estimate PF- F- negative}
  There exist $\tau_1\geq 1$ and $C>0$  such that
  \begin{align}
    \label{eq: estimate PF-}
    C \Big( \norm{H_- \P_{F_-}u} + \norm{H_- \omega} 
    +  \norm{H_- D_n \omega}\Big)
    \geq \val{u_{|x_n=0^-}}_{\H^\hf} 
    + \norm{H_- u}_{L^2(\R; \H^1)},
  \end{align}
  for $\tau \geq \tau_1$ and 
  $u = a_{n n}^- \P_{E-} \opw{\chi} \omega$ with $\omega \in\mooc(\R^n)$.
\end{lem}
%%%% proof of lemma
\begin{proof}
  We compute  
  \begin{align*}
    &2 \re \poscal{\P_{F_-} u}
    {-i H_- \Lambda^1 u}\\ 
    &\qquad = \poscal{i [D_n, -H_-] u}{\Lambda^1 u}
    - i \poscal{[S_-, \Lambda^1] u}{H_- u}
    + 2 \re \poscal{-F_- u}{ H_- \Lambda^1 u}\\
    &\qquad\geq 
    \val{u\bl}_{\H^\hf}^2 
    + 2 \re \poscal{-F_- u} {H_- \Lambda^1 u}
    - C \norm{H_- u}_{L^2(\R; \H^\hf)}^2.
  \end{align*}
  Let now $\tilde{\chi} \in \ST^0$ satisfy the same properties as
  $\chi$, with moreover $\tilde{\chi}=1$ on a \nhd of $\supp(\chi)$. We then
  write
  $$
  f_- = \check{f}_- + r, \quad 
  \text{with}\ \ \check{f}_- = f_- \tilde{\chi} -\lambda(1-\tilde{\chi}), 
  \quad r = (f_- + \lambda)(1-\tilde{\chi}).
  $$
  Observe that $f_- \tilde{\chi} \in \ST^1$ because of the
  support of $\tilde{\chi}$. Hence $\check{f}_- \in \ST^1$. As $-\check{f}_- \geq C
  \lambda$ with Lemma~\ref{lem: postivity garding tau}
  we obtain, for $\tau$ large,  
  $\re \poscal{-\opw{\check{f}_-} u}
    {H_- \Lambda^1 u}
    \gtrsim \norm{H_- u}_{L^2(\R;\H^1)}^2$.
  Note that $r$ does not satisfy the estimates of the semi-classical 
  calculus because of the term $m_- (1-\tilde{\chi})$.
  However, we have 
  \begin{align*}
    \opw{r} u = \opw{r} a_{n n}^- \opw{\chi} D_n \omega 
    + i \opw{r} a_{n n}^- E_- \opw{\chi}\omega.
  \end{align*}
  Applying Lemma~\ref{lemma: mixed composition}, 
  using that $1-\tilde{\chi} \in \ST^0 \subset \S^0$, yields
  \begin{equation*}
  \opw{r} u = R \omega \qquad  \text{with} \ \ 
  R \in \op{\ST^{2}} D_n + \op{\ST^{2}}.
  \end{equation*}
    As $\supp(1-\tilde{\chi}) \cap \supp(\chi) =\emptyset$, the
  composition formula \eqref{eq: composition asymptotic series} (which is valid
  in this case -- see Lemma~\ref{lemma: mixed composition}) yields
  moreover $R \in \op{\ST^{-\infty}} D_n + \op{\ST^{-\infty}}$.  We
  thus find, for $\tau$ \suff large
  \begin{align*}
    \re \poscal{\P_{F_-} u}{-i H_- \Lambda^1 u} 
    + \norm{H_- \omega}^2 + \norm{H_- D_n \omega}^2
    \gtrsim \val{u\bl}_{\H^\hf}^2
    + \norm{H_- u}_{L^2(\R;\H^1)}^2,
  \end{align*}
  and we conclude with the Young inequality.
\end{proof}

\subsection{Increasing imaginary part on a  half-line}
Here we allow the symbols $f_\pm$ to change sign.  For the first-order
factor $\P_{F_\pm}$ this will lead to an estimate that 
exhibits a loss of a half derivative as can be expected.

Let $\psi_\epsilon$ be as defined in Section~\ref{sec: prelim
  cut-off} and let $\tilde{\psi}_\epsilon$ be as in the proof of
  Lemma~\ref{lemma: estimate PE+ PE-}. We define $ \tilde{f}_\pm$ and
  $\tilde{F}_\pm$ as in \eqref{eq: def tilde f} and set $\tilde{\P}_{F_\pm} = D_n + S_\pm + i \tilde{F}_\pm$.   

As $\supp(\tilde{\psi}_\epsilon)$
remains away from the sets $\{ f_\pm =0\}$ the sub-ellipticy property
of Lemma~\ref{lemma: effective sub-ellipticiy condition} is preserved
for $\tilde{f}_\pm $ in place of $f_\pm$. We shall use the following
inequality. 
%%%%%%%%%%%%%%%%%%%%%%%%
% lemma                %
%%%%%%%%%%%%%%%%%%%%%%%%
\begin{lem}
  \label{lemma: pre-garding ineqality}
  There exist $C>0$  such that for $\mu>0$ \suff large we have 
  \begin{align*}
    \rho_\pm = \mu \tilde{f}_\pm^2 + \tau \poi{\xi_n+s_\pm}{\tilde{f}_\pm} 
    \geq C \lambda^2,
  \end{align*}
  with $\lambda^2 = \tau^2 + |\xi'|^2$. 
\end{lem}
%%%% proof of lemma
\begin{proof}
  If $|\tilde{f}_\pm| \leq \delta \lambda$, for $\delta$ small, then
  $\tilde{f}_\pm = f_\pm$ and $\tau \poi{\xi_n+s_\pm}{\tilde{f}_\pm} \geq C
  \lambda^2$ by Lemma~\ref{lemma: effective sub-ellipticiy condition}.

  If $|\tilde{f}_\pm| \geq \delta \lambda$, observing that $\tau
  \poi{\xi_n+s_\pm}{\tilde{f}_\pm} \in \tau \ST^1 \subset
  \ST^2$, we obtain $\rho_\pm \geq C \lambda^2$ by choosing
  $\mu$ sufficiently large.
\end{proof}

We now prove the following estimate for $\P_{F_\pm}$.
%%%%%%%%%%%%%%%%%%%%%%%%
% lemma                %
%%%%%%%%%%%%%%%%%%%%%%%
\begin{lem}
  \label{lemma: increasing root}
  Let $\ell \in \R$.
  There exist $\tau_1\geq 1$ and $C>0$ such that
  \begin{multline*}
    %\label{eq: estimate increasing root}
    C \Big(\norm{H_\pm \P_{F_\pm} \omega}_{L^2(\R; \H^{\ell})} 
    + \val{\omega_{|x_n=0^\pm}}_{\H^{\ell+\hf}}\Big) \\
    \geq  
    \tau^{-\hf} \Big(\norm{H_\pm \omega}_{L^2(\R; \H^{\ell+1})}
    + \norm{H_\pm D_n \omega}_{L^2(\R; \H^{\ell})}\Big),
  \end{multline*}
  for $\tau \geq \tau_1$ and $\omega \in\mooc(\R^n)$.
\end{lem}
%%%%%% proof of lemma
\begin{proof}
 we set
  $u = \op{1-\psi_\epsilon} \omega$.  We start by invoking \eqref{eq:
    microlocal change for M}, and the fact that $[\tilde{\P}_{F+},\Lambda^\ell] \in \op{\ST^\ell}$, and write 
  \begin{align}
    \label{eq: u -> uell}
    \norm{H_+  \tilde{\P}_{F+} \Lambda^\ell u}
    &\lesssim
    \norm{H_+ \Lambda^\ell \tilde{\P}_{F+} u}
    + \norm{H_+ [\tilde{\P}_{F+},\Lambda^\ell] u}  
    \\
    &\lesssim
    \norm{H_+ \tilde{\P}_{F+} u}_{L^2(\R; \H^\ell)}
    + \norm{H_+ u}_{L^2(\R; \H^\ell)}\nonumber\\
    &\lesssim
    \norm{H_+ \P_{F+} u}_{L^2(\R; \H^\ell)}
    +  \norm{H_+ \omega}_{}
    +\norm{H_+  u}_{L^2(\R; \H^\ell)}\nonumber
  \end{align}  
  We set $u_\ell = \Lambda^\ell u$. We then have
  \begin{align*}
    \norm{H_+ \tilde{\P}_{F+} u_\ell}_{}^2
    &= \norm{H_+ (D_n + S_+) u_\ell}_{}^2 
    + \norm{H_+ \tilde{F}_{+} u_\ell}_{}^2
    + 2 \re \poscal{(D_n + S_+) u_\ell}{i  H_+ \tilde{F}_{+} u_\ell}\\
    &\geq \tau^{-1} 
    \re \poscal{ \big(\mu \tilde{F}_+^2 + i \tau \big[D_n + S_+,\tilde{F}_+ \big]\big)
      u_\ell}{H_+ u_\ell} 
    + \poscal{ i [D_n, H_+] u_\ell}{\tilde{F}_+  u_\ell},
  \end{align*}
  with $\mu \tau^{-1}\leq 1$. As the principal symbol (in the
  semi-classical calculus) of $\mu \tilde{F}_+^2 + i \tau \big[D_n +
  S_+,\tilde{F}_+ \big]$ is $\rho_+ = \mu \tilde{f}_+^2 + \tau \poi{ \xi_n
    + s_+}{\tilde{f}_+}$, Lemmas~\ref{lemma: pre-garding ineqality} and
  \ref{lem: postivity garding tau} yield
  \begin{align*}
    \norm{H_+ \tilde{\P}_{F+} u_\ell}_{}^2 
    + \val{u_\ell}_{\H^\hf}^2
    \gtrsim \tau^{-1} \norm{H_+ u_\ell}_{L^2(\R; \H^1)}^2,
  \end{align*}
  for $\mu$ large, \ie, $\tau$ large.
  With \eqref{eq: u -> uell} we obtain, for $\tau$ \suff large,
  \begin{align*}
    \norm{H_+ \P_{F+} u}_{L^2(\R; \H^\ell)} +  \norm{H_+  \omega}_{}
    + \val{u}_{\H^{\ell+\hf}}
    \gtrsim \tau^{-\hf} \norm{H_+ u}_{L^2(\R; \H^{\ell+1})}.
  \end{align*}
  
  We now invoke the corresponding estimate provided by
  Lemma~\ref{lemma: estimates for large tau},
  \begin{align*}
    \norm{H_+ \P_{F+} \op{\psi_\epsilon} \omega}_{L^2(\R; \H^\ell)}
    \gtrsim \val{\op{\psi_\epsilon} \omega\br}_{\H^{\ell+\hf}} 
    + \norm{H_+ \op{\psi_\epsilon} \omega}_{L^2(\R; \H^{\ell+1})}
  \end{align*}
  and we proceed as in the end of the proof of Lemma~\ref{lemma:
    estimate PE+ PE-} to obtain the result for $\P_{F+}$. The same
  computation and arguments, {\em mutatis mutandis}, give the 
  result for $\P_{F-}$.
\end{proof}

% Proof of the Carleman estimate
\section{Proof of the Carleman estimate}
\label{Sec: proof Carleman}

From the estimates for the first-order factors obtained in
Section~\ref{sec: estimate 1st order factor} we shall now prove
Proposition~\ref{mainprop} which gives the result of
Theorem~\ref{1.thm.main} and Theorem~\ref{thm.main - nonhomogeneous}
(see the end of Section~\ref{sec: presentation}).
\subsection{The geometric hypothesis}

In section~\ref{sec: choice weight function} we chose a weight function $\varphi$ that satisfies the following condition 
\begin{equation}
  \label{mainassu2}
  \frac{\alpha_{+}}{\alpha_{-}}>
   \sup_{x',\xi' \atop |\xi'|\geq 1} \frac{m_+(x',\xi') \br}
   {m_-(x',\xi')\bl}, 
  \qquad \alpha_\pm = \p_{x_n} {\varphi_\pm}_{|x_n=0^\pm}.
\end{equation}

Let us explain the immediate consequences of that assumption: first of all,
we can reformulate it by saying that
\begin{equation}\label{maasre}
\exists \sigma>1,\quad
\frac{\alpha_{+}}{\alpha_{-}}=\sigma^2
\sup_{x',\xi' \atop |\xi'|\geq 1} 
\frac{m_+(x',\xi')\br}{m_-(x',\xi')\bl}.
\end{equation}
Let $1<\sigma_{0}<\sigma$.

First, consider $(x',\xi', \tau)\in \R^{n-1}\times \R^{n-1}\times \R^{+,\ast}$,
  $|\xi'| \geq 1$, such
  that 
  \begin{equation}
  \label{cone00} 
  \tau \alpha_{+}\ge \sigma_{0} m_+(x',\xi')\br.  
  \end{equation}
  Observe that we then have 
  \begin{align}
  \label{eq: ineq cone00} 
   \tau \alpha_{+} -m_{+}(x',\xi')\br 
  &\ge \tau\alpha_{+}(1-\sigma_{0}^{-1})
  \ge \frac{\sigma_{0}-1}{2\sigma_{0}}\tau\alpha_{+}
  +\frac{\sigma_{0}-1}{2} m_+(x',\xi')\br\\ 
  &\geq
  C \lambda.\nonumber   
  \end{align} 
  We choose $\tau$ \suff large, say
  $\tau\geq \tau_2>0$, so that this inequality remains true for
  $0\leq  |\xi'|\leq 2$. It also remains true for $x_n>0$ small.
  As $f_{+} = \tau (\varphi' -\alpha_+)
  + \tau \alpha_{+} -m_{+}(x,\xi')$, for the support of $v_+$
  sufficiently small, we obtain $f_{+} \geq C \lambda$, which means
  that $f_{+}$ is elliptic positive in that region.
  
Second, if we now have $|\xi'| \geq 1$ and
  \begin{equation}\label{cone01}
    \tau \alpha_{+}\le \sigma m_+(x',\xi')\br ,
  \end{equation}
  we get that $\tau\alpha_{-}\le
  \sigma^{-1} m_-(x',\xi')\bl$: otherwise we would have
  $\tau\alpha_{-}>\sigma^{-1} m_-(x',\xi')\bl$ and thus
  \begin{align*}
    &\frac{m_-(x',\xi')\bl}{\sigma \alpha_{-}}
    <\tau
    \le \frac{\sigma m_+(x',\xi')\br}{\alpha_{+}},
  \end{align*}
  implying 
  \begin{align*}
    &\frac{\alpha_{+}}{\alpha_{-}}
    <\sigma^2\frac{m_+(x',\xi')\br}{m_-(x',\xi')\bl}
    \leq \sigma^2 \sup_{x',\xi' \atop |\xi'|\geq 1} \frac{m_+(x',\xi') \br}
   {m_-(x',\xi')\bl}
    =\frac{\alpha_{+}}{\alpha_{-}}\quad\text{which is impossible.}
  \end{align*}
  As a consequence we have
  \begin{multline}
  \label{eq: ineq cone01}
    \tau \alpha_{-} -m_{-}(x',\xi')\bl
    \le -m_{-}(x',\xi')\bl \frac{(\sigma-1)}{\sigma}
    \\
    \le-m_{-}(x',\xi')\bl \frac{(\sigma-1)}{2\sigma}
    -\frac{(\sigma-1)}{2} \tau\alpha_{-}\leq - C \lambda.
  \end{multline}
  With $f_- = \tau (\varphi' - \alpha_-) + \tau \alpha_{-}
  -m_{-}(x,\xi')$, for the support of $v_-$ sufficiently small, we
  obtain $f_{-} \leq - C \lambda$, which means that $f_{-}$ is elliptic
  negative in that region. 

\begin{figure}
\begin{center}
 %\scalebox{0.55}{\hskip-85pt\includegraphics{pic002.pdf}} \vskip-105pt 
 \begin{picture}(0,0)%
\includegraphics{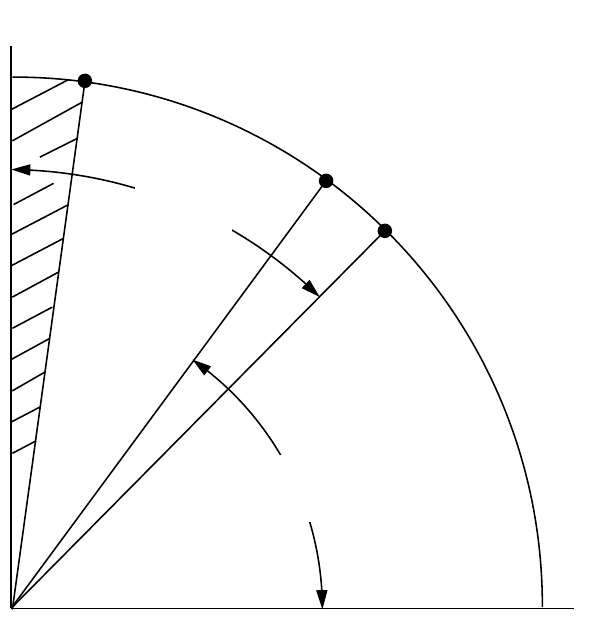}%
\end{picture}%
\setlength{\unitlength}{3947sp}%
\begingroup\makeatletter\ifx\SetFigFont\undefined%
\gdef\SetFigFont#1#2#3#4#5{%
  \reset@font\fontsize{#1}{#2pt}%
  \fontfamily{#3}\fontseries{#4}\fontshape{#5}%
  \selectfont}%
\fi\endgroup%
\begin{picture}(2829,3033)(547,-2275)
\put(562,635){\makebox(0,0)[lb]{\smash{{\SetFigFont{10}{12.0}{\rmdefault}{\mddefault}{\updefault}{\color[rgb]{0,0,0}$\tau$}%
}}}}
\put(3361,-2215){\makebox(0,0)[lb]{\smash{{\SetFigFont{10}{12.0}{\rmdefault}{\mddefault}{\updefault}{\color[rgb]{0,0,0}$|\xi'|$}%
}}}}
\put(1807,-1651){\makebox(0,0)[lb]{\smash{{\SetFigFont{8}{9.6}{\rmdefault}{\mddefault}{\updefault}{\color[rgb]{0,0,0}$\tilde{\Gamma}_\sigma$}%
}}}}
\put(2053,-1447){\makebox(0,0)[lb]{\smash{{\SetFigFont{8}{9.6}{\rmdefault}{\mddefault}{\updefault}{\color[rgb]{0,0,0}elliptic $-$}%
}}}}
\put(2059,-1309){\makebox(0,0)[lb]{\smash{{\SetFigFont{8}{9.6}{\rmdefault}{\mddefault}{\updefault}{\color[rgb]{0,0,0}$F_-$}%
}}}}
\put(1081,-589){\makebox(0,0)[lb]{\smash{{\SetFigFont{8}{9.6}{\rmdefault}{\mddefault}{\updefault}{\color[rgb]{0,0,0}elliptic $+$}%
}}}}
\put(1081,-433){\makebox(0,0)[lb]{\smash{{\SetFigFont{8}{9.6}{\rmdefault}{\mddefault}{\updefault}{\color[rgb]{0,0,0}$F_+$}%
}}}}
\put(1261,-247){\makebox(0,0)[lb]{\smash{{\SetFigFont{8}{9.6}{\rmdefault}{\mddefault}{\updefault}{\color[rgb]{0,0,0}$\Gamma_{\sigma_0}$}%
}}}}
\put(2455,-325){\makebox(0,0)[lb]{\smash{{\SetFigFont{8}{9.6}{\rmdefault}{\mddefault}{\updefault}{\color[rgb]{0,0,0}$\tau \alpha_+ = \sigma_0 m_+(x',\xi')\br$}%
}}}}
\put(2149,-37){\makebox(0,0)[lb]{\smash{{\SetFigFont{8}{9.6}{\rmdefault}{\mddefault}{\updefault}{\color[rgb]{0,0,0}$\tau \alpha_+ = \sigma m_+(x',\xi')\br$}%
}}}}
\put(942,448){\makebox(0,0)[lb]{\smash{{\SetFigFont{8}{9.6}{\rmdefault}{\mddefault}{\updefault}{\color[rgb]{0,0,0}$\epsilon \tau = \valjp{\xi'}$}%
}}}}
\end{picture}%

 \caption{The 
  overlapping microlocal regions $\Gamma_{\sigma_{0}}$, and
  $\wt\Gamma_{\sigma}$ in the $\tau,\val{ \xi'}$ plane above a point
  $x'$. Dashed is the region used in Section~\ref{sec: prelim cut-off}
  which is kept away from the overlap of $\Gamma_{\sigma_{0}}$, and
  $\wt{\Gamma}_{\sigma}$.}
  \label{fig: overlapping microlocal regions}
\end{center}
\end{figure}
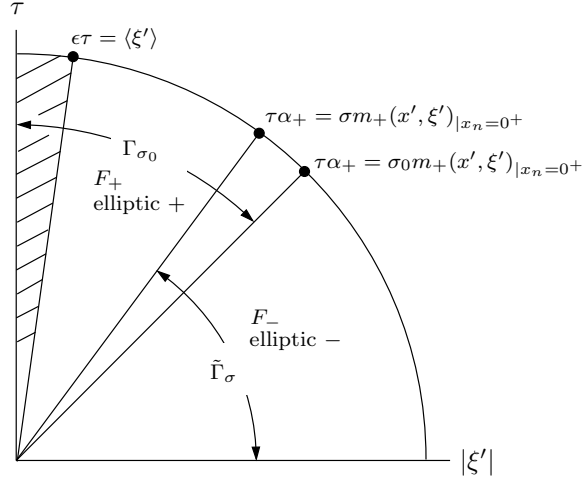

\medskip
We have thus proven the following result.
\begin{lem}
  \label{4.lem.keybus}
  Let $\sigma>\sigma_{0}>1$, and $\alpha_{\pm}$,
  be positive numbers such that
  \eqref{maasre} holds.  For $s>0$, we define the following cones in
  $\R^{n-1}_{x'} \times \R^{n-1}_{\xi'}\times{\R_{+}^*}$ by
  \begin{align*}
  &\Gamma_{s} = \big\{ (x',\tau,\xi');\ |\xi'| < 2\ \text{or}\ 
  \tau \alpha_{+}>s m_+(x',\xi')\br \big\},\\
  &\wt{\Gamma}_{s} = \big\{ (x',\tau,\xi');\ |\xi'| > 1\ \text{and}\ 
  \tau \alpha_{+}<s m_+(x',\xi')\br \big\}.
  \end{align*}
  For the supports of $v_+$ and $v_-$ \suff small and $\tau$ \suff large, 
  we have $\R^{n-1}\times\R^{n-1}\times\R_{+}^*= 
  \Gamma_{\sigma_{0}}\cup \wt{\Gamma}_{\sigma}$
  and
  \begin{multline*}
    \Gamma_{\sigma_{0}} \subset
    \big\{
    (x',\xi',\tau)\! \in \R^{n-1}\!\!\times \R^{n-1}\!\!\times\R_{+}^\ast; \ 
      \forall x_n\geq 0,\ f_{+}(x,\xi')\ge C \lambda, \\
       \text{if}\ (x',x_n) \in \supp(v^+)
    \big\},
    \end{multline*}
  \begin{multline*}
    \wt{\Gamma}_{\sigma} \subset
    \big\{
      (x',\xi',\tau)\!\in\!\R^{n-1}\!\!\times \R^{n-1}\!\!\times\R_{+}^\ast; \
      \forall x_n\leq 0,\
      f_{-}(x,\xi')\leq - C \lambda, \\
      \text{if}\ (x',x_n) \in \supp(v^-)
    \big\}.
  \end{multline*}
\end{lem}
\begin{nb}{ The key result for the sequel is that property
    \eqref{mainassu2} is securing the fact that the overlapping open
    regions $\Gamma_{\sigma_{0}}$ and $\wt{\Gamma}_{\sigma}$ are such that on
    $\Gamma_{\sigma_{0}}$, $f_{+}$ is elliptic positive and on
    $\wt{\Gamma}_{\sigma}$, $f_{-}$ is elliptic negative.  Using a
    partition of unity and symbolic calculus, we shall be able to
    assume that either $F_{+}$ is elliptic positive, or $F_{-}$ is
    elliptic negative.  }
\end{nb}

\begin{nb}{Note that we can keep the preliminary cut-off region of Section~\ref{sec: prelim cut-off} away from the 
overlap of $\Gamma_{\sigma_{0}}$ and $\wt{\Gamma}_{\sigma}$ by
choosing $\epsilon$ \suff small (see \eqref{eq: Fourier multiplier} and Lemma~\ref{lemma: estimates for large tau}). This is illustrated in Figure~\ref{fig: overlapping microlocal regions}.}
\end{nb}

With the two overlapping cones, for $\tau\geq \tau_2$, we introduce an
homogeneous partition of unity
\begin{equation}\label{3.partition}
  1=\chi_{0}(x',\xi',\tau)+\chi_{1}(x',\xi',\tau),\quad
  \underbrace{\supp(\chi_{0})\subset 
  \Gamma_{\sigma_{0}}}_{\substack{\val{\xi'}\lesssim \tau,\ 
      f_{+}\text{ elliptic $>0$}}},\qquad
  \underbrace{\supp(\chi_{1})\subset 
  \wt{\Gamma}_{\sigma}}_{\substack{\val{\xi'}\gtrsim \tau,\ 
      f_{-}\text{ elliptic $<0$}}}.
\end{equation}
Note that $\chi_j'$, $j=0,1$, are supported at the overlap of the
 regions $\Gamma_{\sigma_{0}}$ and $\wt{\Gamma}_{\sigma}$, where
 $\tau \lesssim |\xi'|$.  Hence, $\chi_0$ and $\chi_1$ satisfy the
 estimates of the semi-classical calculus and we have $\chi_0$,
 $\chi_1 \in \ST^0$. With these symbols we associate the
 following operators.
\begin{equation}\label{}
  \Xi_{j}=\opw{\chi_{j}},\ j=0,1\text{ and we have $\Xi_{0}+\Xi_{1}=\Id$.}
\end{equation}
From the transmission conditions~\eqref{eq: transmission conditions
varphi} we find
\begin{equation}
  \label{eq: microlocal trans conditions}
  \Xi_j {v_+}\br - \Xi_j {v_-}\bl = \Xi_j  \theta_\varphi,
\end{equation}
and 
\begin{multline*}
  a_{nn}^+(D_{n}+T_{+}+i\tau\varphi'_{+})\Xi_j {v_{+}}\br - 
  a_{nn}^-(D_{n}+T_{-}+i\tau\varphi'_{-}) \Xi_j {v_{-}} \bl\\
  = \Xi_j \Theta_\varphi 
  + \opw{\kappa_0} v\br+ \opw{\tilde{\kappa}_0}\theta_\varphi,  \quad j=0,1, 
\end{multline*}
with $\kappa_0, \tilde{\kappa}_0 \in \ST^0$ that originate
from commutators and \eqref{eq: microlocal trans conditions}.  Defining
\begin{align}
  \label{eq: trace V}
  \V_{j, \pm} 
  = a_{nn}^\pm (D_{n}+S_{\pm}+i\tau\varphi'_{\pm}) \Xi_j {v_{\pm}}_{|x_n=0^\pm}
\end{align}
and recalling \eqref{eq: def S}
 we find
\begin{align}
  \label{eq: microlocal trans conditions 2}
  \V_{j, +} - \V_{j, -} = \Xi_j \Theta_\varphi  
  + \opw{\kappa_1} v\br 
  + \opw{\tilde{\kappa}_1} {\theta_\varphi}, \quad \kappa_1, \tilde{\kappa}_1 \in \ST^0.
\end{align}
We shall now prove microlocal Carleman estimates in the  two regions $\Gamma_{\sigma_{0}}$ and $\wt{\Gamma}_{\sigma}$.

\subsection{Region $\Gamma_{\sigma_{0}}$: both roots are
  positive on the positive half-line}
\label{sec: f + positive}

On the one hand, from Lemma~\ref{lemma: estimate PE+ PE-} we have
\begin{equation}
  \label{eq: est fact PE+ in action}
  \norm{H_{+}\P_{+} \Xi_0 v_{+}}\gtrsim 
  \val{\V_{0,+}-ia_{nn}^+M_{+}\Xi_0 {v_{+}}\br}_{\H^{\hf}}
  +\norm{H_{+} \P_{F+} \Xi_0  v_{+}}_{L^2(\R;\H^{1})},
\end{equation}
The positive ellipticity of $F_{+}$ on the support of $\chi_{0}$
allows us to reiterate the estimate by 
Lemma~\ref{lemma: estimate PF+ F+ positive}
to obtain
\begin{multline*}
  \norm{H_{+}\P_{+} \Xi_0 v_{+}}
  +  \norm{H_+ v_+}
  \gtrsim
  \val{\V_{0,+}-ia_{nn}^+M_{+}\Xi_0 {v_{+}}\br}_{\H^{\hf}}
  +\val{\Xi_0 {v_{+}}\br}_{\H^{3/2}}\\
  \quad + \norm{H_{+}\Xi_0 v_{+}}_{L^2(\R;\H^2)} 
  + \norm{H_{+}D_n \Xi_0 v_{+}}_{L^2(\R;\H^1)}.
\end{multline*}
Since we have also
\begin{equation}
\label{eq: est ineg triang}
\val{\V_{0,+}}_{\H^{\hf}}\lesssim
\val{\V_{0,+}-ia_{nn}^+ M_{+}\Xi_0 {v_+}\br}_{\H^{\hf}}
+\val{\Xi_0 {v_{+}}\br}_{\H^{3/2}},
\end{equation}
writing the $\H^\hf$ norm as $\val{.}_{\H^\hf}\sim \tau^\hf
\val{.}_{L^2} + \val{.}_{H^\hf}$ and using the regularity of $M_+ \in
\op{\S^1}$ in the standard calculus, we obtain
\begin{multline}
  \label{eq: low freq est +}
  \norm{H_{+}\P_{+} \Xi_0 v_{+}}
  +  \norm{H_+ v_+}
  \gtrsim
  \val{\V_{0,+}}_{\H^{\hf}}
  +\val{\Xi_0 {v_{+}}\br}_{\H^{3/2}}\\
  + \norm{H_{+}\Xi_0  v_{+}}_{L^2(\R;\H^2)} 
  + \norm{H_{+}\Xi_0 D_n  v_{+}}_{L^2(\R;\H^1)}.
\end{multline}

On the other hand, with Lemma~\ref{lemma: increasing root}
we have, for $k=0$ or $k=\hf$,
\begin{multline*}
  \norm{H_{-}\P_{-} \Xi_0 v_{-}}_{L^2(\R; \H^{-k})}
  + \val{\V_{0,-}+ia_{nn}^- M_{-}\Xi_0 {v_{-}}\bl}_{\H^{\hf-k}}\\
  \gtrsim 
  \tau^{-\hf} \norm{H_{-} \P_{E-} \Xi_0  v_{-}}_{L^2(\R;\H^{1-k})}. 
\end{multline*}
This gives
\begin{multline*}
  \norm{H_{-}\P_{-} \Xi_0 v_{-}}
  + \tau^k \val{\V_{0,-}+ia_{nn}^- M_{-}\Xi_0 {v_{-}}\bl}_{\H^{\hf-k}}
  \gtrsim 
  \tau^{k-\hf} \norm{H_{-} \P_{E-} \Xi_0  v_{-}}_{L^2(\R;\H^{1-k})},
\end{multline*}
which with Lemma~\ref{lemma: estimate PE+ PE-}  yields
\begin{multline*}
  \norm{H_{-}\P_{-} \Xi_0 v_{-}}
  + \tau^k \val{\V_{0,-}+ia_{nn}^- M_{-}\Xi_0 {v_{-}}\bl}_{\H^{\hf-k}}
  + \tau^{k-\hf}  \val{\Xi_0  {v_{-}}\bl}_{\H^{\frac{3}{2}-k}}\\
   \gtrsim 
  \tau^{k-\hf}\Big( \norm{H_{-} \Xi_0  v_{-}}_{L^2(\R;\H^{2-k})}
  + \norm{H_{-} \Xi_0 D_n v_{-}}_{L^2(\R;\H^{1-k})} \Big).
\end{multline*}
Arguing as for~\eqref{eq: est ineg triang} we find
\begin{multline}
  \label{eq: low freq est -}
  \norm{H_{-}\P_{-} \Xi_0 v_{-}}
  + \tau^k \val{\V_{0,-}}_{\H^{\hf-k}}
  + \tau^k \val{\Xi_0  {v_{-}}\bl}_{\H^{\frac{3}{2}-k}}\\
   \gtrsim 
  \tau^{k-\hf}\Big( \norm{H_{-} \Xi_0  v_{-}}_{L^2(\R;\H^{2-k})}
  + \norm{H_{-} \Xi_0 D_n v_{-}}_{L^2(\R;\H^{1-k})} \Big).
\end{multline}

Now, from the transmission conditions \eqref{eq: microlocal trans
  conditions}--\eqref{eq: microlocal trans conditions 2}, by adding
$\varepsilon \eqref{eq: low freq est -} + \eqref{eq: low freq est +}$
we obtain
\begin{multline*}
  \norm{H_{-}\P_{-} \Xi_0 v_{-}}
  +\norm{H_{+}\P_{+} \Xi_0 v_{+}}
  + \tau^k\Big( \val{\theta_\varphi}_{\H^{\frac{3}{2}-k}}
  + \val{\Theta_\varphi}_{\H^{\hf-k}} + \val{v\br}_{\H^{\hf-k}}\Big)
  + \norm{H_+ v_+}
  \\
   \gtrsim 
   \tau^k \Big(\val{\V_{0,-}}_{\H^{\hf-k}} + \val{\V_{0,+}}_{\H^{\hf-k}}
   +\val{\Xi_0  {v_{-}}\bl}_{\H^{\frac{3}{2}-k}} 
   +\val{\Xi_0  {v_{+}}\br}_{\H^{\frac{3}{2}-k}}\Big) \\
   + \tau^{k-\hf}\Big( \norm{\Xi_0  v}_{L^2(\R;\H^{2-k})}
   + \norm{H_{-} \Xi_0 D_n v_{-}}_{L^2(\R;\H^{1-k})}
   + \norm{H_{+} \Xi_0 D_n v_{+}}_{L^2(\R;\H^{1-k})} 
   \Big).
\end{multline*}
by choosing $\varepsilon>0$ \suff small and $\tau$ \suff large.
Finally, recalling the form of $\V_{0,\pm}$, arguing as for~\eqref{eq: est ineg triang} we obtain
\begin{multline}
  \label{eq: estimate low freq}
  \norm{H_{-}\P_{-} \Xi_0 v_{-}}
  +\norm{H_{+}\P_{+} \Xi_0 v_{+}}
  + \tau^k\Big( \val{\theta_\varphi}_{\H^{\frac{3}{2}-k}}
  + \val{\Theta_\varphi}_{\H^{\hf-k}} + \val{v\br}_{\H^{\hf-k}}\Big)
  + \norm{H_+ v_+}
  \\
   \gtrsim 
   \tau^k \Big( \val{\Xi_0 D_n {v_-}\bl}_{\H^{\hf-k}}
   + \val{\Xi_0 D_n {v_+}\br}_{\H^{\hf-k}} 
   +\val{\Xi_0  {v_{-}}\bl}_{\H^{\frac{3}{2}-k}} 
   +\val{\Xi_0  {v_{+}}\br}_{\H^{\frac{3}{2}-k}}\Big) \\
   + \tau^{k-\hf}\Big( \norm{\Xi_0  v}_{L^2(\R;\H^{2-k})}
   + \norm{H_{-} \Xi_0 D_n v_{-}}_{L^2(\R;\H^{1-k})}
   + \norm{H_{+} \Xi_0 D_n v_{+}}_{L^2(\R;\H^{1-k})} 
   \Big),
\end{multline}
for $k=0$ or $k=\hf$.

\begin{rem}
        \label{rem: case k=0 - 1}
  Note that in the case $k=0$, recalling the form of the second-order operators 
$\P_{\pm}$, we can estimate the additional terms 
         $\tau^{-\hf} \norm{H_{\pm} \Xi_0 D_n^2 v_{\pm}}$.   
\end{rem}

\subsection{Region $\tilde{\Gamma}_{\sigma}$: only one root is positive on the positive half-line}
\label{sec: f - negative}

This case is more difficult a priori since we cannot expect to control
$v\br$ directly from the estimates of the first-order factors.
Nevertheless when the positive ellipticity of $F_{+}$ is violated,
then $F_{-}$ is elliptic negative: this is the result of our main
geometric assumption in Lemma \ref{4.lem.keybus}.

As in~\eqref{eq: est fact PE+ in action} we have 
\begin{equation*}
  \norm{H_{+}\P_{+} \Xi_1 v_{+}}\gtrsim 
  \val{\V_{1,+}-ia_{nn}^+M_{+}\Xi_1 {v_{+}}\br}_{\H^{\hf}}
  +\norm{H_{+} \P_{F+} \Xi_1  v_{+}}_{L^2(\R;\H^1)}.
\end{equation*}
 and using Lemma~\ref{lemma: estimate PF- F- negative}
 for the negative half-line, we have
 \begin{multline*}
    \norm{H_- \P_{-} \Xi_1 v_{-}} + \norm{H_- v_{-}} +  \norm{H_- D_n v_{-}}
    \\
    \gtrsim  \val{\V_{1,-}+ia_{nn}^- M_{-}\Xi_1 {v_{-}}\bl}_{\H^\hf} 
    + \norm{H_- \P_{E-} \Xi_1  v_{-}}_{L^2(\R; \H^1)}.
  \end{multline*}
  A quick glance at the above estimate shows that none could be
  iterated in a favorable manner, since $F_{+}$ could be negative on
  the positive half-line and $E_{-}$ is indeed positive on the
  negative half-line.  We have to use the additional information given
  by the transmission conditions.  From the above inequalities, we
  control
  $$
  \tau^k \Big( \val{\V_{1,-}+ia_{nn}^-M_{-}\Xi_1 {v_{-}}\bl}_{\H^{\hf-k}}
  +\val{-\V_{1,+}+ia_{nn}^+ M_{+}\Xi_1 {v_{+}}\br}_{\H^{\hf-k}}\Big),
  $$
  for $k=0$ or $ \hf$, 
  which, by the transmission conditions \eqref{eq: microlocal trans
    conditions}--\eqref{eq: microlocal trans conditions 2} implies the
  control of
  \begin{align*}
    &\tau^k \val{\V_{1,-}-\V_{1,+} 
      + ia_{nn}^-M_{-}\Xi_1 {v_{-}}\bl+ ia_{nn}^+M_{+}\Xi_1 {v_{+}}\br}_{\H^{\hf-k}}
      \\
    &\geq 
    \tau^k \val{(a_{nn}^-M_{-}+ a_{nn}^+M_{+}) \Xi_1 {v_+}\br}_{\mathcal H^{\hf-k}}\\
    &\quad - C \tau^k  \big(\val{\Theta_\varphi}_{\mathcal H^{\hf-k}}
      + \val{\theta_\varphi}_{\mathcal H^{\frac{3}{2}-k}}
      + \val{{v_+}\br}_{\mathcal H^{\hf-k}} \big).
\end{align*}
Let now $\tilde{\chi}_1 \in \ST^0$ satisfying the same
properties as $\chi_1$, with moreover $\tilde{\chi}_1=1$ on a \nhd of
$\supp(\chi_1)$. We then write
  $$
  m_\pm = \check{m}_\pm + r, \quad 
  \text{with}\ \ \check{m}_\pm = m_\pm \tilde{\chi}_1 
  +\lambda (1-\tilde{\chi}_1), 
  \quad r = (m_\pm + \lambda)(1-\tilde{\chi}_1).
  $$
  We have $\check{m}_\pm \geq C \lambda$  and $\check{m}_\pm \in \ST^1$ because of the support of $\tilde{\chi}_1$.  Because of the supports of
  $1-\tilde{\chi}_1$ and $\chi_1$, in particular $\tau\lesssim
  \val{\xi'}$ in $\supp(\chi_1)$, Lemma~\ref{lemma: mixed composition}
  yields $r \sharp \chi_1 \in \ST^{-\infty}$.  With
  Lemma~\ref{lem: postivity garding tau} and \eqref{eq: microlocal trans conditions} we thus obtain
  \begin{multline*}
    \val{\V_{1,-}+ia_{nn}^-M_{-}\Xi_1 {v_{-}}\bl}_{\H^{\hf-k}}
  +\val{-\V_{1,+}+ia_{nn}^+ M_{+}\Xi_1 {v_{+}}\br}_{\H^{\hf-k}}\\
  +\val{\Theta_\varphi}_{\mathcal H^{\hf-k}}
      + \val{\theta_\varphi}_{\mathcal H^{\frac{3}{2}-k}}
      + \val{{v_+}\br}_{\mathcal H^{\hf-k}}
  \gtrsim \val{\Xi_1 {v_-}\bl}_{\H^{\frac{3}{2}-k}}
  + \val{\Xi_1 {v_+}\br}_{\H^{\frac{3}{2}-k}}.
  \end{multline*}
  From the form of $\V_{1,+}$ we moreover obtain
  \begin{multline*}
    \val{\V_{1,-}+ia_{nn}^-M_{-}\Xi_1 {v_{-}}\bl}_{\H^{\hf-k}}
    +\val{-\V_{1,+}+ia_{nn}^+ M_{+}\Xi_1 {v_{+}}\br}_{\H^{\hf-k}}\\
    +\val{\Theta_\varphi}_{\mathcal H^{\hf-k}}
      + \val{\theta_\varphi}_{\mathcal H^{\frac{3}{2}-k}}
      + \val{{v_+}\br}_{\mathcal H^{\hf-k}}
    \gtrsim \val{\Xi_1 {v_-}\bl}_{\H^{\frac{3}{2}-k}}
    +  \val{\Xi_1 {v_+}\br}_{\H^{\frac{3}{2}-k}}\\
    + \val{\Xi_1 D_n {v_-}\bl}_{\H^{\hf-k}}
    + \val{\Xi_1 D_n {v_+}\br}_{\H^{\hf-k}}.
  \end{multline*}
  We thus have
  \begin{multline*}
    \norm{H_- \P_{-} \Xi_1 v_{-}} + \norm{H_{+}\P_{+} \Xi_1 v_{+}}
    + \tau^k \big(\val{\Theta_\varphi}_{\mathcal H^{\hf-k}}
      + \val{\theta_\varphi}_{\mathcal H^{\frac{3}{2}-k}}
    + \val{{v_+}\br}_{\mathcal H^{\hf-k}}\Big)  
    + \norm{H_- v_{-}}\\ +  \norm{H_- D_n v_{-}}
    \gtrsim  \tau^k \Big( \val{\Xi_1 {v_-}\bl}_{\H^{\frac{3}{2}-k}}
    +  \val{\Xi_1 {v_+}\br}_{\H^{\frac{3}{2}-k}}
    + \val{\Xi_1 D_n {v_-}\bl}_{\H^{\hf-k}}\\
    + \val{\Xi_1 D_n {v_+}\br}_{\H^{\hf-k}}
    + \norm{H_- \P_{E-} \Xi_1  v_{-}}_{L^2(\R; \H^{1-k})}
    +\norm{H_{+} \P_{F+} \Xi_1  v_{+}}_{L^2(\R;\H^{1-k})}\Big),
\end{multline*}
for $k=0$ or $\hf$.
The remaining part of the discussion is very similar to the last part
of the argument in the previous subsection. By Lemmas~\ref{lemma:
  estimate PE+ PE-} and \ref{lemma: increasing root} we have
\begin{multline*}
  \norm{H_- \P_{E-} \Xi_1  v_{-}}_{L^2(\R; \H^{1-k})}
  + \val{\Xi_1 {v_-}\bl}_{\H^{\frac{3}{2}-k}}\\
  \gtrsim \norm{H_- \Xi_1  v_{-}}_{L^2(\R; \H^{2-k})}
  + \norm{H_- \Xi_1 D_n v_{-}}_{L^2(\R; \H^{1-k})}
\end{multline*}
and
\begin{multline*}
  \norm{H_+ \P_{F_+} \Xi_1 v_+}_{L^2(\R; \H^{1-k})} 
  + \val{\Xi_1 {v_+}\br}_{\H^{\frac{3}{2}-k}}\\
  \gtrsim  
    \tau^{-\hf} \Big(\norm{H_+ \Xi_1  v_{+}}_{L^2(\R; \H^{2-k})} 
    + \norm{H_+ \Xi_1 D_n v_{+}}_{L^2(\R; \H^{1-k})} \Big).
\end{multline*}
Since $\val{\Xi_1{ v_\pm}_{|x_n=0^\pm}}_{\H^{\frac{3}{2}-k}}$ are already controlled,
we control as well the \rhs of the above inequalities and have
\begin{multline}
  \label{eq: estimate high freq}
    \norm{H_- \P_{-} \Xi_1 v_{-}} + \norm{H_{+}\P_{+} \Xi_1 v_{+}}
    + \tau^k \big(\val{\Theta_\varphi}_{\mathcal H^{\hf-k}}
      + \val{\theta_\varphi}_{\mathcal H^{\frac{3}{2}-k}}
    + \val{{v_+}\br}_{\mathcal H^{\hf-k}}\Big) 
    + \norm{H_- v_{-}} \\ 
    +  \norm{H_- D_n v_{-}}
    \gtrsim \tau^k \Big(\val{\Xi_1 {v_-}\bl}_{\H^{\frac{3}{2}-k}}
    +  \val{\Xi_1 {v_+}\br}_{\H^{\frac{3}{2}-k}}
    + \val{\Xi_1 D_n {v_-}\bl}_{\H^{\hf-k}}\\
    + \val{\Xi_1 D_n {v_+}\br}_{\H^{\hf-k}}\Big)
    + \tau^{k-\hf}\Big( \norm{\Xi_1  v}_{L^2(\R;\H^{2-k})}
    + \norm{H_{-} \Xi_1 D_n v_{-}}_{L^2(\R;\H^{1-k})}\\
    + \norm{H_{+} \Xi_1 D_n v_{+}}_{L^2(\R;\H^{1-k})}
    \Big).
\end{multline}

\begin{rem}
        \label{rem: case k=0 - 2}
  Note that in the case $k=0$, recalling the form of the second-order operators 
$\P_{\pm}$, we can estimate the additional terms 
$\tau^{-\hf}\norm{H_{\pm} \Xi_1 D_n^2 v_{\pm}}$.
\end{rem}

\subsection{Patching together microlocal estimates}
\label{sec: microlocal patching}
We now sum estimates~\eqref{eq: estimate low freq} and \eqref{eq:
  estimate high freq} together. By the triangular inequality, this
gives, for $k=0$ or $\hf$,
\begin{multline*}
  \ssum_{j=0,1} \Big(\norm{H_{-}\P_{-} \Xi_j v_{-}}
  +\norm{H_{+}\P_{+} \Xi_j v_{+}}\Big)
  + \tau^k \Big(\val{\Theta_\varphi}_{\mathcal H^{\hf-k}}
      + \val{\theta_\varphi}_{\mathcal H^{\frac{3}{2}-k}}
    + \val{{v_+}\br}_{\mathcal H^{\hf-k}}\Big) \\
  +  \norm{H_+ v_+}
  + \norm{H_- v_{-}} +  \norm{H_- D_n v_{-}}\\
   \gtrsim \tau^k \Big( \val{{v_-}\bl}_{\H^{\frac{3}{2}-k}} 
   + \val{{v_+}\br}_{\H^{\frac{3}{2}-k}} 
   +  \val{ D_n {v_-}_{|x_n=0^-}}_{\H^{\hf-k}}
   +  \val{ D_n {v_+}_{|x_n=0^+}}_{\H^{\hf-k}}\Big)\\
   + \tau^{k-\hf}\Big( \norm{  v}_{L^2(\R;\H^{2-k})}
  + \norm{H_{-} D_n v_{-}}_{L^2(\R;\H^{1-k})}
  + \norm{H_{+} D_n v_{+}}_{L^2(\R;\H^{1-k})}\Big).
\end{multline*}
For $\tau$ \suff large we now obtain
\begin{multline*}
  \ssum_{j=0,1} \Big(\norm{H_{-}\P_{-} \Xi_j v_{-}}
  +\norm{H_{+}\P_{+} \Xi_j v_{+}}\Big)
  + \tau^k \Big(\val{\Theta_\varphi}_{\mathcal H^{\hf-k}}
      + \val{\theta_\varphi}_{\mathcal H^{\frac{3}{2}-k}}
      \Big)\\
       \gtrsim \tau^k \Big( \val{{v_-}\bl}_{\H^{\frac{3}{2}-k}} 
   + \val{{v_+}\br}_{\H^{\frac{3}{2}-k}} 
   +  \val{ D_n {v_-}_{|x_n=0^-}}_{\H^{\hf-k}}
   +  \val{ D_n {v_+}_{|x_n=0^+}}_{\H^{\hf-k}}\Big)\\
   + \tau^{k-\hf}\Big( \norm{  v}_{L^2(\R;\H^{2-k})}
  + \norm{H_{-} D_n v_{-}}_{L^2(\R;\H^{1-k})}
  + \norm{H_{+} D_n v_{+}}_{L^2(\R;\H^{1-k})}\Big).
\end{multline*}
Arguing with commutators, as in the end of Lemma~\ref{lemma: estimate
  PE+ PE-}, noting here that the second order operators $\P_{\pm}$
belong to the semi-classical calculus, \ie $\P_{\pm} \in \ST^2$, we
otbain, for $\tau$ \suff large,
\begin{multline*}
  \norm{H_{-}\P_{-} v_{-}}
  +\norm{H_{+}\P_{+}  v_{+}}
  + \tau^k \Big(\val{\Theta_\varphi}_{\mathcal H^{\hf-k}}
      + \val{\theta_\varphi}_{\mathcal H^{\frac{3}{2}-k}}
      \Big)\\
      \gtrsim \tau^k \Big( \val{{v_-}\bl}_{\H^{\frac{3}{2}-k}} 
   + \val{{v_+}\br}_{\H^{\frac{3}{2}-k}} 
   +  \val{ D_n {v_-}_{|x_n=0^-}}_{\H^{\hf-k}}
   +  \val{ D_n {v_+}_{|x_n=0^+}}_{\H^{\hf-k}}\Big)\\
   + \tau^{k-\hf}\Big( \norm{  v}_{L^2(\R;\H^{2-k})}
  + \norm{H_{-} D_n v_{-}}_{L^2(\R;\H^{1-k})}
  + \norm{H_{+} D_n v_{+}}_{L^2(\R;\H^{1-k})}\Big).
\end{multline*}
In particular this estimate allows us to absorb the perturbation in
$\Psi^1$ as defined by \eqref{3.classe} by taking $\tau$ large
enough. For $k=\hf$ we obtain the result of
Proposition~\ref{mainprop}, which concludes the proof of the Carleman
estimate.

\begin{nb}The case $k=0$ provides higher Sobolev norm estimates of the trace
terms ${v_\pm}_{|x_n=0^\pm}$ and $D_n {v_\pm}_{|x_n=0^\pm}$.  It also
allows one to estimate $\tau^{-\hf} \norm{H_{\pm} D_n^2 v_{\pm}}$ as
noted in Remarks~\ref{rem: case k=0 - 1} and \ref{rem: case k=0 -
2}. These estimation are obtained at the price of higher requirements
(one additional tangential half derivative) on the non-homogeneous
transmission condition functions $\theta$ and $\Theta$.
\end{nb}  

\subsection{Convexification}
\label{sec: convexification}
We want now to modify slightly the weight function $\varphi$, for
instance to allow some convexification.  We started with
$\varphi=H_{+}\varphi_{+}+H_{-}\varphi_{-}$ where $\varphi_{\pm}$ were
given by \eqref{30weight} and our proof relied heavily on a smooth
factorization in first-order factors.  
We modify $\varphi_{\pm}$ into
$$
\Phi_{\pm}(x',x_{n})=\underbrace{\alpha_{\pm}x_{n}+\frac12\beta x_{n}^2}_{\varphi_{\pm}(x_{n})}+\kappa(x',x_{n}),\quad \kappa\in\moo(\Omega;\R),\quad \val{d\kappa}\
\text{bounded on $\Omega$.}
$$
We shall prove below that the Carleman estimates of
Theorem~\ref{1.thm.main} and Theorem~\ref{thm.main - nonhomogeneous}
also holds in this case if we choose $\norm{\kappa'}_{L^\io}$
sufficiently small.

We start by inspecting what survives in our factorization argument.
We have from \eqref{2.ppm},
$
\mathcal P_{\pm}=
(D+i\tau d\Phi_{\pm})\cdot A_{\pm}(D+i\tau d\Phi_{\pm}),
$
so that, modulo $\Psi^1$,
\begin{multline}
   \label{eq: conjugated op - convex weight}
   \mathcal P_{\pm}\equiv a_{nn}^\pm \Big(
   \big[D_{n}+S_\pm(x,D')
   +i\tau\bigl(\p_{n}\Phi_{\pm}+S_\pm(x,\p_{x'}\Phi_{\pm})\bigr)
   \big]^2
\\+\frac{b_{jk}^\pm}{a_{nn}^\pm} 
  (D_{j}+i\tau\p_{j}\Phi_{\pm})
  (D_{k}+i\tau\p_{k}\Phi_{\pm})\Big).
\end{multline}
(See also \eqref{3.kbgf44}.)
The new difficulty comes from the fact that the roots in the
variable $D_{n}$ are not necessarily smooth: when $\Phi$ does not
depend on $x'$, the symbol of the term
$b_{jk}^\pm(D_{j}+i\tau\p_{j}\Phi_{\pm})(D_{k}+i\tau\p_{k}\Phi_{\pm})$
equals $b_{jk}^\pm\xi_{j}\xi_{k}$ and thus is positive elliptic with a
smooth positive square root.  It is no longer the case when we have an
actual dependence of $\Phi$ upon the variable $x'$; nevertheless, we
have, as $\p_{x'} \Phi_\pm = \p_{x'} \kappa$,
\begin{multline*}
   \re{\big( \frac{b_{jk}^\pm}{a_{nn}^\pm}
   (\xi_{j}+i\tau\p_{j}\kappa)(\xi_{k}+i\tau\p_{k}\kappa )\big)}
   = \frac{b_{jk}^\pm}{a_{nn}^\pm} \xi_j \xi_k
   -\tau^2\frac{b_{jk}^\pm}{a_{nn}^\pm} \p_{j}\kappa \p_{k}\kappa
   \\
   \ge 
   (\lambda_{0}^\pm)^2\val{\xi'}^2
   -\tau^2(\lambda_{1}^\pm)^2\val{\p_{x'}\kappa}^2
   \ge \frac34(\lambda_{0}^\pm)^2\val{\xi'}^2,\qquad
\text{if}\  
\tau\norm{\p_{x'}\kappa}_{L^\io}\le \frac{\lambda_{0}^\pm}{2\lambda_{1}^\pm}\val{\xi'},
\end{multline*}
where 
\begin{align*}
  \lambda_0^\pm = \inf_{x',\xi \atop |\xi'|=1} 
  \Big( \frac{b_{jk}^\pm}{a_{nn}^\pm} \xi_j \xi_k\Big)^\hf_{|x_n = 0^\pm}, \qquad 
  \lambda_1^\pm = \sup_{x',\xi \atop |\xi'|=1} 
  \Big( \frac{b_{jk}^\pm}{a_{nn}^\pm} \xi_j \xi_k\Big)^\hf_{|x_n = 0^\pm}, 
\end{align*}

As a result, the roots are smooth when
$\tau\norm{\p_{x'}\kappa}_{L^\io}
\le \frac{\lambda_{0}^\pm}{2\lambda_{1}^\pm}
\val{\xi'}$.

\begin{figure}
\begin{center}
 \begin{picture}(0,0)%
\includegraphics{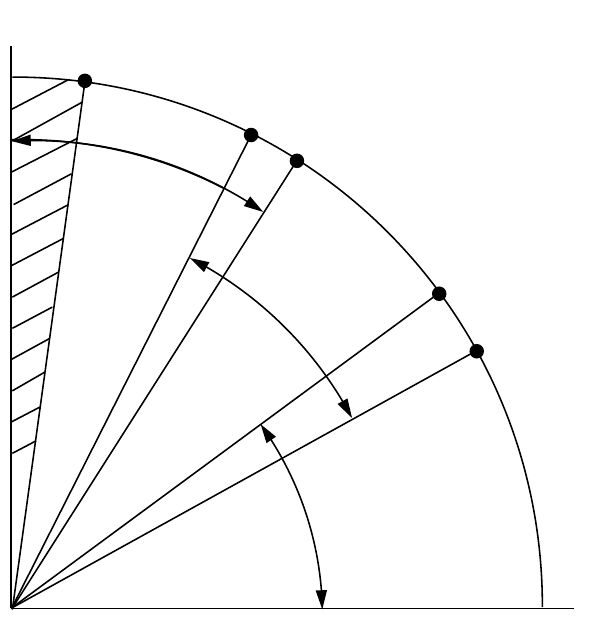}%
\end{picture}%
\setlength{\unitlength}{3947sp}%
\begingroup\makeatletter\ifx\SetFigFont\undefined%
\gdef\SetFigFont#1#2#3#4#5{%
  \reset@font\fontsize{#1}{#2pt}%
  \fontfamily{#3}\fontseries{#4}\fontshape{#5}%
  \selectfont}%
\fi\endgroup%
\begin{picture}(2829,3033)(547,-2275)
\put(562,635){\makebox(0,0)[lb]{\smash{{\SetFigFont{10}{12.0}{\rmdefault}{\mddefault}{\updefault}{\color[rgb]{0,0,0}$\tau$}%
}}}}
\put(3361,-2215){\makebox(0,0)[lb]{\smash{{\SetFigFont{10}{12.0}{\rmdefault}{\mddefault}{\updefault}{\color[rgb]{0,0,0}$|\xi'|$}%
}}}}
\put(942,448){\makebox(0,0)[lb]{\smash{{\SetFigFont{8}{9.6}{\rmdefault}{\mddefault}{\updefault}{\color[rgb]{0,0,0}$\epsilon \tau = \valjp{\xi'}$}%
}}}}
\put(2098,-1687){\makebox(0,0)[lb]{\smash{{\SetFigFont{8}{9.6}{\rmdefault}{\mddefault}{\updefault}{\color[rgb]{0,0,0}${\mathfrak F}_-$}%
}}}}
\put(2109,-1802){\makebox(0,0)[lb]{\smash{{\SetFigFont{8}{9.6}{\rmdefault}{\mddefault}{\updefault}{\color[rgb]{0,0,0}elliptic $-$}%
}}}}
\put(2719,-652){\makebox(0,0)[lb]{\smash{{\SetFigFont{8}{9.6}{\rmdefault}{\mddefault}{\updefault}{\color[rgb]{0,0,0}$\tau \alpha_+ = \sigma m_+(x',\xi')\br$}%
}}}}
\put(2901,-929){\makebox(0,0)[lb]{\smash{{\SetFigFont{8}{9.6}{\rmdefault}{\mddefault}{\updefault}{\color[rgb]{0,0,0}$\tau \alpha_+ = \sigma_0 m_+(x',\xi')\br$}%
}}}}
\put(958,-341){\makebox(0,0)[lb]{\smash{{\SetFigFont{8}{9.6}{\rmdefault}{\mddefault}{\updefault}{\color[rgb]{0,0,0}non smooth}%
}}}}
\put(964,-234){\makebox(0,0)[lb]{\smash{{\SetFigFont{8}{9.6}{\rmdefault}{\mddefault}{\updefault}{\color[rgb]{0,0,0}roots }%
}}}}
\put(1760,176){\makebox(0,0)[lb]{\smash{{\SetFigFont{8}{9.6}{\rmdefault}{\mddefault}{\updefault}{\color[rgb]{0,0,0}$2 \tau \lambda_1^+ \norm{\p_{x'} \kappa}_{L^\infty}  = \lambda_0^+ |\xi'|$}%
}}}}
\put(2032,-32){\makebox(0,0)[lb]{\smash{{\SetFigFont{8}{9.6}{\rmdefault}{\mddefault}{\updefault}{\color[rgb]{0,0,0}$4 \tau \lambda_1^+ \norm{\p_{x'} \kappa}_{L^\infty}  = \lambda_0^+ |\xi'|$}%
}}}}
\put(1802,-517){\makebox(0,0)[lb]{\smash{{\SetFigFont{8}{9.6}{\rmdefault}{\mddefault}{\updefault}{\color[rgb]{0,0,0}${\mathfrak F}_+$}%
}}}}
\put(1813,-656){\makebox(0,0)[lb]{\smash{{\SetFigFont{8}{9.6}{\rmdefault}{\mddefault}{\updefault}{\color[rgb]{0,0,0}elliptic $+$}%
}}}}
\end{picture}%

 \caption{The 
  overlapping microlocal regions in the case of a convex weight function.}
  \label{fig: microlocal regions - convexification}
\end{center}
\end{figure}
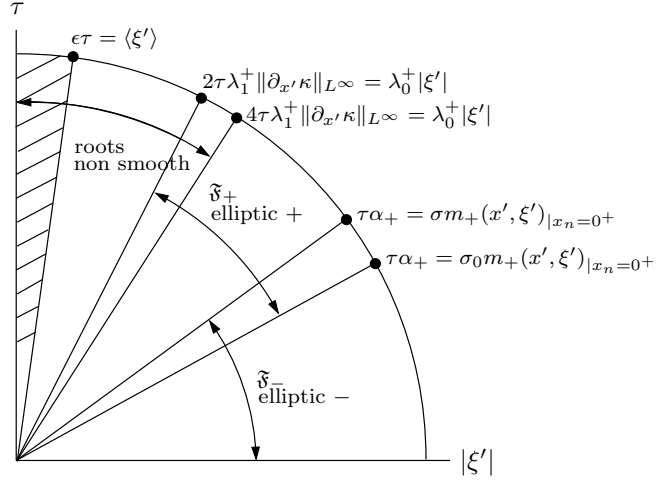

In this case, we define $\mathfrak{m}_{\pm}\in \S^1$ such that 
\begin{align*}
        \text{for $\val {\xi'}\ge 1$, }
  \mathfrak{m}_\pm(x,\xi')  = \Big( \frac{b_{jk}^\pm}{a_{nn}^\pm}
   (\xi_{j}+i\tau\p_{j}\kappa)(\xi_{k}+i\tau\p_{k}\kappa)\Big)^\hf, 
   \quad \mathfrak{m}_\pm(x,\xi') \geq C \valjp{\xi'}.
\end{align*}
Here we use the principal value of the square root function for complex numbers.

Introducing 
\begin{align*}
  &{\mathfrak e}_\pm = \tau \big( \p_{n}\Phi_{\pm} + S_\pm(x,\p_{x'}\kappa)\big)
 + \re {\mathfrak m}_\pm(x,\xi'),\\
  &{\mathfrak f}_\pm = \tau \big( \p_{n}\Phi_{\pm} + S_\pm(x,\p_{x'}\kappa)\big)
 - \re {\mathfrak m}_\pm(x,\xi')
\end{align*}
we set ${\mathfrak E}_\pm = \op{{\mathfrak e}_\pm}$ and ${\mathfrak F}_\pm = \op{{\mathfrak f}_\pm}$ and 
\begin{align*}
   &\P_{{\mathfrak E}\pm} = D_{n}+S_\pm(x,D') -\opw{\im {\mathfrak m}_\pm} + i{\mathfrak E}_\pm, \\
   &\P_{{\mathfrak F}\pm} = D_{n}+S_\pm(x,D') + \opw{\im {\mathfrak m}_\pm} + i {\mathfrak F}_\pm.
\end{align*}
Modulo the operator class $\Psi^1$, as in Section~\ref{sec: pseudo factor}, we may write 
\begin{equation*}
  \P_{+}\equiv \P_{{\mathfrak E}+} a_{nn}^+ \P_{{\mathfrak F}+}, 
  \qquad 
  \P_{-}\equiv \P_{{\mathfrak F}-} a_{nn}^- \P_{{\mathfrak E}-},
\end{equation*}

We keep the notation $m_\pm$ for the symbols that correspond to
the previous sections, \ie, if $\kappa$ vanishes: 
\begin{align*}
  m_{\pm}(x,\xi') = \Big( \frac{b_{jk}^\pm}{a_{nn}^\pm}
   \xi_{j}\xi_{k}\Big)^\hf, \quad |\xi'| \geq 1,
\end{align*}
 
As above, see \eqref{mainassu2}, we choose the weight function  such that 
the following property is fulfilled 
\begin{equation*}
  \frac{\alpha_{+}}{\alpha_{-}}>\sup_{x',\xi' \atop |\xi'|\geq 1} \frac{m_+(x',\xi') \br}
   {m_-(x',\xi')\bl}, 
  \qquad \alpha_\pm = \p_{x_n} {\varphi_\pm}_{|x_n=0^\pm},
\end{equation*}
and we let $\sigma >1$ be such that 
$$
\frac{\alpha_{+}}{\alpha_{-}}  =\sigma^2 \sup_{x',\xi' \atop |\xi'|\geq 1} \frac{m_+(x',\xi') \br}
   {m_-(x',\xi')\bl}.
$$
We also introduce $1 < \sigma_0 < \sigma$. As in Section~\ref{sec:
pseudo factor} we set $f_\pm = \tau \varphi_\pm' -m_\pm$
(compare with ${\mathfrak f}_\pm$ above).

We can choose $\alpha_{+}/\norm{\p_{x'}\kappa}_{L^\io}$  large enough so that
$$
\frac{\sigma m^+\br}{\alpha_{+}}
< \frac{\lambda_{0}^+ |\xi'|}{4\lambda_{1}^+\norm{\p_{x'}\kappa}_{L^\io}}
$$
and 
\begin{align}
        \label{eq: convex f +/- positive}
f_\pm \geq C \lambda, \quad \text{if}\  \tau\ge 
\val{\xi'}\frac{\lambda_{0}^+}{4\lambda_{1}^+\norm{\p_{x'}\kappa}_{L^\io}}
\ \  \text{for}\ |x_n|\ \text{\suff small}.
\end{align}

We may then consider the following cases.
\begin{enumerate}
\item When $\tau \alpha_+  \le \sigma m^+(x',\xi')\br$, then
arguing as for \eqref{cone01}-\eqref{eq: ineq cone01} we find $$\tau
(\alpha_- + \beta x_n) - m_-(x',\xi')\bl \leq -C \lambda,$$ if the
support of $v^-$ is chosen sufficiently small.  It follows that
${\mathfrak F}_{-}$ is elliptic negative, if
$\alpha_{+}/\norm{\kappa'}_{L^\io}$ is \suff small. In this
region we may thus argue as we did in Section~\ref{sec: f - negative}.

\item When $\dis \frac{\lambda_{0}^+ \val{\xi'}}{2\lambda_{1}^+\norm{\p_{x'}\kappa}_{L^\io}}
\ge  \tau\ge \frac{\sigma_{0} m_+(x',\xi')}{\alpha_{+}}$, the factorization is valid.
Arguing as in \eqref{cone00}-\eqref{eq: ineq cone00} we find that
$$\tau (\alpha_- + \beta x_n) - m_+(x',\xi') \leq C \lambda,$$
if the support
of $v^+$ is chosen sufficiently small.
It follows that
${\mathfrak F}_{+}$ is elliptic positive, if
$\alpha_{+}/\norm{\kappa'}_{L^\io}$ is \suff small.  In this
region we may thus argue as we did in Section~\ref{sec: f + positive}.

It is important to note that for $\beta$ large and
 $\norm{\kappa'}_{L^\infty}$ and $\norm{\kappa''}_{L^\infty}$
 \suff small the weight functions $\Phi_\pm$ satisfy the (necessary and
 sufficient) sub-ellipticity condition \eqref{eq: sub-ellipticity
 condition} with a loss of a half derivative.  Then the counterpart of
 Lemma~\ref{lemma: effective sub-ellipticiy condition} becomes, for
 $\norm{\kappa'}_{L^\infty}$ \suff small,
 $$|{\mathfrak f}_\pm| \leq \delta \lambda \quad \Longrightarrow \quad 
  C^{-1} \tau \leq |\xi'| \leq C \tau 
  \ \ \text{and}\ \  \{ \xi_n + s_\pm + \im (\mathfrak{m}_\pm),
   {\mathfrak f}_\pm\} \geq C' \lambda,$$
 for some $\delta>0$ chosen \suff small.
 This allows us to then obtain the same results as that of
 Lemma~\ref{lemma: increasing root} for the first-order factors
 $\P_{{\mathfrak F}\pm}$.

\item Finally we consider the region
$\dis  \tau\ge 
\val{\xi'}\frac{\lambda_{0}^+}{4\lambda_{1}^+\norm{\p_{x'}\kappa}_{L^\io}}$. There the roots are no longer smooth,
but we are well-inside an elliptic region; with perturbation argument,
we may in fact 
disregard the contribution of $\kappa$.

From \eqref{eq: conjugated op - convex weight} we may write 
\begin{align}
   \label{eq: conjugated op - convex weight2}
   \mathcal P_{\pm}\equiv \underbrace{a_{nn}^\pm \Big(
   \big[D_{n}+S_\pm(x,D')
   +i\tau \p_{n}\varphi_{\pm}
   \big]^2 +\frac{b_{jk}^\pm}{a_{nn}^\pm} 
  D_{j} D_{k}\Big)}_{P_{\pm}^0} + R_\pm, 
\end{align}
with $R_\pm = R_{1,\pm}(x,D',\tau) D_n + R_{2,\pm}(x,D',\tau)$, 
where $R_{j,\pm} \in \opw{\ST^j}$, $j=1,2$, that satisfy 
\begin{align}
   \label{eq: smallness remainder}
   \norm{R_{j,\pm}(x,D',\tau) u}
   \leq C \norm{\kappa'}_{L^\infty}
   \norm{u}_{L^2(\R; \H^{j})} 
\end{align}
The first term $P_{\pm}^0$ in \eqref{eq: conjugated op - convex weight2}
corresponds to the conjugated operator in the sections above, where
the weight function only depended on the $x_n$ variable.
This term can be factored in two pseudo-differential first-order terms:
\begin{align}
    \P_{+}^0\equiv \P_{E+} a_{nn}^+ \P_{F+}, 
  \qquad 
  \P_{-}^0\equiv \P_{F-} a_{nn}^- \P_{E-},
\end{align}
with the notation we introduced in Section~\ref{sec: pseudo factor}.
In this third region we have $f_\pm \geq C \lambda$
by \eqref{eq: convex f +/- positive}. Let $\chi_2 \in \ST^0$ be a symbol
that localizes in this region and set $\Xi_2 = \opw{\chi_2}$.

For $\norm{\kappa'}_{L^\infty}$ bounded with \eqref{eq: smallness
remainder2} we have
\begin{align}
        \label{eq: smallness remainder2}
        &\norm{H_\pm R_{1,\pm} D_n \Xi_2 v_\pm} 
        \lesssim \tau^k \norm{\kappa'}_{L^\infty} 
        \norm{H_\pm D_n \Xi_2 v_\pm}_{L^2(\R; \H^{1-k}} 
        + C(\kappa) \norm{H_\pm D_n v_\pm}, \\
        \label{eq: smallness remainder3}
        &\norm{H_\pm R_{2,\pm} D_n \Xi_2 v_\pm} 
        \lesssim \tau^k \norm{\kappa'}_{L^\infty} 
        \norm{H_\pm \Xi_2 v_\pm}_{L^2(\R; \H^{2-k}} 
        + C(\kappa) \norm{H_\pm v_\pm},
\end{align}
for $k=0$ or $\hf$. 

On the one hand, arguing as in Section~\ref{sec: f + positive} we have 
(see \eqref{eq: low freq est +})
\begin{multline}
  \label{eq: low freq est + convex}
  \norm{H_{+}\P_{+}^0 \Xi_2 v_{+}}
  +  \norm{H_+ v_+}
  \gtrsim
  \val{\V_{2,+}}_{\H^{\hf}}
  +\val{\Xi_2 {v_{+}}\br}_{\H^{3/2}}\\
  + \norm{H_{+}\Xi_2  v_{+}}_{L^2(\R;\H^2)} 
  + \norm{H_{+}\Xi_2 D_n  v_{+}}_{L^2(\R;\H^1)},
\end{multline}
 where $\V_{2,\pm}$ is given as in \eqref{eq: trace V}.

On the other hand, with Lemma~\ref{lemma: estimate PF- F- positive} we have
\begin{multline*}
    \norm{H_- \P_{-}^0 \Xi_2 v_{-}}_{L^2(\R;\H^{-k})} 
    + \norm{H_- v_{-}} +  \norm{H_- D_n v_{-}}\\
    + \val{\V_{2,-}+ia_{nn}^- M_{-}\Xi_2 {v_{-}}\bl}_{\H^{\hf-k}} 
    \gtrsim  
    \norm{H_- \P_{E-} \Xi_2  v_{-}}_{L^2(\R; \H^{1-k})},
  \end{multline*}
for $k=0$ or $\hf$, which gives
\begin{multline*}
    \norm{H_- \P_{-}^0 \Xi_2 v_{-}} 
    + \tau^k \norm{H_- v_{-}} + \tau^k \norm{H_- D_n v_{-}}\\
    + \tau^k \val{\V_{2,-}+ia_{nn}^- M_{-}\Xi_2 {v_{-}}\bl}_{\H^{\hf-k}} 
    \gtrsim  
    \tau^k \norm{H_- \P_{E-} \Xi_2  v_{-}}_{L^2(\R; \H^{1-k})}.
  \end{multline*}

Combined with Lemma~\ref{lemma: estimate PE+ PE-} we obtain 
\begin{multline}
        \label{eq: low freq est - convex}
    \norm{H_- \P_{-}^0 \Xi_2 v_{-}} 
    + \tau^k \Big(\norm{H_- v_{-}} +  \norm{H_- D_n v_{-}}
    + \val{\V_{2,-}}_{\H^{\hf-k}} 
    + \val{\Xi_2 {v_{-}}_{|x_n=0^-}}_{\H^{\frac{3}{2}-k}}\Big)\\
    \gtrsim \tau^k \norm{H_-  \Xi_2 v_{-}}_{L^2(\R;\H^{2-k})}
    + \tau^k \norm{H_+  \Xi_2 D_n v_{-}}_{L^2(\R;\H^{1-k})}
  \end{multline}
Now, from the transmission conditions \eqref{eq: microlocal trans
  conditions}--\eqref{eq: microlocal trans conditions 2}, by adding
$\varepsilon \eqref{eq: low freq est - convex} + \eqref{eq: low freq est + convex}$
we obtain, for $\varepsilon$ small, 
\begin{multline}
    \norm{H_{+}\P_{+}^0 \Xi_2 v_{+}}
    +\norm{H_- \P_{-}^0 \Xi_2 v_{-}} 
    + \tau^k\Big( \val{\theta_\varphi}_{\H^{\frac{3}{2}-k}}
    + \val{\Theta_\varphi}_{\H^{\hf-k}} + \val{v\br}_{\H^{\hf-k}}\Big)
    \\
    + \tau^k  \big( \norm{H_- v_{-}} +  \norm{H_- D_n v_{-}}
    \big)
    +  \norm{H_+ v_+}+  \norm{H_+ D_n v_{+}}
    \\
    \gtrsim 
    \tau^k \Big( \val{\Xi_2 D_n {v_-}_{|x_n=0^-}}_{\H^{\hf-k}}
    + \val{\Xi_2 D_n {v_+}_{|x_n=0^+}}_{\H^{\hf-k}}\\
    +\val{\Xi_2  {v_-}\bl}_{\H^{\frac{3}{2}-k}} 
    +  \val{\Xi_2  {v_+}\br}_{\H^{\frac{3}{2}-k}}  
    + \norm{ \Xi_2 v}_{L^2(\R;\H^{2-k})}\\
    + \norm{H_- \Xi_2 D_n v_{-}}_{L^2(\R;\H^{1-k})}
    + \norm{H_+  \Xi_2 D_n v_{+}}_{L^2(\R;\H^{1-k})}
    \Big).
  \end{multline}
 With~\eqref{eq: smallness remainder2}--\eqref{eq: smallness remainder3} we see that the same estimate
  holds for $\P_\pm$ in place of $\P_\pm^0$ for
  $\norm{\kappa'}_{L^\io}$ chosen sufficiently small.  This estimate
  is of the same quality as those obtained in the two other regions.
\end{enumerate}

Summing up, we have obtained three microlocal overlapping regions and
estimates in each of them. The three regions are illustrated in
Figure~\ref{fig: microlocal regions - convexification}.  As we did
above we make sure that the preliminary cut-off region of
Section~\ref{sec: prelim cut-off} does not interact with the
overlapping zones by choosing $\epsilon$ \suff small (see \eqref{eq:
Fourier multiplier} and Lemma~\ref{lemma: estimates for large tau}).

The overlap of the regions allows us to use a partition of unity
argument and we can conclude as in Section~\ref{sec:
microlocal patching}.

% Necessity of the assumption on the weight function
\section{Necessity of the geometric assumption on the weight function}
\label{sec: counter-example}

Considering the operator $\mathcal L_{\tau}$ given by
\eqref{1.toy007}, we may wonder about the relevance of conditions
\eqref{1.toybis} to derive a Carleman estimate.  In the simple model
and weight used here, it turns out that we can show that condition
\eqref{1.toybis} is necessary for an estimate to hold.  For simplicity,
we consider a {\em piecewise constant} case $c=H_{+}c_{+}+H_{-}c_{-}$ as in
Section~\ref{sec: sketch of the proof}. 

%%%%%%%%%%%%%%%%%%%%%%%%
% theorem              %
%%%%%%%%%%%%%%%%%%%%%%%%
\begin{theorem}
  \label{theorem: sharp condition on weight}
  Let us assume that
  \eqref{1.toy011} is violated, \ie,
  \begin{equation}
    \label{5.violation}
    \exists \xi_{0}'\in \R^{n-1}\setminus 0,\quad
    \frac{\alpha_{+}}{\alpha_{-}}<\frac{m_{+}(\xi'_{0})}{m_{-}(\xi'_{0})}.
  \end{equation}
  Then, for any neighborhood $V$ of the origin, $C>0$, and $\tau_0>0$,
  there exists
  \begin{align*}
    v = H_{+}v_{+}+H_{-}v_{-},  \quad v_{\pm}\in \mooc(\R^n),
  \end{align*}
  satisfying the transmission conditions \eqref{eq: trans
    cond1}--\eqref{eq: trans cond2} at $x_n=0$, and $\tau \geq
  \tau_0$, such that
  \begin{align*}
    \supp (v) \subset V \quad \text{and} \quad 
    C \norm{\mathcal L_{\tau} v}_{L^2(\R^{n-1}\times \R)} 
    \leq \norm{v}_{L^2(\R^{n-1}\times \R)},
  \end{align*}
  \end{theorem}
  To prove Theorem~\ref{theorem: sharp condition on weight} we wish to
  construct a function $v$, depending on the parameter $\tau$, such
  that $\norm{\mathcal L_{\tau}v}_{L^2}\ll \norm{v}_{L^2}$ as $\tau$
  becomes large. The existence of such a quasi-mode $v$ obviously
  ruins any hope to obtain a Carleman estimate for the operator
  $\mathcal L$ with a weight function satisfying
  \eqref{5.violation}. The remainder of this section is devoted to
  this construction.
  
  We set 
  \begin{align}
    \label{eq: definition Mtau}
    (\mathcal M_{\tau} u)(\xi',x_{n})
    &=H_{+}(x_{n})c_{n}^+\big(D_{n}+ i e_+\big)\big(D_{n}+if_+\big)u_{+}\\
    &\quad +
    H_{-}(x_{n})c_{n}^- \big(D_{n}+ie_-\big) \big(D_{n}+if_-\big)u_{-},
    \nonumber
  \end{align}
  that is, the action of the operator $\mathcal L_{\tau}$ given in
  \eqref{1.toy007} in the Fourier domain with respect to $x'$.
  Observe that the terms in each product commute here.  We start by
  constructing a quasi-mode for $\mathcal M_{\tau}$, \ie, functions
  $u_{\pm}(\xi',x_{n})$ compactly supported in the $x_{n}$ variable
  and in a conic neighborhood of $\xi'_{0}$ in the variable $\xi'$
  with $\norm{\mathcal M_{\tau}u}_{L^2}\ll \norm{u}_{L^2}$, so that
  $u$ is nearly an eigenvector of $\mathcal M_{\tau}$ for the
  eigenvalue 0. 

  \medskip
  Condition~\ref{5.violation} implies that there exists $\tau_{0}>0$
  such that
 $$
 \frac{m_{-}(\xi_{0}')}{\alpha_{-}}<\tau_{0}<\frac{m_{+}(\xi'_{0})}{\alpha_{+}}
 \Longrightarrow
 \tau_{0}\alpha_{+}-m_{+}(\xi'_{0})<0<\tau_{0}\alpha_{-}-m_{-}(\xi'_{0}).
 $$
 By homogeneity we may in fact
  choose $(\tau_0, \xi'_0)$ such that $\tau_0^2 + |\xi'_0|^2=1$.
 We have thus, using the notation in \eqref{1.toy007},
 \begin{equation*}
   f_{+}(x_{n}=0)=\tau\alpha_{+}-m_{+}(\xi')<0<f_{-}(x_{n}=0)
   =\tau \alpha_{-}-m_{-}(\xi'),
  \end{equation*}
  for $(\tau, \xi')$ in a conic neighborhood $\Gamma$ of $(\tau_{0},
  \xi'_{0})$ in $\R\times\R^{n-1}$. 
  Let $\chi_1 \in \mooc(\R)$,
  $0\leq \chi_1\leq 1$, with $\chi_1 \equiv 1$ in a \nhd of $0$, such that 
  $\supp(\psi) \subset \Gamma$ with
  \begin{align*}
    \psi(\tau,\xi') = \chi_1\Big(\frac{\tau}{(\tau^2 +
      |\xi'|^2)^\hf} - \tau_0\Big) \chi_1\Big(\Big| \frac{\xi}{(\tau^2
      + |\xi'|^2)^\hf} - \xi_0 \Big|\Big).
  \end{align*}
  We thus have 
  \begin{align*}
    f_+ (x_{n}=0) \leq -C \tau, \quad  C'\tau \leq f_{-}(x_{n}=0) \quad \text{in } \supp(\psi).  
  \end{align*}
  
  Let $(\tau,\xi') \in \supp (\psi)$.  We can solve the equations
  \begin{align*}
    &\bigl(D_{n}+if_+(x_n,\xi'))\bigr) q_{+}=0
    \quad\text{on $\R_{+}$}, 
    \quad f_+(x_n,\xi') = \tau \varphi'(x_n) - m_+(\xi') 
    = f_{+}(0) +\tau\beta x_{n},\\
    &\bigl(D_{n}+if_-(x_n,\xi'))\bigr) q_{-}=0
    \quad\text{on $\R_{-}$}, 
    \quad f_-(x_n,\xi') = \tau \varphi'(x_n) - m_-(\xi')
    = f_{-}(0) +\tau\beta x_{n},\\
    &\bigl(D_{n}+ie_-(x_n,\xi'))\bigr) \tilde{q}_{-}=0
    \quad\text{on $\R_{-}$}, 
    \quad e_-(x_n,\xi') = \tau \varphi'(x_n) + m_-(\xi')
    = e_{-}(0) +\tau\beta x_{n},
  \end{align*}
that is
  \begin{align*}
    &q_{+}(\xi',x_{n})
    = Q_+(\xi',x_{n})q_{+}(\xi',0), 
    \quad Q_+(\xi',x_{n}) = e^{x_{n}\big(f_{+}(0)+\frac{\tau\beta x_{n}}{2} \big)},\\ 
    &q_{-}(\xi',x_{n})
    = Q_-(\xi',x_{n})q_{-}(\xi',0), 
    \quad Q_-(\xi',x_{n}) = e^{x_{n}\big(f_{-}(0)+\frac{\tau\beta x_{n}}{2} \big)},\\
    &\tilde{q}_{-}(\xi',x_{n})
    = \tilde{Q}_-(\xi',x_{n})\tilde{q}_{-}(\xi',0), 
    \quad \tilde{Q}_-(\xi',x_{n}) 
    = e^{x_{n}\big(e_{-}(0)+\frac{\tau\beta x_{n}}{2} \big)}.  
  \end{align*}
  Since $f_{+}(0)<0$ a solution of the form of $q_+$ is a good idea on
  $x_{n}\ge 0$ as long as $\tau\beta x_{n}+2f_{+}(0)\le 0$, \ie,
  $x_{n}\le 2\val{f_{+}(0)}/\tau\beta$.  Similarly as $f_-(0)>0$ (\resp
  $e_-(0)>0$) a solution of the form of $q_-$ (\resp $\tilde{q}_-$) is a
  good idea on $x_{n}\leq 0$ as long as $\tau\beta x_{n}+2f_{-}(0)\geq
  0$ (\resp $\tau\beta x_{n}+2e_{-}(0)\geq 0$).  To secure this we
  introduce a cut-off function $\chi_{0}\in \mooc((-1,1);[0,1])$,
  equal to 1 on $[-\hf,\hf]$ and for $\gamma \geq 1$ we define
  \begin{align}
  \label{eq: solution droite}
    u_{+}(\xi',x_{n}) &= Q_+ (\xi',x_{n}) \psi(\tau,\xi')
    \chi_{0}\Big(\frac{\tau\beta \gamma x_{n}}{\val{f_{+}(0)}}\Big),
  \end{align}
  and 
  \begin{align}
     \label{eq: solution gauche}
    u_{-}(\xi',x_{n})&=a Q_- (\xi',x_{n})\psi(\tau,\xi')
    \chi_{0}\Big(\frac{\tau\beta \gamma x_{n}}{f_{-}(0)}\Big)
    +b \tilde{Q}_- (\xi',x_{n})\psi(\tau,\xi')
    \chi_{0}\Big(\frac{\tau\beta \gamma x_{n}}{e_{-}(0)}\Big),  
  \end{align}
  with $a,b \in \R$, and 
  $$u(\xi', x_{n})=H_{+}(x_{n})u_{+}(\xi', x_{n})+
  H_{-}(x_{n})u_{-}(\xi', x_{n})
  $$ 
  The factor $\gamma$ is introduced to control the size of the support
  in the $x_n$ direction.
  Observe that we can satisfy the transmission condition~\eqref{eq:
    trans cond1}--\eqref{eq: trans cond2} by choosing the coefficients
  $a$ and $b$.
  Transmission condition~\eqref{eq: trans cond1}
  implies
  \begin{align}
  \label{eq: trans cond sec 5}
    a+b =1.
  \end{align}
  Transmission condition~\eqref{eq: trans cond2} and the equations
  satisfied by $Q_+$, $Q_-$ and $\tilde{Q}_-$ imply
  \begin{align}
   \label{eq: trans cond 2 sec 5}
    c_+ m_+ = c_- (a-b)m_-.
  \end{align}
  In particular note that $a-b\geq 0$ which gives $a \geq \hf$.

  We have the following lemma.
  %%%%%%%%%%%%%%%%%%%%%%%%
  % sub-lemma            %
  %%%%%%%%%%%%%%%%%%%%%%%% 
  \begin{lem}
    \label{sub-lemma: estimate quasi-mode}
    For $\tau$ sufficiently large we have
    \begin{align*}
      \norm{\mathcal M_{\tau} u}_{L^2(\R^{n-1}\times \R)}^2
      \leq C (\gamma^2 + \tau^{2})  \gamma \tau^{n-1} e^{-C' \tau/\gamma} 
    \end{align*}
    and
    \begin{align*}
     \norm{u}_{L^2(\R^{n-1}\times \R)}^2
      \geq C \tau^{n-2} \big(1-e^{-C' \tau /\gamma}\big).
    \end{align*}
  \end{lem}
  See Section~\ref{proof: estimate quasi-mode} for a proof.
  
  We now introduce
  \begin{align*}
    v_\pm (x',x_n) = (2\pi)^{-(n-1)}\chi_0(|\tau^\hf x'|) \check{\hat{u}}_\pm(x',x_n) 
    = (2\pi)^{-(n-1)} \chi_0(|\tau^\hf x'|) \hat{u}_\pm(-x',x_n), 
  \end{align*}
  that is, a localized version of the inverse Fourier
  transform (in $x'$) of $u_\pm$. The functions $v_\pm$ are smooth and compactly
  supported in $\R^{n-1}_\pm\times \R$ and they satisfy  transmission
  conditions~\eqref{eq: trans cond1}--\eqref{eq: trans cond2}. We set
  $v(x', x_{n})=H_{+}(x_{n})v_{+}(x', x_{n})+ H_{-}(x_{n})v_{-}(x',
  x_{n})$. In fact we have the following estimates.
  %%%%%%%%%%%%%%%%%%%%%%%%
  % sub-lemma            %
  %%%%%%%%%%%%%%%%%%%%%%%% 
  \begin{lem}
    \label{sub-lemma: estimate quasi-mode 2}
    Let $N \in \N$. For $\tau$ sufficiently large we have
    \begin{align*}
      \norm{\mathcal L_{\tau} v}_{L^2(\R^{n-1}\times \R)}^2
      \leq C (\gamma^2 + \tau^{2})  \gamma \tau^{n-1} e^{-C' \tau/\gamma} 
      + C_{\gamma,N} \tau^{-N}
    \end{align*}
    and
    \begin{align*}
     \norm{v}_{L^2(\R^{n-1}\times \R)}^2
      \geq C \tau^{n-2} \big(1-e^{-C' \tau /\gamma}\big) - C_{\gamma,N} \tau^{-N}.
    \end{align*}
  \end{lem}
See Section~\ref{proof: estimate quasi-mode 2} for a proof.

We may now conclude the proof of Theorem~\ref{theorem: sharp condition
  on weight}. In fact, if $V$ is an arbitrary neighborhood of the
origin, we choose $\tau$ and $\gamma$ sufficiently large so that
$\supp(v) \subset V$. We then keep $\gamma$ fixed.  The estimates of
Lemma~\ref{sub-lemma: estimate quasi-mode 2} show that
\begin{align*}
  \norm{\mathcal L_{\tau} v}_{L^2(\R^{n-1}\times \R)} 
  \norm{v}_{L^2(\R^{n-1}\times \R)}^{-1} \ \ \mathop{\longrightarrow}_{\tau \to \infty} 
  \  \ 0.
\end{align*} 

%%%%%%%%%%%%%%%%%%%%%%%%
% remark               %
%%%%%%%%%%%%%%%%%%%%%%%%
\begin{rem}
  As opposed to the analogy we give at the beginning of
  Section~\ref{sec: explaining assumption}, the construction of this
  quasi-mode does not simply rely on one of the first-order factor.
  The transmission conditions are responsible for this fact.  The
  construction relies on the factor $D_n + i f_+$ in $x_n \geq
  0$, \ie, a one-dimensional space of solutions (see \eqref{eq:
  solution droite}), and on both factors $D_n + i f_-$ and $D_n + i
  e_-$ in $x_n \geq 0$, \ie, a two-dimensional space of solutions
  (see \eqref{eq: solution gauche}). See also \eqref{eq: trans cond sec 5} and \eqref{eq: trans cond 2 sec 5}.
\end{rem}

%% Appendices
\section{Appendix}
% A few facts on pseudo-differential operators

\subsection{A few facts on pseudo-differential operators}\label{app.sub.pseudo}

\subsubsection{Standard classes and Weyl quantization}
We define for $m\in \R$ the class of tangential symbols $\S^m$ as the smooth
functions on $\R^n\times\R^{n-1}$ such that, for all $(\alpha,
\beta)\in \N^n\times \N^{n-1}$,
\begin{equation}\label{004app}N_{\alpha\beta}(a)=
\sup_{(x,\xi')\in \R^n\times \R^{n-1}} \valjp{\xi'}^{-m+\val\beta} 
\val{(\p_{x}^\alpha\p_{\xi'}^\beta a)(x,\xi')}<\io,
\end{equation}
with $\valjp{\xi'}^2 = 1 + |\xi'|^2$.
The quantities on the \lhs above are called the semi-norms of the
symbol $a$.  For $a\in \S^m$, we define $ \op{a} $ as the operator
defined on $\mathscr S(\R^n)$ by
\begin{equation}
  \label{eq: def psido}
  (\op{a} u)(x',x_{n})= a(x,D') u (x',x_{n}) 
  =\int\limits_{\R^{n-1}} e^{i x'\cdot \xi'}
  a(x',x_{n},\xi')\hat u(\xi',x_{n}) d\xi'(2\pi)^{1-n},
\end{equation}
with $(x',x_{n})\in \R^{n-1}\times\R$, where $\hat u$ is the partial
Fourier transform of $u$ with respect to the variable
$x'$. 
For all $(k,s)\in \Z\times \R$ we have
\begin{equation}\label{5.quanti}
  \op{a}: H^k(\R_{x_{n}}; H^{s+m}(\R_{x'}^{n-1})) \to
  H^k(\R_{x_{n}}; H^{s}(\R_{x'}^{n-1}))\quad\text{continuously},
\end{equation}
and the norm of this mapping depends only on $\{N_{\alpha\beta}(a)\}_{\val\alpha+\val \beta\le \mu(k,s,m,n)}$,
where $\mu:\Z\times\R\times\R\times \N\rightarrow \N.$

We shall also use the {\em Weyl quantization}
of $a$ denoted by $\opw{a}$ and given by the formula
\begin{align}
  \label{eq: weyl quantization}
  (\opw{a} u)(x',x_{n})&=a^w(x,D') u(x',x_{n})\\
  &= \mathop{\iint}_{\R^{2n-2}} e^{i (x'-y')\cdot \xi'}
  a\big(\frac{x'+y'}{2},x_{n},\xi'\big)
  u(y',x_{n}) dy'd\xi'(2\pi)^{1-n}. \nonumber
\end{align}
Property \eqref{5.quanti} holds as well for $\opw{a}$.  A nice feature
of the Weyl quantization that we use in this article is the simple
relationship with adjoint operators with the formula
\begin{equation}
\label{eq: weyl adjoint}
\bigl(\opw{a}\bigr)^*=\opw{\bar a},
\end{equation}
so that for a real-valued symbol $a\in \S^m$ 
$(\opw{a})^*=\opw{a}$.
We have also for $a_{j}\in \S^{m_{j}}, j=1,2,$
\begin{align}
\opw{a_{1}}\opw{a_{2}}&=\opw{a_{1}\sharp a_{2}},
\quad a_{1}\sharp a_{2} \in \S^{m_{1}+m_{2}},
\end{align}
with, for any $N \in \N$, 
\begin{align}
  \label{eq: composition asymptotic series}
 \big(a_{1}\sharp a_{2}\big)(x,\xi) 
 - \ssum_{j < N} \big( i \sigma(D_{x'}, D_{\xi'}; D_{y'}, D_{\eta'})/2)^j
 a_1(x,\xi) a_2(y,\eta)/j!\big|_{(y,\eta) = (x,\xi)} \in \S^{m-N}, 
\end{align}
where $\sigma$ is the symplectic two-form, \ie, $\sigma (x,\xi;y, \eta) 
= y \cdot \xi - x \cdot \eta$.
In particular,
\begin{align}
\opw{a_{1}}\opw{a_{2}}&=\opw{a_{1}a_{2}}+\opw{r_{1}},
\quad r_{1}\in \S^{m_{1}+m_{2}-1},\\
\text{with }r_{1}&=\frac{1}{2i}\poi{a_{1}}{a_{2}}+r_{2},
\quad r_{2}\in\S^{m_{1}+m_{2}-2},\\
[\opw{a_{1}},\opw{a_{2}}]&=\opw{\frac{1}{i}\poi{a_{1}}{a_{2}}}+\opw{r_{3}},
\quad r_{3}\in \S^{m_{1}+m_{2}-3},
\end{align}
where $\poi{a_{1}}{a_{2}}$ is the Poisson bracket.
Moreover, for $b_{j}\in \S^{m_{j}}, j=1,2,$ both real-valued, we have
\begin{align}\label{5.morcom}
[\opw{b_{1}},i\opw{b_{2}}]&=\opw{\poi{b_{1}}{b_{2}}}+\opw{s_{3}},\quad s_{3}
\text{ real-valued}\in \S^{m_{1}+m_{2}-3}.
\end{align}
%%%%%%%%%%%%%%%%%%%%%%%%
% lemma                %
%%%%%%%%%%%%%%%%%%%%%%%%
\begin{lem}\label{5.lem.gaarde}
  Let $a\in \S^1$ such that  $a(x,\xi')\geq \mu
    \valjp{\xi'}$, with $\mu\geq 0$. Then there exists $C>0$ such that
  \begin{align*}
    \opw{a}+C\ge \mu\valjp{D'},\qquad (\opw{a})^2+C\ge  \mu^2\valjp{D'}^2.
  \end{align*}
\end{lem}
%%%% proof of lemma
\begin{proof}
  The first statement follows from the sharp G{\aa}rding
  inequality~\cite[Chap. 18.1 and 18.5]{Hoermander:V3} applied to the
  nonnegative first-order symbol $a(x,\xi')-\mu\valjp{\xi'}$; moreover
  $(\opw{a})^2=\opw{a^2}+\opw{r}$ with $r\in \S^0$, so that the
  Fefferman-Phong inequality~\cite[Chap. 18.5]{Hoermander:V3} applied
  to the second-order $a^2-\mu^2\valjp{\xi'}^2$ implies the result.
\end{proof}

\subsubsection{Semi-classsical pseudo-differential calculus with a large parameter}
\label{section: semi-classical calculus}
We let $\tau \in \R$ be such that $\tau \geq \tau_0 \geq 1$. We set
$\lambda^2 = 1 + \tau^2 + |\xi'|^2$.  We define for $m\in \R$ the
class of symbols $\ST^m$ as the smooth functions on
$\R^n\times\R^{n-1}$, depending on the parameter $\tau$, such that,
for all $(\alpha, \beta)\in \N^n\times \N^{n-1}$,
\begin{equation}
  \label{eq: seminorm tau}
  N_{\alpha\beta}(a)=
  \sup_{{(x,\xi')\in \R^n\times \R^{n-1}} \atop {\tau \geq \tau_0}} 
  \lambda^{-m+\val\beta} 
  \val{(\p_{x}^\alpha\p_{\xi'}^\beta a)(x,\xi',\tau)}<\io.
\end{equation}
Note that $\ST^0 \subset \S^0$.
The associated operators are defined by \eqref{eq: def psido}.  
We can introduce Sobolev spaces and Sobolev norms which are adapted to
the scaling large parameter $\tau$.   Let $s \in
\R$; we set
\begin{equation*}
   \norm{u}_{\H^s} := \norm{\Lambda^s u}_{L^2(\R^{n-1})}, \quad \text{with}\ 
   \Lambda^s := \op{\lambda^s}
\end{equation*}
and
\begin{equation*}
  \H^s= \H^s(\R^{n-1}) := \{ u \in {\mathscr S}'(\R^{n-1});\ 
  \norm{u}_{\H^s}< \infty\}.  
\end{equation*}
The space $\H^s$ is algebraically equal to the classical Sobolev space
$H^s(\R^{n-1})$, which norm is denoted by $\norm{.}_{H^s}$.  For $s
\geq 0$ note that we have
\begin{equation*}
  \norm{u}_{\H^s} \sim \tau^s \norm{u}_{L^2(\R^{n-1})} 
  + \norm{\valjp{D'}^s u}_{L^2(\R^{n-1})}.
\end{equation*}
 If $a \in \ST^m$ then, for all $(k,s)\in \Z\times \R$, we have
\begin{equation}
  \op{a}: H^k(\R_{x_{n}}; \H^{s+m}) \to
  H^k(\R_{x_{n}}; \H^{s}(\R_{x'}^{n-1}))\quad\text{continuously},
\end{equation}
and the norm of this mapping depends only on $\{N_{\alpha\beta}(a)\}_{\val\alpha+\val \beta\le \mu(k,s,m,n)}$,
where $\mu:\Z\times\R\times\R\times \N\rightarrow \N.$

For the calculus with a large parameter we shall also use the Weyl
quantization of \eqref{eq: weyl quantization}. All the formul\ae\xspace listed
in \eqref{eq: weyl adjoint}--\eqref{5.morcom} hold as well, with
$\S^m$ everywhere replaced by $\ST^m$. We shall often use the
G{\aa}rding inequality as stated in the following lemma.
%%%%%%%%%%%%%%%%%%%%%%%%
% lemma                %
%%%%%%%%%%%%%%%%%%%%%%%%
\begin{lem}
  \label{lem: postivity garding tau}
  Let $a\in \ST^m$ such that $\re a \geq C \lambda^m$. Then 
  \begin{align*}
    \re(\opw{a} u, u) \gtrsim \norm{u}_{L^2(\R; \H^{\frac{m}{2}})}^2,
  \end{align*}
  for $\tau$ \suff large.
\end{lem} 
%%%% proof of lemma
\begin{proof}
  The proof follows from the Sharp G{\aa}rding
  inequality~\cite[Chap. 18.1 and 18.5]{Hoermander:V3} applied to the
  nonnegative symbol $a - C \lambda^m$.
\end{proof}

For technical reasons we shall often need the following result.
%%%%%%%%%%%%%%%%%%%%%%%%
% lemma                %
%%%%%%%%%%%%%%%%%%%%%%%%
\begin{lem}
  \label{lemma: mixed composition}
  Let $m,m' \in \R$ and $a_1(x,\xi') \in \S^m$ and $a_2(x,\xi', \tau)
  \in \ST^{m'}$ such that the essential support of $a_2$
  is contained in a  region where $\valjp{\xi'} \gtrsim \tau$. Then 
  $$\opw{a_1}\opw{a_2} = \opw{b_1}, \quad \opw{a_2} \opw{a_1} = \opw{b_2},$$
  with $b_1,b_2 \in \ST^{m+m'}$. Moreover the 
  asympotic series of \eqref{eq: composition asymptotic series}
  is also valid for these cases (with $\S^m$ replaced by $\ST^m$).
\end{lem}

\begin{proof}
  As the essential support is invariant when we change quantization,  we
  may  simply use the standard quatization in the proof. With $a_1$
  and $a_2$ satisfying the assumption listed above we thus consider $\op{a_1}
  \op{a_2}$. For fixed $\tau$ the standard composition formula applies
  and we have (see \cite[Section 18.1]{Hoermander:V3} or \cite{AG:91})
  \begin{align*}
    \big(a_1 \circ a_2\big) (x,\xi') = (2 \pi)^{1-n} \iint e^{-i y' \cdot \eta'}
  a_1 (x, \xi' - \eta') a_2 (x'-y',x_n, \xi') d y' d \eta'.
  \end{align*}
  Properties of oscillatory integrals (see \eg \cite[Appendices
  I.8.1 and I.8.2]{AG:91}) give, for some $k \in \N$,
  \begin{align*}
    \val{\big(a_1 \circ a_2\big) (x,\xi')} 
    \leq C  \sup_{|\alpha| + |\beta| \leq k \atop {(y',\eta') \in R^{2n-2}}} 
    \valjp{(y',\eta')}^{-|m|}
    \val{\p_{y'}^\alpha \p_{\eta'}^\beta 
      a_1 (x, \xi' - \eta') a_2 (x'-y',x_n, \xi') },
  \end{align*}
  In a region $\valjp{\xi'} \gtrsim \tau$ that contains the essential
  support of $a_2$ we have $\valjp{\xi'} \sim \lambda$.
  With the so-called Peetre inequality, we thus obtain
  \begin{align*}
    \val{\big(a_1 \circ a_2\big) (x,\xi')} 
    \lesssim \valjp{\eta'}^{-|m|} \valjp{\xi'-\eta'}^m \lambda^{m'}
    \lesssim \valjp{\xi'}^m \lambda^{m'} \lesssim  \lambda^{m+m'}.
  \end{align*}
  In a region $\valjp{\xi'} \lesssim \tau$ outside of the essential
  support of $a_2$ we find, for any $\ell \in \N$,
  \begin{align*}
    \val{\big(a_1 \circ a_2\big) (x,\xi')} 
    \lesssim \valjp{\eta'}^{-|m|} \valjp{\xi'-\eta'}^m \lambda^{-\ell}
    \lesssim \valjp{\xi'}^m \lambda^{-\ell} \lesssim \lambda^{m-\ell}.
  \end{align*}
  In the whole phase space we thus have $\val{\big(a_1 \circ a_2\big)
    (x,\xi')} \lesssim \lambda^{m+m'}$. The 
  estimation of $\val{\p_x^\alpha \p_{\xi'}^\beta \big(a_1 \circ
    a_2\big)(x,\xi')}$ can be done similarly to give
  \begin{align*}
    \val{\p_x^\alpha \p_\xi^\beta \big(a_1 \circ
    a_2\big)(x,\xi')}
  \lesssim \lambda^{m+m' -|\beta|}.
  \end{align*}
  Hence $a_1 \circ a_2 \in \ST^{m+m'}$. Moreover, we also
  obtain the asymptotic series (following the references
  cited above)
  \begin{align*}
    \big(a_1 \circ a_2\big)(x,\xi') 
    - \ssum_{j < N} \big( i D_\xi\cdot D_y)^j
    a_1(x,\xi) a_2(y,\eta)/j!\big|_{(y,\eta) = (x,\xi)} \in \ST^{m+m'-N},
  \end{align*}
  where each term is respectively in $\ST^{m+m'-j}$ be the
  arguments given above.  From this series the corresponding 
  Weyl-quantization series follows.

  For the second result, considering the adjoint operator $\big( \op{a_2}
  \op{a_1}\big)^\ast$ yields a composition of operators as in the
  first case. The second result thus follows from the first one.
  \end{proof}

\begin{rem}
  The symbol class and calculus we have introduced in this section can
  be written as $\ST^m = \S(\lambda^m,g)$ in the sense of the
  Weyl-H\"ormander calculus \cite[Sec 18.4--18.6]{Hoermander:V3} with
  the phase-space metric $g = |dx|^2 + |d \xi|^2 / \lambda^2$. 
\end{rem}

% Proofs of some intermediate results
\subsection{Proofs of some intermediate results}

\subsubsection{Proof of Lemma~\ref{lemma: effective sub-ellipticiy condition}}
\label{proof: lemma: effective sub-ellipticiy condition}

For simplicity we remove the $\pm$ notation here.
We first prove that 
there exist $C>0$ and $\eta>0$ such that 
\begin{align}
  \label{eq: sub-ellipticity with positive constants}
  |q_2| \leq \eta \tau^2\ \
  \text{and}\ \ |q_1|\leq \eta \tau^2
  \ \ \Longrightarrow\ \ \{ q_2, q_1\}\geq C \tau^3.
\end{align}

 We set
\begin{align*} 
  \tilde{q}_2 =(\xi_n +s)^2 + \frac{b_{jk}}{a_{n n}}\xi_j \xi_k 
  -  (\varphi')^2, 
  \qquad \tilde{q}_1 = \varphi' (\xi_n +s).
\end{align*}
We have $q_j(x,\xi) = \tau^2 \tilde{q}_j (x,\xi/\tau)$.  Observe next
that we have $\{ q_2, q_1\}(x,\xi) = \tau^3 \{ \tilde{q}_2,
\tilde{q}_1\} (x,\xi/\tau)$.  We thus have $\tilde{q}_2 =0$ and
$\tilde{q}_1=0 \ \Rightarrow \ \{ \tilde{q}_2, \tilde{q}_1\}>0$.  As
$\tilde{q}_2(x,\xi)=0$ and $\tilde{q}_1(x,\xi)=0$ yields a compact set
for $(x,\xi)$ (recall the $x$ lays in a compact set $K$ here), for
some $C>0$, we have
  \begin{align*}
    \tilde{q}_2 =0\ \text{and}\ \tilde{q}_1=0
    \quad \Longrightarrow \quad \{ \tilde{q}_2, \tilde{q}_1\}>C.
  \end{align*}
  This remain true locally, \ie, for some $C'>0$ and $\eta>0$,
  \begin{align*}
    |\tilde{q}_2| \leq \eta \ \text{and}\ |\tilde{q}_1|\leq \eta
    \quad \Longrightarrow \quad \{ \tilde{q}_2, \tilde{q}_1\}>C'.
  \end{align*}
  Then \eqref{eq: sub-ellipticity with positive constants} follows.

  We note that $q_2^\pm=0$ and $q_1^\pm=0$ implies $\tau \sim
|\xi'|$. Hence, for $\tau$ sufficiently large we have \eqref{eq: q2 for
  xi' large}. We thus obtain
\begin{align*}
  q_2^\pm =0\ \text{and}\ q_1^\pm=0 \quad \Leftrightarrow \quad 
  \xi_n+s_\pm =0\ \text{and}\  \tau \varphi_\pm' = m_\pm.
\end{align*}

Let us assume that $|f| \leq \delta \lambda$ with $\delta$ small
and $\lambda^2 = 1 + \tau^2 + |\xi'|^2$.  Then
\begin{align}
  \label{eq: equiv tau xi}
  \tau \lesssim |\xi'| \lesssim \tau.
\end{align}
 We set $\xi_n =- s$, \ie, we choose $q_1=0$.  A direct
 computation yields
\begin{align*}
  \{ q_2, q_1\} = \tau  e \varphi' 
  \{ \xi_n + s, f\} + \tau f \varphi' \{ \xi_n + s, e\} 
  \quad\text{if} \ \ \xi_n+s=0.
\end{align*}
With \eqref{eq: q2 for xi' large} we have $|q_2| \leq C \delta
\tau^2$. For $\delta$ small, by \eqref{eq: sub-ellipticity with
  positive constants} we have $\{ q_2 ,q_1\} \geq C \tau^3$.  Since $f
\tau \varphi' \{ \xi_n + s, e\} \leq C \delta \tau^3$ we obtain $ e
\tau \varphi' \{ \xi_n + s, f\} \geq C \tau^3$, with $C>0$, for
$\delta$ sufficiently small. With~\eqref{eq: equiv tau xi} we have
$\tau \lesssim e \lesssim \tau$ and the result follows.  \hfill
$\blacksquare$ \endproof

\subsubsection{Proof of Lemma~\ref{lemma: estimates for large tau}}
\label{proof: estimates for large tau}

We set $s=2\ell+1$ and $\omega_1 = \op{\psi_\epsilon} \omega$. 
We write 
\begin{align*}
  &2 \re ( \P_{F+} \omega_1, i H_+ \tau^s \omega_1)\\
  &\qquad = (i[D_n,H_+] \omega_1, \tau^s \omega_1) 
  + 2 (F_+ \omega_1, H_+ \tau^s \omega_1)\\
  &\qquad= \tau^s \val{{\omega_1}\br}_{L^2(\R^{n-1})}^2
  + 2 (\tau^{s+1} \varphi' \omega_1, H_+ \omega_1)
  - 2 (\tau^s M_+ \omega_1, H_+ \omega_1)\\
  &\qquad \geq \tau^s \val{{\omega_1}\br}_{L^2(\R^{n-1})}^2 
  + 2 (\tau^{s+1} C_0 \omega_1, H_+ \ \omega_1)
  - 2 C_1\tau^s  \norm{H_+\omega_1}_{L^2(\R;H^\hf(\R^{n-1}))}^2, 
\end{align*}
by \eqref{eq: bornitude M}. 
We have
\begin{align*}
  &2 (\tau^{s+1} C_0 \omega_1, H_+ \ \omega_1)
  - 2 C_1 \tau^s \norm{H_+ \omega_1}_{L^2(\R;H^\hf(\R^{n-1}))}^2\\
  &\qquad 
  =2 \tau^s (2 \pi)^{1-n} \int\limits_0^\infty\int\limits_{\R^{n-1}} 
  \big(C_0\tau - C_1 \valjp{\xi'}\big)   
  |\psi_\epsilon(\tau,\xi') \hat{\omega}(\xi', x_n)|^2    d \xi' d x_n
\end{align*}
As $\tau \geq C \valjp{\xi}/\epsilon$ in $\supp(\psi_\epsilon)$, for
$\epsilon$ \suff small we have 
\begin{align*}
  &2 (\tau^{s+1} C_0 \omega_1, H_+ \ \omega_1)
  - 2 C_1 \tau^s \norm{H_+ \omega_1}_{L^2(\R;H^\hf(\R^{n-1}))}^2\\
  &\qquad \gtrsim \int\limits_0^\infty
  \int\limits_{\R^{n-1}} \lambda^{s+1} |\psi_\epsilon(\tau,\xi')
  \hat{\omega}(\xi', x_n)|^2 d \xi' d x_n 
  \gtrsim \norm{H_+ \omega_1}_{L^2(\R;\H^{\ell+1})}^2.
\end{align*}
Similarly we find $\tau^s \val{{\omega_1}\br}_{L^2(\R^{n-1})}^2
\gtrsim \val{{\omega_1}\br}_{\H^{\ell+\hf}}^2$.  The result for
$\P_{E+}$ follows from the Young inequality.  The proof is identical for
$\P_{F+}$.

On the other side of the interface we write
\begin{align*}
  &2 \re (H_- \P_{F-} \omega_1, i H_- \tau^s \omega_1)\\
  &\qquad = (i[D_n,H_-] \omega_1, \tau^s \omega_1) 
  + 2 (F_- \omega_1, H_- \tau^s \omega_1)\\
  &\qquad= - \tau^s \val{{\omega_1}\bl}_{L^2(\R^{n-1})}^2
  + 2 (\tau^{s+1} \varphi' \omega_1, H_- \omega_1)
  - 2 (\tau^s M_- \omega_1, H_- \omega_1),
\end{align*}
which yields a boundary contribution with the opposite sign. 
\hfill \qedsymbol \endproof

\subsubsection{Proof of Lemma~\ref{sub-lemma: estimate quasi-mode}}
\label{proof: estimate quasi-mode}
Let $(\tau,\xi') \in \supp(\psi)$. We choose $\tau$ sufficiently large
so that, through $\supp(\psi)$, $|\xi'|$ is itself sufficiently large,
so as to have the symbol $m_\pm$  homogeneous --see~\eqref{eq: m pm for xi' large}.
    
    We set 
    \begin{align*}
    &y_{+}(\xi',x_{n})= Q_+ (\xi',x_{n}) 
    \chi_{0}\Big(\frac{\tau\beta \gamma x_{n}}{\val{f_{+}(0)}}\Big),
    \\
    &y_{-}(\xi',x_{n})= a Q_- (\xi',x_{n})
    \chi_{0}\Big(\frac{\tau\beta \gamma x_{n}}{f_{-}(0)}\Big)
    +b \tilde{Q}_- (\xi',x_{n})
    \chi_{0}\Big(\frac{\tau\beta \gamma x_{n}}{e_{-}(0)}\Big).
  \end{align*}
  
  On the one hand we have $ i(D_{n}+i f_{+})y_{+} = \frac{\tau\beta
    \gamma}{\val{f_{+}(0)}} Q_+ (\xi',x_{n})
  \chi'_{0}\Big(\frac{\tau\beta \gamma x_{n}}{\val{f_{+}(0)}}\Big), $ and
  \begin{align*}
    (\mathcal M_{\tau}y_+)(\xi',x_{n})
    &= 2 \tau\beta\gamma c_+ m_+ \frac{Q_+ (\xi',x_{n})}{\val{f_{+}(0)}} 
  \chi'_{0}\Big(\frac{\tau\beta \gamma x_{n}}{\val{f_{+}(0)}}\Big)\\
  &\quad - (\tau\beta \gamma)^2 c_+ \frac{Q_+ (\xi',x_{n})}{\val{f_{+}(0)}^2}
   \chi''_{0}\Big(\frac{\tau\beta \gamma x_{n}}{\val{f_{+}(0)}}\Big),
  \end{align*}
  as $D_n + i e_+ = D_n + i (f_+ + 2 m_+)$,
  so that 
  \begin{align*}
    \int\limits_{0}^{+\io}\val{(\mathcal M_{\tau}y_+)(\xi',x_{n})}^2 d x_{n}
    &\leq 8 c_+^2 m_+^2 \Big(\frac{\tau\beta\gamma}{f_{+}(0)}\Big)^2  
    \int\limits_{0}^{+\io} 
    \chi'_{0}\Big(\frac{\tau\beta \gamma x_{n}}{\val{f_{+}(0)}}\Big)^2
    e^{x_{n}(2 f_{+}(0)+\tau\beta x_{n})} dx_{n}
    \\
    &\quad + 
    2 c_+^2 \Big(\frac{\tau\beta \gamma}{f_{+}(0)}\Big)^4
    \int\limits_{0}^{+\io} 
    \chi''_{0}\Big(\frac{\tau\beta \gamma x_{n}}{\val{f_{+}(0)}}\Big)^2
    e^{x_{n}(2 f_{+}(0)+\tau\beta x_{n})} dx_{n}.
  \end{align*}

  On the support of $\chi_0^{(j)}\Big(\frac{\tau\beta \gamma
    x_{n}}{|f_{+}(0)|}\Big)$, $j=1,2$, we have
  $|f_+(0)|/(2\tau\beta\gamma)\leq x_n \leq |f_+(0)|/(\tau\beta \gamma)$ and in
  particular $2 f_{+}(0) + \tau\beta \gamma x_{n}\le -\val{f_{+}(0)}$ and
  which gives
  \begin{align*}
    &\int\limits_{0}^{+\io}\val{(\mathcal M_{\tau}y_+)(\xi',x_{n})}^2 d x_{n}
    \\
    & \quad \leq  c_+^2 \Big(\frac{\tau\beta \gamma}{f_{+}(0)}\Big)^2  
    \bigg(
    8 m_+^2  \norm{\chi'_{0}}^2_{L^\io} 
    + 2 \Big(\frac{\tau\beta \gamma}{f_{+}(0)}\Big)^2 \norm{\chi''_{0}}^2_{L^\io} 
    \bigg)
    \!\! \int\limits_{\frac{|f_+(0)|}{2\tau \beta \gamma} \leq x_n \leq \frac{|f_+(0)|}{\tau \beta \gamma}}
    \!\!\!\!\!\! e^{-|f_+(0)|x_n} d x_n\\
    & \quad \leq  c_+^2 \frac{\tau\beta \gamma }{|f_{+}(0)|}  
    \bigg(
    4 m_+^2  \norm{\chi'_{0}}^2_{L^\io} 
    + \Big(\frac{\tau\beta\gamma}{f_{+}(0)}\Big)^2  \norm{\chi''_{0}}^2_{L^\io}  
    \bigg)
    e^{-\frac{f_+(0)^2}{2\tau \beta \gamma}}.
  \end{align*}
  Similarly, we have 
  \begin{align*}
    (\mathcal M_{\tau}y_-)(\xi',x_{n})
    &= 2 \tau \beta \gamma c_- m_- 
    \bigg(
    \frac{Q_-(\xi',x_{n})}{f_-(0)} \chi'\Big( \frac{\tau \beta\gamma}{f_-(0)}\Big)
    - \frac{\tilde{Q}_-(\xi',x_{n})}{e_-(0)} \chi'\Big( \frac{\tau \beta \gamma}{e_-(0)}\Big)
    \bigg)\\
    &\quad - c_- (\tau \beta \gamma)^2 
    \bigg(
    \frac{Q_-(\xi',x_{n})}{f_-(0)^2} \chi''\Big( \frac{\tau \beta \gamma}{f_-(0)}\Big)
    + \frac{\tilde{Q}_-(\xi',x_{n})}{e_-(0)^2} \chi''\Big( \frac{\tau \beta \gamma}{e_-(0)}\Big)
    \bigg),
  \end{align*}
  and because of the support of $\chi_0^{(j)}\Big(\frac{\tau\beta \gamma
    x_{n}}{f_{-}(0)}\Big)$, \resp $\chi_0^{(j)}\Big(\frac{\tau\beta \gamma
    x_{n}}{e_{-}(0)}\Big)$,  $j=1,2$, for $x_n \leq 0$, we obtain 
  \begin{align*}
    \int\limits_{-\io}^{0}\val{(\mathcal M_{\tau}y_-)(\xi',x_{n})}^2 d x_{n}
    &\leq 
    2 c_-^2 \frac{\tau\beta \gamma}{f_{-}(0)}  
    \bigg(
    4 m_-^2  \norm{\chi'_{0}}^2_{L^\io} 
    + \norm{\chi''_{0}}^2_{L^\io} \Big(\frac{\tau\beta \gamma}{f_{-}(0)}\Big)^2  
    \bigg)
    e^{-\frac{f_-(0)^2}{2\tau \beta \gamma}}\\
    & \quad + 
    2 c_-^2 \frac{\tau\beta \gamma}{e_{-}(0)}  
    \bigg(
    4 m_-^2  \norm{\chi'_{0}}^2_{L^\io} 
    + \norm{\chi''_{0}}^2_{L^\io} \Big(\frac{\tau\beta \gamma}{e_{-}(0)}\Big)^2  
    \bigg)
    e^{-\frac{e_-(0)^2}{2\tau \beta \gamma}}.
  \end{align*}
  
  Now we have $(\mathcal M_{\tau}u) (\xi',x_{n}) = \psi(\tau,\xi')
  (\mathcal M_{\tau}y) (\xi',x_{n})$.
  As $|\xi'| \sim \tau$ in $\supp(\psi)$ we obtain
 \begin{align*}
   \norm{\mathcal M_{\tau}u}_{L^2(\R^{n-1}\times \R)}^2
   \leq C ( \gamma^2 + \tau^2)\gamma e^{-C'\tau/\gamma} 
   \int_{\R^{n-1}} \psi(\tau,\xi')^2 d \xi'.
 \end{align*}
 With the change of variable $\xi' = \tau \eta$ we find
 \begin{align}
   \label{eq: norme L2 psi}
   \int_{\R^{n-1}} \psi(\tau,\xi')^2 d \xi' = C \tau^{n-1}, 
 \end{align}
 which gives the first result.

On the other hand observe now that 
\begin{align*}
  \norm{y_+}{L^2(\R_+)}^2 &= 
  \int\limits_{0}^{+\io} Q_+(\xi',x_{n})^2 
  \chi_{0}\Big(\frac{\tau\beta \gamma x_{n}}{\val{f_{+}(0)}}\Big)^2 d x_{n}\\
  & \geq \!\!\!
   \int\limits_{0\le \frac{\tau\beta \gamma x_{n}}{\val{f_{+}(0)}}\le \hf}
    \!\!\!\!\!\! e^{x_{n}(2 f_{+}(0)+\tau\beta \gamma x_{n})} dx_{n}
   =\frac{\val{f_{+}(0)}}{\tau \beta \gamma} 
   \int\limits_{0}^{\hf}
   e^{2t\frac{\val{f_{+}(0)}}{\tau \beta \gamma}(f_{+}(0)+
     t\frac{\val{f_{+}(0)}}{2}
     )} dt
   \\
   &\ge \frac{\val{f_{+}(0)}}{\tau\beta \gamma} \int\limits_{0}^{\hf}
   e^{-2t\frac{\val{f_{+}(0)}^2}{\tau\beta \gamma}} dt
   =\frac1{2\val{f_{+}(0)}}\Big(1-e^{-\frac{\val{f_{+}(0)}^2}{\tau\beta \gamma}}\Big).
\end{align*}
We also have 
\begin{align*}
  \norm{y_-}_{L^2(\R_-)}^2 
  &= 
  \int\limits_{-\io}^{0} 
  \bigg( a Q_- (\xi',x_{n})
  \chi_{0}\Big(\frac{\tau\beta \gamma x_{n}}{f_{-}(0)}\Big)
  +b \tilde{Q}_- (\xi',x_{n})
  \chi_{0}\Big(\frac{\tau\beta \gamma x_{n}}{e_{-}(0)}\Big)\bigg)^2 
  d x_{n}\\
  &\geq 
  \!\!\!
   \int\limits_{-\hf\le \frac{\tau\beta \gamma x_{n}}{\val{f_{+}(0)}}\le 0}
    \!\!\!\!\!\!
    e^{x_{n}(2 f_{-}(0)+\tau\beta \gamma x_{n})}
    \Big( a + b e^{x_n (e_-(0) - f_-(0))} \Big)^2 d x_n,
\end{align*}  
and as $e_-(0) - f_-(0) = 2 m_-\geq 0$, $a+b=1$ and $a\geq \hf$, we have $a +
b e^{x_n (e_-(0) - f_-(0))}\geq \hf$, and thus obtain
\begin{align*}
  \norm{y_-}_{L^2(\R_-)}^2 \geq \frac{1}{4} \!\!\!
   \int\limits_{-\hf\le \frac{\tau\beta \gamma x_{n}}{\val{f_{+}(0)}}\le 0}
    \!\!\!\!\!\!
    e^{x_{n}(2 f_{-}(0)+\tau\beta \gamma x_{n})} d x_n
    \geq \frac1{8 f_{-}(0)}\Big(1-e^{-\frac{\val{f_{-}(0)}^2}{\tau\beta \gamma}}\Big),
\end{align*}  
arguing as above.
As a result, using \eqref{eq: norme L2 psi}, we have
\begin{align*}
\norm{u}_{L^2(\R^{n-1}\times \R)}^2
\geq C \tau^{n-2} \big(1-e^{-C' \tau /\gamma}\big) .
\end{align*}
\hfill $\blacksquare$ \endproof

\subsubsection{Proof of Lemma~\ref{sub-lemma: estimate quasi-mode 2}}
\label{proof: estimate quasi-mode 2}
We start with the second result. We set 
$$
z_+ = \big( 1-  \chi_0(|\tau^\hf x'|)\big) \check{\hat{u}}_+(x',x_n),
\quad  \text{ for } x_n \geq 0.
$$
We shall prove that for all $N\in\N$ we have
$\norm{z_+}_{L^2(\R^{n-1}\times \R_+} \leq C_{\gamma,N} \tau^{-N}$.

  From the definition of $\chi_0$ we find
  \begin{align*}
    \norm{z_+}_{L^2(\R^{n-1}\times \R_+}^2 
    \leq \int\limits_{|\tau^\hf x'| \geq \hf } 
    \int\limits_{\R_+} |\hat{u}_+(x',x_n)|^2 d x' d x_n
  \end{align*}
  Recalling the definition of $u_+$ and performing the change of
  variable $\xi' = \tau \eta$ we obtain
  \begin{align*}
    \hat{u}_+(x',x_n) = \tau^{n-1} \int\limits_{\R^{n-1}} 
    e^{i \tau \phi} \tilde{\psi}(\eta)
    \chi_{0}\Big(\frac{\beta \gamma x_{n}}{\val{\tilde{f}_{+}(\eta)}}\Big)
    d \eta,
  \end{align*}
  where  the complex phase
  function is given by
  $$
  \phi =  - x' \cdot \eta 
  -i x_n \big(\tilde{f}_+ (\eta) + \frac{\beta x_n }{2}\big),
  \quad \text{with} \ \ \tilde{f}_+ (\eta) = \alpha_+ - m_+(\eta),
  $$
  and 
  \begin{align*}
    \tilde{\psi}(\eta) = \chi_1\Big(\frac{1}{(1 +
      |\eta|^2)^\hf} - \tau_0\Big) \chi_1\Big(\Big| \frac{\eta}{(1
      + |\eta|^2)^\hf} - \xi_0 \Big|\Big).
  \end{align*}
  Here $\tau$ is chosen sufficiently large so that $m_+$ is
  homogeneous.  Observe that $\tilde{\psi}$ has a compact support
  independent of $\tau$ and that $ \tilde{f}_+ (\eta) +
  \frac{\beta x_n }{2} \leq -C <0$ in the support of the integrand.

  We place ourselves in the neighborhood of a point $x'$ such that
  $|\tau^\hf x'| \geq \hf$. Up to a permutation of the variables we
  may assume that $|\tau^\hf x_1| \geq C$.  We then introduce the
  following differential operator
  \begin{align*}
    L = \tau^{-1} \frac{\p_{\eta_1}}{- i x_1 - x_n \p_{\eta_1} m_+(\eta)},
  \end{align*}
  that satisfies $L e^{i \tau \phi} =  e^{i \tau \phi}$. We thus have
  \begin{align*}
    \hat{u}_+(x',x_n) = \tau^{n-1} \int\limits_{\R^{n-1}} e^{i \tau \phi} 
    (L^t)^N
    \bigg(\tilde{\psi}(\eta)
    \chi_{0}\Big(\frac{\beta \gamma x_{n}}{\val{\tilde{f}_{+}(\eta)}}\Big)
    \bigg)
    d \eta,
  \end{align*}
  and we find 
  \begin{align*}
    |\hat{u}_+(x',x_n)| \leq C_N \frac{\tau^{n-1} \gamma^N}{|\tau x_1|^N} 
    e^{- C \tau x_n}. 
  \end{align*}
  More generally for $|\tau^\hf x'| \geq \hf$ we have
  \begin{align*}
    |\hat{u}_+(x',x_n)| \leq C_N \frac{\tau^{n-1} \gamma^N}{|\tau x'|^N} 
    e^{- C \tau x_n}. 
  \end{align*} 
  Then we obtain 
  \begin{align*}
    &\int\limits_{|\tau^\hf x'| \geq \hf } 
    \int\limits_{\R_+} |\hat{u}_+(x',x_n)|^2 d x' d x_n\\
    &\qquad \leq C_{N+n} \gamma^{N+n} \tau^{n-1} 
    \Big( \int\limits_{|\tau^\hf x'| \geq \hf } \frac{1}{|\tau x'|^{N+n}} d x'\Big)
    \Big( \int\limits_{\R_+} e^{- C \tau x_n} d x_n\Big)\\
    &\qquad \leq C_N'  \gamma^{N+n} \tau^{-\frac{3+N}{2}} 
    \int\limits_{|x'| \geq \hf } \frac{1}{|x'|^{n}} d x'.
  \end{align*}
  
  Similarly, setting $z_- = \big( 1- \chi_0(|\tau^\hf x'|)\big)
  \check{\hat{u}}_-(x',x_n)$ for $x_n\leq 0$ we obtain
  $\norm{z_-}_{L^2(\R^{n-1}\times \R_-} \leq  C_\gamma \tau^{-N}$.
  The second result thus follows from Lemma~\ref{sub-lemma: estimate
    quasi-mode}.

  For the first result we write 
  \begin{align*}
    {\mathcal L}_{\tau} v_\pm = (2\pi)^{-(n-1)} \chi_0(|\tau^\hf x'|) 
    {\mathcal L}_{\tau} \check{\hat{u}}_\pm
    + (2\pi)^{-(n-1)} \big[{\mathcal L}_{\tau}, \chi_0(|\tau^\hf x'|)\big] \check{\hat{u}}_\pm
  \end{align*}
  The first term is estimated using Lemma~\ref{sub-lemma: estimate
    quasi-mode} as 
  $$
  (2\pi)^{-\frac{(n-1)}{2}} 
  \norm{{\mathcal L}_{\tau}\check{\hat{u}}_\pm }_{L^2(\R^{n-1}\times R_\pm)}
  = \norm{{\mathcal M}_{\tau} u_\pm }_{L^2(\R^{n-1}\times R_\pm)}.
  $$
  Observing that ${\mathcal L}_{\tau}$ is a {\em differential} operator
  The commutator is thus a first-order differential operator in $x'$
  with support in a region $|\tau^\hf x'| \geq C$, because of the
  behavior of $\chi_1$ near $0$. The coefficients of this operator
  depend on $\tau$ polynomially.  The zero-order terms can be estimated as
  we did for $z_+$ above with  an additional 
  $\tau$ factor.

  For the first-order term observe that we have 
  \begin{align*}
    \p_{x'_j} \check{\hat{u}}_+(x',\tau)  
    = \tau^{n} \int\limits_{\R^{n-1}} \eta_j e^{i \tau \big( x' \cdot \eta 
      -i x_n (\tilde{f}_+ (\eta) + \frac{\beta x_n }{2}) \big)} \tilde{\psi}(\eta)
    \chi_{0}\Big(\frac{\beta \gamma x_{n}}{\val{\tilde{f}_{+}(\eta)}}\Big)
    d \eta. 
  \end{align*}
  We thus obtain similar estimates as above
  with an additional $\tau^{\frac{3}{2}}$ factor. This concludes the proof.
\hfill $\blacksquare$ \endproof

\bibliography{refjlr-nl}
\nocite{*}
\bibliographystyle{amsplain}

%%%%%%%%%%%%%%
\end{document}